\documentclass[11pt, a4paper, reqno]{amsart}
\usepackage{mathrsfs, mathpazo}
\usepackage{amsfonts, amsbsy}
\usepackage{dsfont}
\usepackage{amssymb, mathabx, mathtools}
\usepackage{xcolor}

\usepackage[colorlinks,citecolor=blue,urlcolor=blue]{hyperref}

\usepackage{amsmath}
\usepackage{framed}
\usepackage{mathrsfs,mathtools}

\newcommand\adda[1]{{\color{blue}{#1}}}

\usepackage{colortbl}
\hypersetup{urlcolor=midnightblue, citecolor=red}

\theoremstyle{remark}

\newcommand{\bn}{\mbox{\boldmath{$n$}}}
\newcommand{\bv}{\mbox{\boldmath{$v$}}}
\newcommand{\bu}{\mbox{\boldmath{$u$}}}
\newcommand{\bh}{\mbox{\boldmath{$h$}}}

\newcommand{\by}{\mbox{\boldmath{$y$}}}

\newcommand{\bB}{\mbox{\boldmath{$B$}}}

\newcommand{\bM}{\mbox{\boldmath{$M$}}}

\newcommand{\bG}{\mbox{\boldmath{$G$}}}
\newcommand{\bU}{\mbox{\boldmath{$U$}}}
\newcommand{\bW}{\mbox{\boldmath{$W$}}}
\newcommand{\bH}{\mbox{\boldmath{$H$}}}

\newcommand{\bw}{\mbox{\boldmath{$\mathrm{w}$}}}
\def\d{\mathrm{d}}

\DeclareMathOperator{\diver}{div}

\DeclareMathOperator{\curl}{curl}

\newcommand\dela[1]{}

\selectfont

\theoremstyle{plain}
\newtheorem{theorem}{\textbf{Theorem}}[section]
\newtheorem{lemma}[theorem]{\textbf{Lemma}}

\newtheorem{proposition}[theorem]{\textbf{Proposition}}
\newtheorem{cor}[theorem]{\textbf{Corollary}}
\theoremstyle{definition}

\newtheorem{remark}[theorem]{\textbf{Remark}}

\newtheorem{definition}[theorem]{\textbf{Definition}}

\newtheorem{assumption}[theorem]{\textbf{Assumption}}
\numberwithin{equation}{section}
\numberwithin{figure}{section}
\begin{document}
\title[Weak solutions to the 3D stochastic ferrofluids]{Existence of weak solution of 3D ferrohydrodynamic equations with transport noise: Bloch-Torrey regularisation}
\author{A. Ndongmo Ngana$^{1}$ and P. A. Razafimandimby$^{2}$}
\dedicatory{\vspace{-10pt}\normalsize{$^{1}$ University of York, Heslington, York YO10 5DD, UK \\
$^{2}$ School Of Mathematical Sciences, Dublin City University, Collins Avenue, Dublin D09W6Y4, Ireland}}
\thanks{
\textbf{A. Ndongmo Ngana is fully supported by the EPSRC-UKRI Research Agency through the Marie Sk\l{}odowska-Curie Action.}
}

	
\begin{abstract}
In this article, we consider a stochastic ferrohydrodynamic system  which describes the Bloch-Torrey regularization of the motion of an electrically conducting ferrofluids driven by transport noise filling a 3D bounded domain with a smooth boundary. We prove the global existence of a probabilistic weak solution of the stochastic system by making use of the combination of Galerkin method and compactness method. The system under study is basically a coupling of the Navier-Stokes equations with internal rotation, the Maxwell equations and the ferromagnetization equations. Thus, our result can be seen as a modest generalization of the existing results on the global existence of  weak solutions of stochastic Magnetohydrodynamic (MHD),  Navier-Stokes and Bloch type equations on 3D bounded domain.  
\end{abstract}
	
\maketitle
	
\section{Introduction}
Magnetic fluids (also called ferrofluids) are suspensions of ferromagnetic nanoparticles in a liquid carrier (water, oil, etc.). The suspended particles conform to Brownian motion, which makes the ferrofluids dynamics stochastic. Ferrofluids can be controlled by an external field, especially the magnetic field, which gives rise to a wealth of applications in industry, engineering, technology and medical sciences: cooling and resonance damping for loudspeaker coils, rotating shaft seals in vacuum chambers used in the semiconductor industry, drugs or radioisotopes targeted by magnetic guidance, magnetic resonance imaging contrast enhancement, etc.. See \cite{Pankhurst+Connolly+Jones+Dobson_2003,Zahn_2001} and references therein for the many applications of ferrofluids. There are two widely accepted  models that describe the dynamics of ferrofluids: the Shliomis model \cite{Shliomis_2002} and the Rosensweig model \cite{Rosensweig,Rosensweig_2002}. These models are both nonlinear coupling system of partial differential equations
describing the dynamic of a magnetic fluid with internal rotations. 	

Since the movement of ferrofluids particles follows a thermal Brownian motion, the presence of intrinsic noise and irregular movement in their microscopic individual dynamics seem to be ubiquitous. Despite  stochasticity of the magnetic fluids dynamics, previous mathematical analyses only considered deterministic models, see, for instance, \cite{Kamel+Hamdache+Murat_2008,Kamel,Nochetto_2019,Xie_2019} and references therein.  In order to establish realistic models,  the macroscopic equations of the collective dynamics, which are obtained by  the limit in the microscopic individual dynamics, should contain random perturbations due to noisy external forces and the  randomness of the ambient environment. This motivated us to consider ferrofluids with stochastic perturbation in the present paper and to initiate research on such problems. We are mainly  interested in the electrically conductive Rosensweig equations (ECREs for short) driven by transport noise type and filling a bounded domain $\mathcal{O}\subset \mathbb{R}^3$ with a boundary $\Gamma:=\partial \mathcal{O}$ of class $C^\infty$.  More precisely, given a time horizon $T>0$, 
we consider the following stochastic system 
\begin{subequations}\label{Eq4.1}
	\begin{align}
		& \partial_t \bu + (\bu \cdot \nabla) \bu + \nabla p - \nu \Delta \bu - \mu_0 (\bM \cdot \nabla) \bH \notag
		\\ 
		&\quad = \mu_0 \curl \bH \times \bH - \alpha \curl(\curl \bu - 2 \bw) + \sum_{k=1}^{\infty} (b_k \cdot \nabla) \bu \, \partial_t \beta_{k}^1 \qquad \text{in} ~ (0,T) \times \mathcal{O}, \label{Eq4.1a} 
		\\
		& \partial_t \bw + (\bu \cdot \nabla ) \bw - (\lambda_1 + \lambda_2) \nabla \diver \bw - \lambda_1 \Delta \bw \notag
		\\
		& \quad = 2 \alpha (\curl\bu - 2 \bw) + \mu_0 \bM \times \bH + \sum_{k=1}^{\infty} (f_k \cdot \nabla) \bw \, \partial_t \beta_k^2 \hspace{2cm} \text{in} ~ (0,T) \times \mathcal{O}, \label{Eq4.2a}
		\\
		& \partial_t \bM + (\bu \cdot \nabla ) \bM + \lambda \curl(\curl\bM)-\lambda \nabla \diver \bM \notag \\
		&\quad = \bw \times \bM - \frac{1}{\tau} (\bM - \chi_0 \bH) + \sum_{k=1}^{\infty} (\sigma_k \cdot \nabla) \bM \, \partial_t \beta_k^3 \hspace{2.7cm} \text{in} ~ (0,T) \times \mathcal{O}, \label{Eq4.3a}
		\\
		& \partial_t \bB + \frac{1}{\sigma} \curl(\curl\bH) = \curl(\bu \times \bB) + \sum_{k=1}^{N} (j_k \cdot \nabla) \bH \, \partial_t \beta_{k}^4 \hspace{1.7cm} \text{in} ~ (0,T) \times \mathcal{O}, \label{Eq4.4a}
		\\
		& \bB = \mu_0 (\bM +\bH),\quad \diver \bB= 0,  \hspace{5.8cm} \text{in} ~ (0,T) \times \mathcal{O}, \label{Eq4.5a} 
		\\
		& \diver \bu=0 \hspace{9.2cm} \text{in} ~ (0,T) \times \mathcal{O}. \label{Eq4.6a} 
	\end{align}
\end{subequations}
Roughly speaking, this system, which we call Bloch-Torrey electrically conductive Rosensweig equations (BT-ECREs for short),  is a simplified version of the  Bloch-Torrey regularization of a ferrofluid system proposed in \cite{Kamel}. In \eqref{Eq4.1a}-\eqref{Eq4.6a}, the unknown functions are the velocity $\bu$ of the fluid, its pressure $p$, the fluid internal rotation $\bw$, the magnetization $\bM$ and the magnetic field $\bH$. Here, the viscosity coefficients $\nu, \, \lambda_1$, and $\lambda_2$, the relaxation time $\tau$, the magnetic susceptibility coefficient $\chi_0$, the (small) diffusion coefficient $\lambda$, and the electric conductivity $\sigma$ are all positive and assumed to be constant. The constant   $\mu_0$ is a fundamental constant called the magnetic permeability of the vacuum.  The forcing term $\mu_0(\bM\cdot\nabla)\bH$ in the linear momentum equation is the so-called Kelvin force. The term $\mu_0\bM$ represents the vector moment per unit volume. The equation $\diver \bB=0$ (cf. \eqref{Eq4.5a}) is the Maxwell equation for the magnetic induction $\bB=\mu_0(\bM+\bH)$ in $\mathcal{O}$ and $\bB=\mu_0\bH$ outside $\mathcal{O}$ where the magnetization $\bM$ vanishes. The processes $\beta_k^i, \,i = 1,\ldots,4$ and $k\in \mathbb{N}$ are iid standard Brownian motions.

The noise coefficients $b_k,\, f_k,\, \sigma_k:\mathcal{O} \to \mathbb{R}^3$, $ k \in \mathbb{N}$, $j_k: \mathcal{O} \to \mathbb{R}^3$, $k \in \{1,\ldots, N\}$ are sufficiently regular functions. In particular, we assume that for each $ k \in \mathbb{N}$
\begin{equation*}
	\diver \sigma_k = 0.
\end{equation*}
We endow the problem \eqref{Eq4.1a}-\eqref{Eq4.6a} with the following boundary conditions
\begin{equation}\label{eq4.2} 
	\begin{aligned}
		\bu &=0, \quad \bw= 0, \quad \diver \bM=0 \quad &\mbox{on} \quad (0,T)\times \Gamma, \\
		\bM \times \bn&=0, \quad  \bH \times \bn=0, \quad &\mbox{on} \quad (0,T)\times \Gamma,
	\end{aligned}
\end{equation}
where $\Gamma$ is the boundary of $\mathcal{O}$ and $\bn$ is its outward normal.

\noindent 
The initial condition is given by
\begin{equation}\label{eqt4.3} 
	\begin{aligned}
		(\bu, \bw, \bM, \bH)(t=0) &= (\bu_0,\bw_0,\bM_0,\bH_0) \quad &\text{in} \quad \mathcal{O}, \\
		\diver \bu_0&= \diver(\bM_0+\bH_0)=0, \quad &\text{in} \quad \mathcal{O},
	\end{aligned}
\end{equation}
where $\bu_0,\,\bw_0,\,\bM_0$ and $\bH_0$ are given non-random  initial data.

\noindent


While the deterministic Rosensweig and Shliomis models and their variants have been extensively studied (see \cite{Kamel+Hamdache+Murat_2008,Kamel,Aristide+Paul,Nochetto-2016,Nochetto_2019,Scrobogna-2019,Scrobogne-2021,Xie_2019}), we are not aware of any article that studied a stochastic system of ferrofluids. There is various excellent work on stochastic Navier-Stokes equations subsequent to the research conducted by Bensoussan and Temam in \cite{Bensoussan_1995}. We refer readers, for example, to the papers \cite{Flandoli+Gatarek_1995,Motyl1,Brzezniak+Capinski+Flandoli_1992,Brzezniak+Hausenblas+Zhu_2013} and \cite{Brzezniak+Peszat_1999} just to name a few. We also point out that there is excellent stochastic work on the coupled model of the Maxwell and Navier-Stokes equations (see, e.g., \cite{Chueshov+Millet_2010}, \cite{Barbu+DaPrato}, \cite{Sango}, \cite{Sango+Tesfalem}, \cite{Sundar}  and \cite{Yamazaki}), commonly known as the magnetohydrodynamic (MHD) equations, which describe the motion of electrically conductive fluids and have many applications in astrophysics, geophysics and plasma physics. However, our stochastic models do no fall in these frameworks. In fact, in contrast to ferrofluids, the MHD equations are  models for non-magnetizable but electrically conductive fluids. Furthermore, the dominant body force in MHD is the Lorentz force $\mu_0 \curl \bH \times \bH$, whereas for ferrofluids, the Kelvin force $\mu_0(\bM\cdot\nabla)\bH$ is the most important one, leading to different kinds of nonlinearities and difficulties. 
  
In this paper, we prove the existence of a probabilistic weak solution to the system \eqref{Eq4.1} by means of a combination of the Galerkin and compactness methods. Not only do we partially extend the known results in the deterministic case, but we also modestly generalize the solution existence results of 3D MHD and Navier-Stokes equations and Bloch-type equations with random perturbations. In fact, notice that in the absence of fluid carrier and magnetic field, Eq. \eqref{Eq4.3a} reduces to a deterministic Bloch-Torrey type  equation, which is a modification of the Bloch-type equation with diffusion  proposed by Torrey in \cite{Torrey2}, see also \cite{Gaspin}. The system reduces to the Magnetohydrodynamic (MHD) equations when $\bM\equiv 0$ and $\bw\equiv 0$. 
  
In a subsequent paper, we will study the case where $\lambda=0$, which appears to be a challenging problem due to the fact that \eqref{Eq4.1} will be reduced to a system of ferrofluid models consisting of parabolic and hyperbolic equations, and that the Kelvin force  $\mu_0(\bM\cdot\nabla)\bH$  is a priori unbounded. We can, however, conjecture that for the passage to the  limit $\lambda \to 0$ to be successful, we need that the electric conductivity $\sigma$ is not too small, which agrees with the deterministic case, see \cite{Aristide+Paul} and also Remark \ref{Rem:Limit-gamma}. This conjecture means that natural ferrofluids, which are naturally dielectric, cannot be obtained from the electric conductivity limit of artificially developed electric conducting models. 

\noindent
The present paper is split into three main parts. The first one, consisting of Section \ref{sect2} concerns the abstract framework. We also state and prove some auxilliary results which are essential to our analysis. In Section \ref{sect3}, we introduce our stochastic framework, formulate the problem, outline the general assumptions, and state our main result. The last part, consisting of Section \ref{sect5} contains the proof of the main theorem on the existence of a martingale solutions. 

\section{Deterministic Setting of BT-ECREs}\label{sect2}

\subsection{Setting of functional spaces}	
Throughout the paper, we will use the notation $Q_t:=(0,t) \times \mathcal{O}$ for every $t \in(0,T]$, and set $Q:= Q_T$ and $\Sigma:= (0,T) \times \Gamma$, while $a_2 \wedge a_3 = \min(a_2,a_3)$ for any real numbers $a_2$ and $a_3$. Next, 
for any (real) Banach space $X$, its (topological) dual is denoted by $X'$ and the duality pairing between $X'$ and $X$ by $\langle \cdot, \cdot \rangle_{X',X}$. If $E_1$ and $X_1$ are separable Hilbert spaces, then by $L_2(E_1,X_1)$ we will denote the Hilbert space of all Hilbert-Schmidt operators from $E_1$ to $X_1$ endowed with the canonical norm $\|\cdot\|_{L_2(E_1,X_1)}$. For any $1 \leq p \leq \infty$ and $s \in \mathbb{R}$, we denote by $L^p(\mathcal{O})$ and $W^{s,p}(\mathcal{O})$ the usual Lebesgue and Sobolev spaces of scalar functions, respectively. If $p=2$, we simply write $W^{s,2}(\mathcal{O}) = H^s(\mathcal{O}) $. We denote by $H_0^1(\mathcal{O})$ the closure of $\mathcal{C}_0^\infty(\mathcal{O})$ in $H^1(\mathcal{O})$. To shorten the notation, we use the notations $\mathbb{L}^p$, $\mathbb{W}^{s,p}$, $\mathbb{H}^s$, and $\mathbb{H}_0^1$, to denote the spaces $[L^p(\mathcal{O})]^3$, $[W^{s,p}(\mathcal{O})]^3$, $[H^s(\mathcal{O})]^3$, and $[H_0^1(\mathcal{O})]^3$, respectively. \newline 
Following the notations used in \cite{Temam1} for the Navier-Stokes equations and in \cite{Gerbeau} for the Magnetohydrodynamics of liquid metals, we introduce the following spaces
\begin{equation*}
	\begin{aligned}
		\mathcal{V} &= \{\bv \in [\mathcal{C}_0^\infty(\mathcal{O})]^3 ~ \text{such that} ~ \diver \bv= 0\}, 
		\\
		H &= \text{the closure of} ~ \mathcal{V} ~ \text{in} ~ \mathbb{L}^2, 
		\\
		V &= \text{the closure of} ~ \mathcal{V} ~ \text{in} ~ \mathbb{H}_0^1, 
		\\ 
		H(\curl,\mathcal{O}) &= \{\bv \in \mathbb{L}^2;~\curl\bv \in \mathbb{L}^2\},
		\\
		H(\curl^0,\mathcal{O}) &= \{\bv \in H(\curl,\mathcal{O});~\curl \bv = 0\},
		\\
		H(\diver,\mathcal{O}) &= \{\bv \in \mathbb{L}^2;~\diver \bv \in L^2(\mathcal{O})\},
		\\
		H(\diver^0,\mathcal{O}) &= \{\bv \in H(\diver,\mathcal{O});~\diver \bv = 0\},
		\\
		H_0(\diver,\mathcal{O}) &= \{\bv \in H(\diver,\mathcal{O});~ \bv \cdot \bn = 0 ~ \text{on} ~ \partial \mathcal{O}\},
		\\
		H_{\bn} &=\{\bv \in \mathbb{L}^2; ~ \bv \times \bn=0  \text{ on } \partial \mathcal{O}\}.
	\end{aligned}
\end{equation*}
We also introduce the spaces:
\begin{equation*}
	\begin{aligned}
		V_1 &= H_{\bn}\cap \mathbb{H}^1,
		\\
		V_2 & = V_1 \cap H(\diver^0,\mathcal{O}),
		\\ 
		E(\mathcal{O}) &= \{\bv: \bv= \nabla h,\; h \in L^2_{\operatorname{loc}}(\mathcal{O}), \hspace{0,1cm} \nabla h \in \mathbb{L}^2\}, 
		\\
		E_1(\mathcal{O}) &=\{\bv : \bv = \nabla h \in E(\mathcal{O}),\hspace{0,1cm} \Delta h \in L^2(\mathcal{O}) \}.
	\end{aligned}
\end{equation*}
The space $H$ can also be characterized in the following way (see \cite[Theorem I.1.4]{Temam})
\begin{equation*}
	H=\{\bu \in H(\diver,\mathcal{O}) :~ \diver\bu = 0 ~ \text{in} ~ \mathcal{O},~ \bu \cdot \bn\lvert_{\Gamma} = 0 ~ \text{on} ~ \Gamma\},
\end{equation*}
and also by (cf. \cite[Theorem 1.4]{Simader})
\begin{equation*}
	H = \{\bu \in \mathbb{L}^2: \langle \nabla h,\bu \rangle = 0 ~ \forall \hspace{0.1cm} \nabla h \in (E_2(\mathcal{O}))'\},
\end{equation*}
where $\langle \cdot, \cdot \rangle$ denotes the dual pairing between $E_2(\mathcal{O})$ and its dual $(E_2(\mathcal{O}))'$. \newline
The space $V$ has the following characterization (see \cite[Theorem 1.6]{Temam1}),
\begin{equation*}
	V = \{\bu \in \mathbb{H}^1:~ \diver\bu = 0 ~ \text{in} ~ \mathcal{O},~ \bu \lvert_{\Gamma} = 0 ~ \text{on} ~ \Gamma\}.
\end{equation*}
We denote by $(\cdot, \cdot)$ and $|\cdot|$ the inner product and the norm induced by the inner product and the norm in $\mathbb{L}^2$ on $H$, respectively. We endow $H$ with the scalar product and norm of $\mathbb{L}^2$. As usual, we equip the space $V$ with the gradient-scalar product $ (\nabla \bu, \nabla \bv), \, \bu, \bv\in V,$ and the gradient-norm $|\nabla \cdot|$, which is equivavent to the $\mathbb{H}_0^1$-norm. The spaces $H(\diver,\mathcal{O})$ and $H(\curl,\mathcal{O})$ are Hilbert spaces when equipped with  the scalar products defined by
\begin{align*}
	(\bu,\bv)_{H(\diver,\mathcal{O})}= (\bu,\bv) + (\diver\bu, \diver\bv) \quad &\forall \bu, \bv \in H(\diver,\mathcal{O}),
	\\
	(\bu_1,\bv_1)_{H(\curl,\mathcal{O})}= (\bu_1,\bv_1) + (\curl\bu_1, \curl\bv_1) \quad &\forall \bu_1, \bv_1 \in H(\curl,\mathcal{O}),
\end{align*}
respectively. The space $E(\mathcal{O})$ is equipped with the norm
\begin{equation*}
	\|\bv\|_{E(\mathcal{O})}= |\nabla h| \;\; \text{ for } \bv= \nabla h \in E(\mathcal{O}),
\end{equation*}
while the space $E_1(\mathcal{O})$ is equipped with the following norm 
\begin{equation*}
	\|\bv\|_{E_1(\mathcal{O})}:= (|\nabla h|^2 + |\Delta h|^2)^\frac{1}{2} \quad \text{for} \quad \bv=\nabla h\in E_1(\mathcal{O}).
\end{equation*}
We recall that, see \cite[Corollary 3.1]{Amrouche_2013}, there exists a positive constant $C_0$ such that 
\begin{equation}\label{eq2.1}
	\lvert \nabla \bM \lvert^2 \leq C_0(\lvert \bM \rvert^2 + \lvert \curl\bM \rvert^2 + \lvert \diver\bM \rvert^2), \;\; \bM \in V_1,
\end{equation}
which provides the spaces $V_1$ and $V_2$ with the equivalent norm associated to the inner product defined by
\begin{equation}\label{Eq2.2}
	[\cdot;\cdot]=(\cdot,\cdot) + (\curl\cdot,\curl\cdot) + (\diver\cdot, \diver\cdot).
\end{equation}
Now, we define the Hilbert space $\mathbb{H}$ by 
\begin{equation*}
	\mathbb{H}:= H \times \mathbb{L}^2 \times H_{\bn} \times  H_{\bn},
\end{equation*}
endowed with the scalar product whose associated norm is given by
\begin{equation*}
	\|(\bu,\bw,\bM,\bH)\|_{\mathbb{H}} = (|\bu|^2 + |\bw|^2 + |\bM|^2 + |\bH|^2)^{1/2}, \quad \forall (\bu,\bw,\bM,\bH) \in \mathbb{H}.
\end{equation*}
Let $I$ be a subset of $[0,\infty)$. The set of continuous functions $f : I \to X$ is denoted by $\mathcal{C}(I;X)$. The space $\mathcal{C}_w(I;X)$ consists of all functions $f \in L^\infty(I;X)$ such that the map $t \in I \mapsto \langle \varphi,f(t) \rangle_{X',X}$ is continuous, for all $\varphi \in X'$. We denote by $L_{w}^{2}(0,T; X)$ the space $L^2(0,T;X)$ equipped with the weak topology.
\subsection{Setting of bilinear operators and their properties}

Throughout this paper, we denote by $\bu_i$ the $i$-th entry of any vector-valued $\bu$ and we set $\partial_i g= \frac{\partial g}{\partial x_i}$. \newline 
Let $p,q,r \in [1,\infty]$ such that $\frac 1 p + \frac 1 q + \frac 1 r \leq 1$. Now we introduce the trilinear form $b(\cdot,\cdot,\cdot)$ defined by 
\begin{equation}
	b(\boldsymbol{\phi},\boldsymbol{\psi},\bv) 
	= \sum_{i,j = 1}^{3} \int_{\mathcal{O}} \boldsymbol{\phi}_i(x) \partial_i \boldsymbol{\psi}_j(x)  \bv_j(x) \,\d x, \quad \boldsymbol{\phi} \in \mathbb{L}^p, \;\; \boldsymbol{\psi} \in \mathbb{W}^{1,q}, \;\; \text{and} \quad \bv \in \mathbb{L}^r.
\end{equation}
$b(\cdot,\cdot,\cdot)$ is the well-known trilinear form used in the mathematical analysis of the Navier-Stokes equations (see, for instance, \cite{Temam2}). \newline 
Let $X\in \left\{\mathbb{H}_0^1, V\right\}$. By using the density of $C_c^\infty(\mathcal{O})$ and the Gagliardo-Nirenberg inequality, we can prove as in \cite[Lemmata II.1.3 \& II.1.6 ]{Temam} that there exists a positive constant $C$ depending only on $\mathcal{O}$ such that
\begin{equation*}
	|b(\bu,\bv,\bw)|\leq C \|\bu\|_{\mathbb{L}^4} |\nabla\bw| \|\bv\|_{\mathbb{L}^4} \leq  C |\bu|^{\frac14} |\nabla \bu|^{\frac34} |\bv |^{\frac14} |\nabla \bv|^{\frac34} |\nabla \bw|, \quad \forall \bu \in V, \, \bv ,\bw \in X .
\end{equation*}
Using this inequality we can prove the following facts. 
\begin{enumerate}
\item  There exists a continuous bilinear map $B_0 : V \times V \to V'$ such that
	\begin{equation*}
		\langle B_0(\bu,\bv), \bv_1 \rangle_{V',V}= b(\bu, \bv, \bv_1), \quad \forall \bu, \, \bv, \, \bv_1 \in V.
	\end{equation*} 
Furthermore, the bilinear form $B_0(\cdot,\cdot)$ enjoys the following properties:
	\begin{itemize}
\item[] for any $\bu, \, \bv, \, \bv_1 \in V$, we have 
		\begin{equation}\label{eq2.6}
			\langle B_0(\bu,\bv), \bv_1 \rangle_{V',V}
			= - b(\bu, \bv_1,\bv) \quad \text{and} \quad \left \langle B_0(\bu,\bv), \bv \right \rangle_{V',V}= 0.
		\end{equation}
See \cite[Section II.1.2]{Temam2} for more details on these properties.
	\end{itemize}
\item There exists a continuous bilinear map $B_1: V \times \mathbb{H}_0^1 \to \mathbb{H}^{-1}$ such that
	\begin{equation*}	
		\langle B_1(\bu,\bv), \bw \rangle_{\mathbb{H}^{-1},\mathbb{H}_0^1}
		= b(\bu, \bv, \bw), \quad \forall \bu \in V, \,\bw, \,\bv \in \mathbb{H}_0^1.
	\end{equation*}
\item There is also a continuous bilinear map $B_2: V \times V_1 \to V'_1$ such that
	\begin{equation*}
		\langle B_2(u,\bM), \bv \rangle_{V'_1, V_1}
		= b(\bu, \bM, \bv), \quad \forall \bu \in V, \, \bM,\,\bv \in V_1.
	\end{equation*}
Note that there exists a constant $C_1>0$ such that for  $X\in \left\{ V', \mathbb{H}^{-1}, V'_1\right\}$ and $k\in \{0,1,2\}$
	\begin{equation}\label{eq2.6a}
		\lVert B_k(\bu, \bv)\rVert_{X} \leq C_1  |\bu|^{\frac14}  |\nabla \bu|^{\frac34} |\bv|^{\frac14} |\nabla \bv|^{\frac34}, \, \bu \in V, \, \bv \in X.
	\end{equation}
\end{enumerate}
Hereafter, we introduce the following map
\begin{equation*}
	M_1(\bM,\bH,\bv):= \sum_{i,j=1}^3 \int_{\mathcal{O}} \bM_i(x) (\partial_i \bH_j(x)) \bv_j(x) \, \d x, \quad \forall \bM,\,\bH \in V_1,\; \bv \in V,
\end{equation*}
which is well-defined since we are working on a space whose dimension is smaller or equal to three. The properties of the map $M_1: V_1 \times V_1 \times V \to \mathbb{R}$ are given in the following lemma.
\begin{lemma}\label{lem1}
Let $\bM, \bH \in V_1$ and $\bv \in V$. We have
	\begin{equation}\label{eq2.2}
		M_1(\bM,\bH,\bv)
		=-\int_{\mathcal{O}}[(\bM(x) + \bH(x)) \cdot \nabla] \bv(x) \cdot \bH(x) \d x - \int_{\mathcal{O}} \curl\bH(x) \cdot (\bH(x) \times \bv(x)) \d x,
	\end{equation}
and
	\begin{equation}\label{eqt2.2}
		M_1(\bM,\bH,\bv) 
		= \int_{\mathcal{O}}\curl\bH(x) \cdot(\bM(x) \times \bv(x)) \, \d x - \int_{\mathcal{O}}(\bv(x) \cdot \nabla)\bM(x) \cdot \bH(x) \, \d x.
	\end{equation}
Furthermore, there exists a positive constant $C_2>0$  such that
\begin{equation}\label{eq2.3}
\vert M_1(\bM,\bH,\bv)\vert 
\leq C_2 |\bM|^\frac{1}{4} \|\bM\|_{V_1}^\frac{3}{4} \|\bH\|_{V_1} \|\bv\|_{\mathbb{L}^4},
\end{equation}		
	         \begin{equation}\label{eq2.4}
	         	\vert M_1(\bM,\bH,\bv)\vert 
	         	\leq C_2\|\bM + \bH\|_{\mathbb{L}^4} \|\bv\|_V \|\bH\|_{\mathbb{L}^4} + C_2 |\curl\bH| \|\bH\|_{\mathbb{L}^4} \|\bv\|_{\mathbb{L}^4}.
	         \end{equation}
If $\bv \in V \cap \mathbb{H}^2$, then there exists a positive constant $C_3>0$ such that
	\begin{equation}\label{eq2.5}
		\vert M_1(\bM,\bH,\bv)\vert
		\leq C_3 (|\bM + \bH| \|\bH\|_{\mathbb{L}^4} + |\curl\bH| |\bH|)  \|\bv\|_{\mathbb{H}^2}.
	\end{equation}
By the above results, there exists a continuous bilinear map $M_0 : V_1 \times V_1\to V'$ such that
	\begin{equation*}
		\langle M_0(\bM,\bH), \bv \rangle_{V', V}
		= M_1(\bM, \bH,\bv) \dela{b(\bM, \bH,\bv)}, \quad \forall \bM, \,\bH \in V_1,\, \bv \in V.  
	\end{equation*}
\end{lemma}
\begin{proof}
    Please see Appendix \ref{Appendix-1}.
\end{proof}
Before we procedd further we make the following observations. 
\begin{remark}
From \eqref{eq2.2} and \eqref{Eq4.5a}, we easily see that $M_1(\bM,\bH,\bH)= 0$, for all $\bM, \bH \in V_1$.
\end{remark}
Now we define a trilinear map $M_2: V_2 \times V\times \mathbb{H}^1 \to \mathbb{R}$ by 
$$ M_2(\bu,\bB, \psi)
:= \int_{\mathcal{O}} [\curl (\bu \times \bB)] \cdot \psi \, \d x,\; \forall \bB \in V_2,\, \bu \in V \text{ and } \psi \in \mathbb{H}^1.$$ 
We collect some properties of the trilinear map $M_2$ in the following lemma whose proof is given in Appendix \ref{Appendix-1}.
\begin{lemma}\label{lem1B}
There exists a bilinear operator $\tilde{M}_2: V \times V_2 \to V'_2$ such that
\begin{itemize}
\item for any $\bu \in V,~ \bB, \psi \in V_2$ or any $\psi \in V_1$, we have
		\begin{equation}\label{eq2.9B}
			\langle \tilde{M}_2(\bu,\bB), \psi \rangle_{V_2',V_2} = M_2(\bu,\bB, \psi)= - b(\bB,\psi,\bu) - b(\bu,\psi,\bB).
		\end{equation}
\item  There exists a positive constant $\tilde{C_2}=\tilde{C_2}(\mathcal{O})>0$ such that
	\begin{equation}\label{eq2.9}
	 \|\tilde{M}_2(\bu,\bB)\|_{V'_2} \leq \tilde{C_2} |\bB|^\frac{1}{4} \|\bB\|_{V_2}^\frac{3}{4} |\bu|^\frac{1}{4} |\nabla \bu|^\frac{3}{4}, \quad \forall \bu \in V, ~ \bB \in V_2. 	 
	\end{equation}   
\end{itemize}
\end{lemma}

In the next few lemmas, we introduce  several bilinear and continuous mappings that are also necessary for the analysis of  the stochastic BT-ECREs.
\begin{lemma}\label{lem:Bilinear-R0}
There exists a bilinear map $R_0(\cdot,\cdot):V \times H(\curl,\mathcal{O}) \to V'$ such that
	\begin{equation*}
		\langle R_0(\bu, \bw), \tilde{\bv} \rangle_{V',V}= \int_{\mathcal{O}} [\curl(\curl \bu - 2 \bw)] \cdot \tilde{\bv} \, \d x, \quad \bu \in V \cap \mathbb{H}^2, \, \tilde{\bv} \in V, \, \bw \in \mathbb{H}^1,
	\end{equation*}
and that
	\begin{equation}\label{eq2.15}
		\|R_0(\bu, \bw)\|_{V'} 
		\leq |\nabla \bu| + 2 |\curl \bw|, \quad \bu \in V \cap \mathbb{H}^2, ~ \bw \in H(\curl,\mathcal{O}).
	\end{equation}
	
\end{lemma}
\begin{proof}
We prove this lemma in Appendix \ref{Appendix-1}.
\end{proof}
\begin{lemma}\label{Lem:Bilinear-R1}
\begin{itemize} 
\item There exists a bilinear map $R_1: H(\curl,\mathcal{O}) \times \mathbb{L}^4 \to V'$ such that 
		\begin{align*}
			\langle R_1(\bH, \bh), \bv \rangle_{V',V}
			&= \int_{\mathcal{O}} (\curl \bH(x) \times \bh(x)) \cdot \bv(x) \, \d x, \quad & \bH\in H(\curl, \mathcal{O}), ~ \bh \in \mathbb{L}^4, \, \bv \in V,  \\
			|\langle R_1(\bH, \bh), \bv \rangle_{V',V}| 
			&\leq C(\mathcal{O}) |\curl\bH| \|\bh\|_{\mathbb{L}^4} |\nabla \bv|, \quad & \bH\in H(\curl, \mathcal{O}),~ \bh \in \mathbb{L}^4, \, \bv \in V.
		\end{align*}   
\item The map $ R_2: V \times \mathbb{H}^1 \to \mathbb{L}^{2}$  defined by  
		\begin{equation*}
			( R_2(\bu,\bw), \phi )= \int_{\mathcal{O}} (\curl \bu(x) - 2 \bw(x)) \cdot \phi(x) \, \d x, ~ \bu \in V, \, \bw \in \mathbb{L}^2, \, \phi \in \mathbb{L}^2,  
		\end{equation*}
satisfies \begin{equation}\label{Eq:Ineq-R2}
			|R_2(\bu,\bw)| 
			\leq C (|\curl\bu| + |\bw|), \quad \forall \bu \in V, \, \bw \in \mathbb{L}^2.
		\end{equation}
\item Let $R_3:  \mathbb{L}^4 \times \mathbb{L}^4 \to \mathbb{L}^{2}$ be  defined by 
		\begin{equation*}
			(R_3(\bM,\bH), \boldsymbol{\psi})= \int_{\mathcal{O}} (\bM \times \bH) \cdot \boldsymbol{\psi} \,\d x, \quad \bM, \,\bH\in \mathbb{L}^4, \,\boldsymbol{\psi} \in \mathbb{L}^2.
		\end{equation*}
Then, it verifies 
		\begin{equation}\label{Eq:Ineq-R3}
			|R_3(\bM, \bH)| \le C \|\bM\|_{\mathbb{L}^4} \| \bH\|_{\mathbb{L}^4}, \quad \bM, \,\bH\in \mathbb{L}^4. 
		\end{equation}		
\end{itemize}
\end{lemma}

\begin{proof}
One can prove the existence and properties of the bilinear map $R_1$ by using similar arguments as those used in Lemma \ref{lem:Bilinear-R0}; hence, the proof is left to the reader.
The properties of the maps $R_2$ and $R_3$ are easy to prove,  so we also leave that to the reader. 
\end{proof}
\subsection{Some linear maps and their properties}	
Let $\mathcal{P}: \mathbb{L}^2 \to H$ be the Helmhotz-Leray projection from  $\mathbb{L}^2$ onto $H$. We denote by $A=-\mathcal{P} \Delta$ the Stokes operator with domain $D(A)= V \cap \mathbb{H}^2$. It is well-known (see, for instance, \cite[Chapter I, Section 2.6]{Temam2}) that there exists an orthonormal basis $\{\upsilon_j\}_{j=1}^\infty$ of $H$ consisting of the eigenfunctions of the Stokes operator $A$. It is also well-known that $D(A^{1/2})= V$ (see \cite[p. 33]{Foias}). \newline
Now, we define the bilinear map $a_1: \mathbb{H}_0^1 \times \mathbb{H}_0^1 \to \mathbb{R}$ by setting
	\begin{equation*}
		a_1(\bw,\bv):=\int_{\mathcal{O}} \nabla \bw: \nabla \bv \, \d x, \quad \bw, \, \bv \in \mathbb{H}_0^1.
\end{equation*} 
Through an application of Poincar\'e's inequality, we can easily prove that this bilinear map is $\mathbb{H}_0^1$-continuous and $\mathbb{H}_0^1$-coercive. Hence, by applying the Riesz lemma and Lax-Milgram theorem, there exists a unique isomorphism $\mathcal{A}_1: \mathbb{H}_0^1 \to \mathbb{H}^{-1}$, such that
	\begin{equation*}
		a_1(\bw,\bv):= \langle \mathcal{A}_1 \bw, \bv \rangle_{\mathbb{H}^{-1}, \mathbb{H}_0^1}, \quad \bw, \, \bv \in \mathbb{H}_0^1.
	\end{equation*} 
Next, we define an unbounded linear operator $A_1$ in $\mathbb{L}^2$ as follows:
	\begin{eqnarray*}
		\begin{cases}
			D(A_1)&=\{\bw \in \mathbb{H}_0^1: A_1 \bw \in \mathbb{L}^2\}, \\
			A_1 \bw&= \mathcal{A}_1 \bw, \quad \bw \in D(A_1).
		\end{cases}
	\end{eqnarray*}
Under our assumption on $\mathcal{O}$, it is known (see, for instance, \cite[Section 2, p. 64]{Temam3}) that $A_1$ and $D(A_1)$ can be characterized by 
	\begin{eqnarray}
		\begin{cases}
			D(A_1) &:=\mathbb{H}^2 \cap \mathbb{H}_0^1, \\
			A_1 \bw &:= -\Delta \bw, \quad \bw \in D(A_1).
		\end{cases}
	\end{eqnarray}
Let $X\in \{\mathbb{H}_0^1, V_1 \}$ and  $R_5: H(\diver,\mathcal{O})\to X' $ be the linear map defined by 
$$ 
\langle R_5(\bw), \boldsymbol{\phi} \rangle_{X^\prime, X } 
= -\int_{\mathcal{O}} \diver \bw \diver \boldsymbol{\phi} \, \d x,~\forall \bw \in H(\diver,\mathcal{O}), ~ \boldsymbol{\phi} \in X .$$ 
The linear map $R_5$ is continuous and satisfies:  
    \begin{align}
    	\lVert R_5(\bw) \rVert_{X^\prime }\leq& |\diver \bw|, \hspace{2cm} \forall \bw \in H(\diver, \mathcal{O}),\label{Eq:Ineq-R5} \\
    	\langle R_5(\bw), \boldsymbol{\phi} \rangle_{X^\prime, X} 
    	=& \int_{\mathcal{O}} \nabla \diver \bw \cdot \boldsymbol{\phi} \,\d x, \quad \forall \bw \in H(\diver,\mathcal{O}) \cap \mathbb{H}^2, ~ \boldsymbol{\phi} \in X. \nonumber
    \end{align}  
Now let $Y\in \{V_1, V_2\}$ and  $\hat{a}_1$ be the bilinear form defined on $H(\curl, \mathcal{O}) \times V_2$ by 
\begin{equation*}
	\hat{a}_1(\bH,\boldsymbol{\psi}) = \int_{\mathcal{O}} \curl \bH(x) \cdot \curl \boldsymbol{\psi}(x) \,\d x, \, \bH \in H(\curl, \mathcal{O}), \, \boldsymbol{\psi} \in Y.
\end{equation*}
It is clear that $\hat{a}_1$ is a continuous bilinear map. Hence, there exists a continuous linear map  $R_6: H(\curl,\mathcal{O}) \to Y^\prime  $  satisfying 
$$ 
\langle R_6(\bH), \boldsymbol{\phi} \rangle_{Y^\prime, Y}
= \hat{a}_1(\bH,\boldsymbol{\phi}),~\forall \bH \in  H(\curl,\mathcal{O}), ~ \boldsymbol{\phi} \in Y.$$ 
Furthermore, by the Green formula, see \cite[Theorem 2.11]{Girault+Raviart} or \cite[Equation (IV.13)]{Boyer+Fabrie}, we have 
\begin{equation}\label{Eq:Ineq-R6}
	\begin{aligned}
		\langle R_6(\bH), \boldsymbol{\psi}\rangle_{Y^\prime, Y}
		=&	\hat{a}_1(\bH,\boldsymbol{\psi}) 
		\\
		=& \int_{\mathcal{O}} (\curl\curl \bH(x)) \cdot  \boldsymbol{\psi}(x) \, \d x, \quad \bH \in H(\mathrm{curl},\mathcal{O}) \cap \mathbb{H}^2, ~ \boldsymbol{\psi} \in Y, 
		\\
		\lVert R_6(\boldsymbol{\psi}) \rVert_{Y^\prime}
		\leq& |\curl \boldsymbol{\psi}|, \quad \hspace{3.3cm} \boldsymbol{\psi} \in  H(\curl,\mathcal{O}).
	\end{aligned}
\end{equation}  
\section{The Stochastic BT-ECREs and our main result} \label{sect3}
Since in this work we will study a stochastic model governing the motion of electrically conductive ferrofluids, in the following lines, we explicitly define the force injecting energy into the system we will analyze. \newline
Let $(\Omega, \mathcal{F}, \mathbb{F} = \{\mathcal{F}_t\}_{t\in[0,T]},\mathbb{P})$ be a filtered probability space satisfying the usual conditions (namely it is complete, right-continuous and $\mathcal{F}_0$ contains all null sets), with $T>0$ being a given final time. Let $U_1:=l^2(\mathbb{N})$ and $U_2:=\mathbb{R}^N$, where $l^2(\mathbb{N})$ denotes the space of all sequences $(e_n)_{n\in\mathbb{N}}\subset \mathbb{R}$ such that $\sum_{n=1}^\infty |e_n|^2<\infty$. It is a Hilbert space with the scalar product given by $(h,e)_{l^2}:=\sum_{n=1}^{\infty}h_n e_n$, where $h=(h_n)$ and $e=(e_n)$  belong to $l^2(\mathbb{N})$.  For convenience, we fix once and for all two complete standard orthonormal systems $\{e_k^1\}_{k \in \mathbb{N}}$ on $U_1$ and $\{e_k^2\}_{k \in \mathbb{N}}$ on $U_2$.
Throughout the paper, the processes $\beta_k^i$, $ k \in \mathbb{N}$, $ i = 1,\ldots,3$, and $\beta_k^4$, $ k \in \{1,\ldots,N\}$ are iid standard real Bronwnian motions, and $\mathbb{E}$ is the mathematical expectation with respect to the probability measure $\mathbb{P}$. For every $p, r \in [1,+\infty]$ and for every Banach space $X$, the symbols $L^r(\Omega;X)$ and $L^p(0,T;X)$ indicate the usual spaces of strongly measurable Bochner-integrable functions on $\Omega$ and $(0,T)$, respectively.\newline
Let $p,r\in[1,\infty]$. The space  $L^p(\Omega,\mathcal{F},\mathbb{P};L^r(0,T;X))$ consists of all random functions $v: \Omega \times [0,T] \times \mathcal{O}\to L^r(0,T;X)$ such that $v$ is measurable w.r.t $(\omega,t)$ and for all $t$, $v$ is measurable w.r.t $\mathcal{F}_t$. We furthermore  endow this space with the norm
    \begin{equation*}
    	\|v\|_{L^p(\Omega,\mathcal{F},\mathbb{P};L^r(0,T;X))} = \left(\mathbb{E} \|v\|_{L^r(0,T;X)}^p \right )^{1/p};
    \end{equation*}
when $r=\infty$, then the norm in the space $L^p(\Omega,\mathcal{F},\mathbb{P};L^\infty(0,T;X))$ is given by  
    \begin{equation*}
    	\|v\|_{L^p(\Omega,\mathcal{F},\mathbb{P};L^\infty(0,T;X))} = \left(\mathbb{E} \|v\|_{L^\infty(0,T;X)}^p \right )^{1/p}.
    \end{equation*} 
Let $K$ be a separable Hilbert space, $\tilde{W}$ be a $K$-cylindrical Wiener process on $(\Omega, \mathcal{F},\mathbb{F},\mathbb{P})$ and $\tilde{H}$ be a Hilbert space. We denote by $\mathcal{M}^2(\Omega \times [0,T]; L_2(K, \tilde{H}))$ the space of all equivalence classes of $\mathbb{F}$-progressively measurable processes $\phi:\Omega \times [0,T] \to L_2(K, \tilde{H})$ satisfying
\begin{equation*}
  \mathbb{E} \int_0^T \|\phi(t)\|_{L_2(K, \tilde{H})}^2 \d t < \infty.
\end{equation*}
For any $\phi \in \mathcal{M}^2(\Omega \times [0,T]; L_2(K, \tilde{H}))$, the process $M$ defined by $M(t) = \int_0^t \phi(s) d\tilde{W}(s)$, with $t\in[0,T]$ is a $\tilde{H}$-valued martingale. Moreover, the following It\^o isometry holds true
   \begin{equation}\label{eq3.1}
   	\mathbb{E} \left( \left\| \int_0^t \phi(s) \d \tilde{W}(s)  \right\|_{\tilde{H}}^2 \d s \right) = \mathbb{E} \int_0^t \|\phi(s)\|_{L_2(K,\tilde{H})}^2 \d s,\ \ \forall t\in[0,T],
   \end{equation}
and we have the following Burkholder-Davis-Gundy inequality
\begin{equation}\label{Eq:BDG}
\mathbb{E} \sup_{s\in[0,t]} \left\| \int_0^s \phi(\tau) \d \tilde{W}(\tau)  \right\|_{\tilde{H}}^p \leq C(p) \mathbb{E} \left( \int_0^t \|\phi(s)\|_{L_2(K,\tilde{H})}^2 \d s \right)^{p/2}, ~\forall p \in (1,\infty),~ \forall t \in [0,T].
\end{equation}
We refer the reader to \cite[Chapter 4]{Prato} for more detail on the theory of stochastic integration. \newline
Now, we introduce the following standing assumptions.
\begin{assumption} \label{Hypo2}
Throughout of this work, we suppose that we are given some functions $b_k,\, f_k,\, \sigma_k:\mathcal{O} \to \mathbb{R}^3$, $ k \in \mathbb{N}$, $j_k:\mathcal{O} \to \mathbb{R}^3$, $ k \in \{1,\ldots,N\}$ satisfying the following set of conditions:
\item[(H1)] For $ k \in \mathbb{N}$, $b_k := (b_k^{(1)}, b_k^{(2)},b_k^{(3)}),\, f_k := (f_k^{(1)}, f_k^{(2)},f_k^{(3)}),\, \sigma_k := (\sigma_k^{(1)}, \sigma_k^{(2)},\sigma_k^{(3)}) \in \mathbb{W}^{1,\infty}$, $j_k := (j_k^{(1)}, j_k^{(2)},j_k^{(3)}) \in \mathbb{W}^{1,\infty}$. 	
\item [(H2)] For $k \in \{1,\ldots,N\}, \, j_k$ is divergence free vector fields, i.e., $\diver j_k =0$, and $j_k =0$ on $\Gamma$. 
\item[(H3)] We assume also that the vector fields $b_k,\, f_k,\, j_k$ satisfy:
\begin{equation}\label{eq3.3}
	\begin{aligned}
		& C_5 : = \sum\limits_{k=1}^{\infty} \left(\|b_k\|_{\mathbb{L}^\infty}^2 + 
		\|\diver b_k\|_{L^\infty(\mathcal{O})}^2 \right) < \infty, \\
		& C_6 : = \sum\limits_{k=1}^{\infty} \left(\|f_k\|_{\mathbb{L}^\infty}^2 + 
		\|\diver f_k\|_{L^\infty(\mathcal{O}}^2 \right) < \infty,  \\
		& C_7 : = \sum\limits_{k=1}^{N} \|j_k\|_{\mathbb{W}^{1,\infty}}^2  < \infty,  \\
		& C_8 := \sum\limits_{k=1}^{\infty} \|\sigma_k\|_{\mathbb{L}^\infty}^2 < \infty,   
	\end{aligned}
\end{equation}
and
\begin{equation}\label{eq3.4}
	\begin{aligned}
		& \sum_{i,j=1}^{3} \left(2 \delta_{ij} - \sum_{k=1}^{\infty} b_k^{(i)}(x) b_k^{(j)}(x) \right) \zeta_i \zeta_j \geq c_1 |\zeta|^2, \\ 
		&\sum_{i,j=1}^{3} \left(2 \delta_{ij} - \sum_{k=1}^{\infty} f_k^{(i)}(x) f_k^{j}(x)\right) \zeta_i \zeta_j \geq c_2 |\zeta|^2, \\
		& \sum_{i,j = 1}^{3} \left(2 \delta_{ij} - \sum_{k=1}^{N} j_k^{(i)}(x) j_k^{(j)}(x)\right) \zeta_i \zeta_j \geq c_3 |\zeta|^2, \quad \zeta = (\zeta_1,\zeta_2,\zeta_3) \in \mathbb{R}^3,
	\end{aligned}
\end{equation}
for some $c_1, \, c_2, \, c_3 \in (0,2]$.  
\item [(H4)] For all $k\in\mathbb{N}$, $\sigma_k$ is a divergence free vector fields, that is $\diver\sigma_k = 0$ in $\mathcal{O}$. Furthermore, $\sigma_k = 0$ on $\Gamma$, and $\sigma_k$ satisfies the following inequality: 
\begin{equation}\label{eqt3.5}
	\sum_{i,j = 1}^{3} \left(2 \delta_{ij} - \sum_{k=1}^{\infty} \sigma_k^{(i)}(x) \sigma_k^{(j)}(x) \right) \zeta_i \zeta_j \geq c_4 |\zeta|^2, ~ \zeta = (\zeta_1,\zeta_2,\zeta_3) \in \mathbb{R}^3, ~ \text{for some} ~ c_4 \in (0,2].
\end{equation}
\end{assumption}
Let us put 
	\begin{equation}\label{Eqt3.6}
		\begin{aligned}
			F_1(\bu) \bh &= \sum_{k=1}^{\infty} [(b_k \cdot \nabla) \bu] h_k, \quad &\bu \in V,\quad \bh= (h_k) \in U_1:=l^2(\mathbb{N}), \\
			F_2(\bw) \zeta&= \sum_{k=1}^{\infty} [(f_k \cdot \nabla) \bw] \zeta_k, \quad &\bw \in \mathbb{H}_0^1(\mathcal{O}), \quad \zeta= (\zeta_k) \in U_1, \\
			G(\bM) \bv&= \sum_{k=1}^{\infty} [ (\sigma_k \cdot \nabla) \bM] v_k, \quad &\bM \in V_1, \quad \bv= (v_k) \in U_1, \\
			F_3(\bH) \by&= \sum_{k=1}^{N} [(j_k \cdot \nabla) \bH ] y_k, \quad &\bH \in V_1, \quad \by= (y_k) \in U_2:=\mathbb{R}^N.
		\end{aligned}
	\end{equation}  
In the following proposition, we will give some properties of the maps $ F_1: V \to L_2(U_1,H)$, $ F_2: \mathbb{H}_0^1 \to L_2(U_1,\mathbb{L}^2)$, $F_3 : V_1 \to L_2(U_2,\mathbb{L}^2)$ and $G: V_1 \to L_2(U_1,\mathbb{L}^2)$. 
\begin{proposition}\label{propo-1}
The maps $F_1,, F_2,, F_3$, and $G$ are Lipschitz continuous, and the following inequalities hold:
\begin{equation}\label{eq3.9}
\begin{aligned}
\|F_1(\bu)\|_{L_2(U_1,H)}^2 
&= \sum_{k=1}^{\infty} \sum_{i,j=1}^{3} (b_k^{(i)}
\partial_i \bu, b_k^{(j)} \partial_j \bu)\leq (2-c_1) \lvert \nabla \bu \rvert^2, 
\\
\Vert F_2(\bw) \Vert_{L_2(U_1,\mathbb{L}^2)}^2 
&= \sum_{k=1}^{\infty} \sum_{i,j=1}^{3}
(f_k^{(i)} \partial_i \bw, f_k^{(j)} \partial_j\bw)\leq (2-c_2) \lvert \nabla \bw \rvert^2,
    \\
\Vert G(\bM) \Vert_{L_2(U_1,\mathbb{L}^2)}^2 
&= \sum_{k=1}^{\infty} \sum_{i,j=1}^{3} (\sigma_k^{(i)}
\partial_i \bM, \sigma_k^{(j)} \partial_j \bM) \leq (2-c_3) \lvert \nabla \bM \rvert^2,
	\\
\Vert F_3(\bH) \Vert_{L_2(U_2,\mathbb{L}^2)}^2 
&= \sum_{k=1}^{N} \sum_{i,j=1}^{3}
(j_k^{(i)} \partial_i \bH, \partial_j \bH) \leq (2-c_4) \lvert \nabla \bH \rvert^2,   
\end{aligned}
\end{equation}
for some $c_1, c_2, c_3,c_4 \in (0,2]$. Moreover, the maps $F_1,\, F_2,\, F_3$ and $G$ extend to linear maps $F_1: H \to L_2(U_1,V')$, $F_2 : \mathbb{L}^2 \to L_2(U_1,\mathbb{H}^{-1})$, $G: \mathbb{L}^2 \to L_2(U_1,V'_1)$ and $F_3 : \mathbb{L}^2 \to L_2(U_2,V'_2)$, respectively, and
\begin{equation}\label{Eq3.10}
	\begin{aligned}
		\|F_1(\bu)\|_{L_2(U_1,V')}^2 &\leq 2C_1 |\bu|^2, \quad &\forall \bu \in H,
		\\
		\|F_2(\bw)\|_{L_2(U_1,\mathbb{H}^{-1})}^2 &\leq 2C_6 |\bw|^2, \quad &\forall \bw \in \mathbb{L}^2, 
		\\
		\|G(\bM)\|_{L_2(U_1,V'_1)}^2 &\leq C_8 |\bM|^2, \quad &\forall \bM \in \mathbb{L}^2 \cap H_{\bn}, 
		\\
		\|F_3(\bH)\|_{L_2(U_2,V'_2)}^2 &\leq C_7 |\bH|^2, \quad &\forall \bH \in \mathbb{L}^2 \cap H_{\bn}.
	\end{aligned}
\end{equation}
\end{proposition} 
\begin{proof}
    The proof is given in Appendix \ref{Appendix-1}.
\end{proof}

\noindent
Now, using the previous notations, the model \eqref{Eq4.1}-\eqref{eqt4.3} can be formally written in the following abstract form
	\begin{subequations}\label{Eq4.3}
		\begin{align}
			& \bu(t) + \int_0^t [B_0(\bu(s),\bu(s)) + \nu A \bu(s) + \alpha R_0(\bu(s), \bw(s))] \d s - \mu_0 \int_0^t  M_0(\bM(s),\bH(s)) \d s \notag 
			\\
			& \quad = \bu_0 + \mu_0 \int_0^t R_1(\bH(s), \bH(s)) \d s + \int_0^t F_1(\bu(s)) \d \beta^1(s) ~ \text{in}~ V', \label{eqt4.1a} 
			\\
			& \bw(t) + \int_0^t [B_1(\bu(s),\bw(s)) - (\lambda_1 + \lambda_2) R_5(\bw(s)) + \lambda_1 A_1 \bw(s) - 2 \alpha R_2(\bu(s),\bw(s))] \d s \notag 
			\\ 
			&\quad = \bw_0 + \mu_0 \int_0^t R_3(\bM(s),\bH(s)) \, \d s + \int_0^t F_2(\bw(s)) \d \beta^2(s) ~ \text{in}~ \mathbb{H}^{-1}, \label{eqt4.2a}
			\\
			& \bM(t) + \int_0^t [B_2(\bu(s),\bM(s)) - R_3(\bw(s),\bM(s))] \d s + \frac{1}{\tau} \int_0^t (\bM(s) - \chi_0 \bH(s)) \d s \notag
			\\
			& = \bM_0 - \lambda \int_0^t \curl \curl\bM(s) \d s + \lambda \int_0^t \nabla \diver \bM(s) \d s + \int_0^t G(\bM(s)) \d \beta^3(s) ~ \text{in}~ V'_1, \label{eqt4.3a}
			\\
			&  \bB(t) + \frac{1}{\sigma} \int_0^t R_6(\bH(s)) \, \d s - \int_0^t  \tilde{M}_2(\bu(s),\bB(s)) \d s = \bB_0  + \int_0^t F_3(\bH(s))  \d \beta^4(s) ~ \text{in}~ V'_2, \label{eqt4.4a}   
		\end{align}
	\end{subequations}		
for all $t \in[0,T]$ with $\bB= \mu_0 (\bM +\bH)$. 
The notion of weak martingale solution to system \eqref{Eq4.3} is defined as follows:
\begin{definition}\label{def2}
By a weak martingale solution of the system \eqref{Eq4.3} on $[0,T]$, we mean a system consisting of a complete and filtered probability space $(\bar{\Omega}, \bar{\mathcal{F}} ,\bar{\mathbb{F}},\bar{\mathbb{P}})$, with the filtration $\bar{\mathbb{F}}= (\bar{\mathcal{F}}_t)_{t \in [0,T]}$ satisfying the usual conditions, and $\bar{\mathbb{F}}$-progressively measurable processes
$(\bar{\bu}(t), \bar{\bw}(t), \bar{\bM}(t), \bar{\bH}(t), \bar{\beta}^i(t))_{t \in [0,T]}$, $i=1,\ldots,4$ such that
		
\item[1.] $ \{\bar{\beta}^i(t)\}_{t \in [0,T]} $ is a real-valued Wiener process, with $i \in \{1,2,3,4\}$,
		
\item[2.] $(\bar{\bu}, \bar{\bw}, \bar{\bM}, \bar{\bH}): \bar{\Omega} \times [0,T] \to V \times \mathbb{H}_0^1 \times V_1 \times V_1$ and $ \bar{\mathbb{P}}$-a.s.
	\begin{equation}
		\begin{aligned}
			\bar{\bu} &\in \mathcal{C}_w([0,T];H) \cap L^\infty(0,T;H) \cap L^2(0,T;V), \\
			\bar{\bw} &\in \mathcal{C}_w([0,T];\mathbb{L}^2) \cap L^\infty(0,T;\mathbb{L}^2) \cap L^2(0,T;\mathbb{H}_0^1), \\
			\bar{\bM}, \, \bar{\bH} &\in \mathcal{C}_w([0,T];\mathbb{L}^2) \cap L^\infty(0,T;\mathbb{L}^2) \cap L^2(0,T; V_1),	
		\end{aligned}
\end{equation}  
\item[3.] the following identities hold
		\begin{align}\label{eq4.6}
				\langle \bar{\bu}(t),\bv \rangle_{V^\prime, V}
				&= \langle\bu_0, \bv \rangle_{V^\prime, V} + \int_{Q_t}(\bar{\bu} \cdot \nabla) \bv \cdot \bar{\bu} \, \d x \, \d s - \nu \int_{Q_t} \nabla \bar{\bu} : \nabla \bv \, \d x \, \d s \notag
				\\
				&\quad - \mu_0 \int_{Q_t} [(\bar{\bM} + \bar{\bH})\cdot \nabla]\bv \cdot \bar{\bH}(s) \, \d x \, \d s
				\\
				& \quad  - \alpha \int_{Q_t} (\curl \bar{\bu} - 2\bar{\bw}) \cdot \curl \bv \, \d x \, \d s + \left \langle \int_0^t F_1(\bar{\bu}(s))\d \bar{\beta}^1(s), \bv\right\rangle_{V^\prime, V}, \notag
		\end{align}
\begin{equation}\label{eqt4.7}
\begin{aligned}
\langle\bar{\bw}(t), \psi \rangle_{\mathbb{H}^{-1}, \mathbb{H}_0^1} 
&= \langle\bw_0, \psi \rangle_{\mathbb{H}^{-1}, \mathbb{H}_0^1}  + \int_{Q_t} [(\bar{\bu} \cdot \nabla) \psi \cdot \bar{\bw} - (\lambda_1 + \lambda_2)  \diver\bar{\bw} \diver \psi] \,\d x \,\d s \\
&\quad + \int_{Q_t} (- \lambda_1 \nabla \bar{\bw} : \nabla \psi + 2 \alpha \bar{\bu} \cdot \curl\psi  - 4\alpha \bar{\bw} \cdot \psi) \, \d x \,\d s \\ 
&\quad  + \mu_0 \int_{Q_t} (\tilde{\bM} \times \bar{\bH}) \cdot \psi \, \d x \,\d s + \left \langle \int_0^t F_2(\bar{\bw}(s))\d \bar{\beta}^2(s), \psi\right\rangle_{\mathbb{H}^{-1}, \mathbb{H}_0^1}, 
\end{aligned}
\end{equation}  		
\begin{equation}\label{eq4.8}
			\begin{aligned}
				\langle\bar{\bM}(t), \psi_1\rangle_{V^\prime_1, V_1}
				&= \langle\bM_0, \psi_1\rangle_{V^\prime_1, V_1} + \int_{Q_t} (\bar{\bu} \cdot \nabla) \psi_1 \cdot \bar{\bM} \, \d x \, \d s + \int_{Q_t} (\bar{\bw} \times \bar{\bM}) \cdot \psi_1 \, \d x \, \d s
				\\
				&\quad - \frac{1}{\tau} \int_{Q_t} (\bar{\bM} -\chi_0 \bar{\bH}) \cdot \psi_1 \, \d x \, \d s - \lambda \int_{Q_t} \curl\bar{\bM}(s) \cdot \curl \psi_1 \, \d x \, \d s 
				\\ 
				&\quad -\lambda \int_{Q_t} \diver \bar{\bM} \cdot \diver\psi_1 \, \d x \, \d s + \left\langle \int_0^t G(\bar{\bM}(s))\d\bar{\beta}^3(s), \psi_1\right\rangle_{V^\prime_1, V_1} , 
			\end{aligned}
		\end{equation}

\begin{equation}\label{eq4.9}
 \begin{aligned}
	\langle \bar{\bB}(t), \psi_2\rangle_{V^\prime_2, V_2} 
	   &= \langle \bB_0, \psi_2 \rangle_{V^\prime_2, V_2} -\frac{1}{\sigma} \int_{Q_t} \curl \bar{\bH} \cdot \curl \psi_2 \, \d x \, \d s
				\\
		&\quad + \int_{Q_t} [\curl(\bar{\bu} \times \bar{\bB})]\cdot \psi_2 \, \d x \, \d s + \left \langle \int_0^t F_3(\bar{\bH}(s))\d\bar{\beta}^4(s), \psi_2 \right \rangle_{V^\prime_2, V_2},
 \end{aligned}
\end{equation}  
for any $\bv \in V$, $ \psi \in \mathbb{H}_0^1$, $\psi_1 \in V_1$ and $\psi_2 \in V_2$, for almost $t \in [0,T]$ and $\bar{\mathbb{P}}$-a.s.         
\end{definition}
\begin{remark}
Note that the equations \eqref{eq4.6}-\eqref{eq4.9} imply that almost surely \\
$\by(\cdot):= (\bar{\bu}(\cdot), \bar{\bw}(\cdot), \bar{\bM}(\cdot), \bar{\bH}(\cdot)) \in \mathcal{C} ([0,T]; V' \times (\mathbb{H}^{-1}(\mathcal{O}) \times V'_1 \times V'_2)$, and since $\by(\cdot)$ is also bounded on $\mathbb{H}$, then it is almost surely in $\mathcal{C}_w([0,T]; \mathbb{H})$. This follows by arguing as in \cite[Chap. 3 Par. 3]{Temam1}.
\end{remark}	
The existence of such a weak martingale solution is ensured by the following main theorem.
\begin{theorem}\label{theo 1}
Let $\mathcal{O} \subset \mathbb{R}^3$ be a simply connected bounded domain of class $\mathcal{C}^\infty$. 
Let $\ell^\ast>0$ be a fixed large positive constant such that
\begin{equation}
 \ell^\ast> \max\left(\frac{1}{\mu_0 \sqrt{\sigma C_0}}, \sqrt{\frac{\mu_0+1}{\mu_0 \sigma C_0}},\frac{\mu_0+1}{\mu_0}, 2 \cdot 3^{16} [C(4)]^2 \right),
\end{equation}	
where $C(4)$ is the constant in \eqref{Eq:BDG} with $p=4$. \newline
We choose our constants $c_1,\, c_2, c_3, c_4 \, \in (0,2]$ as follows: 
\begin{equation}\label{eq4.12}
  \begin{aligned}
    2 - \frac{\mu_0 +1}{\sigma \mu_0 C_0 (\ell^\ast)^2} &< c_4 \leq 2, 
    \quad 
    2 - \frac{1}{\sigma \mu_0^2 (\mu_0+1) (\ell^\ast)^2 C_0} &< c_3 \leq 2, 
    \\
    \frac{2\ell^\ast +2 -2\nu}{\ell^\ast +1} &< c_1 \leq 2, \quad \frac{2\ell^\ast +2 -2\lambda_1}{\ell^\ast +1} < c_2 \leq 2.
  \end{aligned}
\end{equation}
Furthermore, we assume that Assumption \ref{Hypo2} is satisfied and  $(\bu_0, \bw_0,\bM_0,\bH_0) \in \mathbb{H}$. Then, for any $\lambda>0$ such that
\begin{equation}\label{eqt4.14}
\max \left(\frac{1}{\sigma \mu_0^2 \ell^\ast} - \sqrt{\frac{1}{(\sigma \mu_0^2 \ell^\ast)^2} - \frac{(\mu_0+1)(2-c_3)C_0}{\sigma \mu_0^2}}, \frac{\ell^\ast (2-c_3) C_0}{2} + \frac{(2-c_4)C_0 \ell^\ast}{2\mu_0(\mu_0+1)}\right)< \lambda < \frac{1}{\sigma \mu_0^2 \ell^\ast},
\end{equation}
there exists a solution $(\bar{\bu}(t), \bar{\bw}(t), \bar{\bM}(t), \bar{\bH}(t), \bar{\beta}^1(t), \bar{\beta}^2(t), \bar{\beta}^3(t), \bar{\beta}^4(t))_{t \in [0,T]}$ of Problem \eqref{Eq4.3} in the sense of Definition \ref{def2}.
\end{theorem} 
Before proceeding to the next result, we state the following important remark.
\begin{remark}\label{Rem:Limit-gamma}
Note that when  $c_i= 2, \,\,\forall i\in \{1,\ldots,4\}$, \textit{i.e.,} in the absence of noises,  the estimate \eqref{eqt4.14} will  lead to the assumption on $\sigma$ and $\mu_0$ in \cite[Theorem 2.4]{Aristide+Paul}, which enabled the authors of \cite{Aristide+Paul} to carry out the study of the limit $\lambda \to 0$. In fact,  when $c_i= 2, \,\,\forall i\in \{1,\ldots,4\}$, then we can take $\ell^\ast=\frac12$ and assume that $\lambda\in (0,1) $ which is true for either $\sigma=\mu_0=1$ or $\sigma\mu_0^2>2$. When some $c_i\neq 2$, then we do not know whether the estimates we obtained in this paper will enable us to pass to the limit $\lambda\to 0$, and hence we postpone that question to a subsequent paper.
\end{remark}

\begin{proposition}
Let the assumptions of Theorem \ref{theo 1} hold. Let $(\bar{\bu}(t), \bar{\bw}(t), \bar{\bM}(t), \bar{\bH}(t), \bar{\beta}^1(t), \\
\bar{\beta}^2(t), \bar{\beta}^3(t), \bar{\beta}^4(t))_{t \in [0,T]}$ be a weak solution of \eqref{Eq4.3} in the sense of Definition \ref{def2}. Then, there exist two stochastic processes $\bar{\bH}_a$ and $\bar{\bH}_d$ such that $\bar{\bH}= \bar{\bH}_a + \bar{\bH}_d$ and
\begin{equation}
	\begin{aligned}
		\bar{\bH}_a \in L^\infty(0,T;\mathbb{L}^2) \cap L^2(0,T;V), ~ &\bar{\mathbb{P}}\text{-a.s.},
		\\
		\bar{\bH}_d \in L^\infty(0,T;\mathbb{L}^2) \cap L^2(0,T;\mathbb{H}^1), ~ &\bar{\mathbb{P}} \text{-a.s.}, 
		\\
		\diver\bar{\bH}_a=0 \quad \text{and} \quad \curl\bar{\bH}_a= \curl\bar{\bH} \quad &\text{in} \quad Q, ~ \bar{\mathbb{P}}\text{-a.s.},
		\\
		\bar{\bH}_d= \nabla \bar{\varphi}_d, ~ \curl\bar{\bH}_d=0, ~ \diver\bar{\bH}_d= - \diver\bar{\bM}=0 \quad &\text{in} \quad Q, ~ \bar{\mathbb{P}} \text{-a.s.},
	\end{aligned}
\end{equation}
where  for almost every $t \in[0,T]$, $\nabla \bar{\varphi}_d(t) \in E(\mathcal{O})$ $\bar{\mathbb{P}}\text{-a.s.}$ and the potential $\bar{\varphi}_d$ solves $\bar{\mathbb{P}}\text{-a.s.}$ the following problem:
\begin{equation}
	\begin{cases}
		-\Delta \bar{\varphi}_d= \diver\bar{\bM} \quad &\text{in} \quad Q, \\
		\frac{\partial \bar{\varphi}_d}{\partial\bn}= - \bar{\bM} \cdot \bn \quad &\text{on} \quad \Sigma.
	\end{cases}
\end{equation}
Furthermore, for all $\psi\in H^1(\mathcal{O})$,
\begin{equation*}
	\int_{\mathcal{O}} \nabla \bar{\varphi}_d \cdot \nabla \psi \, \d x 
	= - \int_{\mathcal{O}} \bar{\bM} \cdot \nabla \psi \, \d x \quad \bar{\mathbb{P}}\text{-a.s.}
\end{equation*}
\end{proposition}
\begin{proof}
$(\bar{\bu}(t), \bar{\bw}(t), \bar{\bM}(t), \bar{\bH}(t), \bar{\beta}^1(t), 
\bar{\beta}^2(t), \bar{\beta}^3(t), \bar{\beta}^4(t))_{t \in [0,T]}$ be a weak solution of \eqref{Eq4.3} in the sense of Definition \ref{def2}. Then, by Definition \ref{def2},  $(\bar{\bu},\bar{\bw},\bar{\bM},\bar{\bH}) \in L^\infty(0,T;\mathbb{H})\cap L^2(0,T;V \times \mathbb{H}_0^1 \times V_1 \times V_1)$ $\bar{\mathbb{P}}$-a.s. Furthermore, since a.e. $t\in[0,T]$ $\bar{\bH}(t) \in \mathbb{L}^2 \cap H(\diver,\mathcal{O})$, we infer from \cite[Corollary 5.5]{Simader} that for a.e. $t\in[0,T]$ there exists a unique $\bar{\bH}_a(t) \in H(\diver,\mathcal{O})$ with $\diver\bar{\bH}_a(t)=0$ in $\mathcal{O}$, and $\bar{\bH}_a(t)\cdot\bn=0$ on $\partial\mathcal{O}$ and a unique $\nabla \bar{\varphi}_d(t) \in E(\mathcal{O})\cap E_1(\mathcal{O})$ such that $\bar{\bH}(t)=\bar{\bH}_a(t) + \nabla \bar{\varphi}_d(t)$ and 
\begin{equation*} 
	\begin{cases}
		\Delta \bar{\varphi}_d(t)=\diver\bar{\bH}(t)= -\diver \bar{\bM}(t) \quad &\text{in} \quad \mathcal{O}, \\ 
		\frac{\partial \bar{\varphi}_d(t)}{\partial\bn}= \nabla \bar{\varphi}_d(t) \cdot\bn
		= \bar{\bH}(t) \cdot\bn \quad &\text{on} \quad \partial\mathcal{O}.
	\end{cases}
\end{equation*}
We put $\bar{\bH}_d= \nabla \bar{\varphi}_d$. It is clear that $\curl \bar{\bH}_d= 0$ and $\bar{\bH} \cdot \bn= \bar{\bH}_d \cdot \bn=\frac{\partial \bar{\varphi}_d}{\partial\bn}= -\bar{\bM} \cdot\bn$.  Now, since $\bar{\bH} \in L^\infty(0,T;H_{\bn})$ $\bar{\mathbb{P}}$-a.s., we infer from \cite[Theorem 1.4]{Simader} that
\begin{equation*}
	\bar{\bH}_a \in L^\infty(0,T;H)\ \ \text{and}\ \  \bar{\bH}_d= \nabla \bar{\varphi}_d\in L^\infty(0,T;\mathbb{L}^2) \quad \bar{\mathbb{P}}\text{-a.s.}
\end{equation*}
In addition,  we have $\curl\bar{\bH} \in L^2(0,T;\mathbb{L}^2)$ $\bar{\mathbb{P}}$-a.s. Thus, we obtain $\bar{\bH}_a \in L^2(0,T;V)$ $\bar{\mathbb{P}}$-a.s., because $\curl \bar{\bH}_a= \curl\bar{\bH}$ and $\diver \bar{\bH}_a= 0$ $\bar{\mathbb{P}}$-a.s.
Finally, since $\bar{\bH} \in L^2(0,T;V_1)$, $\bar{\bH} \in L^\infty(0,T;H_{\bn})$ and $\curl\bar{\bH} \in L^2(0,T;\mathbb{L}^2)$ $\bar{\mathbb{P}}$-a.s., we easily see that $\nabla \bar{\varphi}_d= \bar{\bH}_d= \bar{\bH} - \bar{\bH}_a \in L^\infty(0,T;\mathbb{L}^2)\cap L^2(0,T;V_1)$ $\bar{\mathbb{P}}$-a.s.
\end{proof}
The remaining part of the paper is devoted to the proof of Theorem \ref{theo 1}. For this purpose, we will mainly use the Galerkin approximation scheme and some important compactness results that will be carried out in the next section.
\section{Proof of the existence of martingale solution} \label{sect5}
We first report the following orthogonal decomposition of the Hilbert space $V_1$ (see \cite[Lemma 5]{Kamel}).
\begin{lemma}\label{lem4}
The space $V_1$ introduced in Section \ref{sect2} has the following orthogonal decomposition with respect to the scalar product \eqref{Eq2.2}:
\begin{equation}
  V_1 = V_2 \oplus \mathcal{H},
\end{equation}
with $\mathcal{H}= \{\nabla \phi, \phi\in H_0^1(\mathcal{O}) \cap H^2(\mathcal{O})\} = V_2^\bot$.	Furthermore, there exists an orthogonal basis $\{\bar{\psi}_j\}_{j=1}^\infty \cup \{\nabla \bar{\phi}_j\}_{j=1}^\infty$ of $V_1$ such that $\{\bar{\psi}_j\}_{j=1}^\infty$ is an orthogonal basis of $V_2$ and $\{\nabla \bar{\phi}_j\}_{j=1}^\infty$ is an orthogonal basis of $\mathcal{H}$. \newline
Following the argument in the proof of  \cite[Lemma 5]{Kamel} it can be shown that $\bar{\phi}_j\in C^\infty(\mathcal{O}) \cap \mathcal{H}$. Now, \cite[Lemma 5]{Kamel} implies that $V_2$ is a separable Hilbert space into which the space $[C_c^\infty(\bar{\mathcal{O}})]^3$ is dense. Thus, we can choose the basis $\bar{\psi}_j$ so that $\bar{\psi}_j\in [C^\infty(\mathcal{O})]^3 \cap V_2 $.  
\end{lemma}
We now proceed to the proof of Theorem \ref{theo 1}. Its proof is divided into many steps. In the first step, we introduce a Galerkin approximation scheme; to do so, we will use a special basis consisting of the eigenfunctions of the Stokes operator $A$ and of the scalar product of $\mathbb{H}_0^1 \times V_1$, and we project \eqref{Eq4.1} onto the linear span of the first $n$ eigenfunctions. Having done so, we then discuss the existence of the Galerkin approximation. In step 2, we derive a priori estimates that are used to show the tightness of the family for the approximating solutions in step 3. Furthermore, in step 3, we construct our new probability space by exploiting 
the Skorokhod embedding theorem, see \cite{Jakubowski_1997}. The last step, i.e. step 4, is devoted to the passage to the limit.
\subsection*{Step 1: Galerkin approximation scheme}
It is well-known that we can find an orthonormal basis  $\{\bar{\upsilon}_j\in [C^\infty(\mathcal{O})]^3: j=1,2,\ldots\}$ of $H$ whose elements are the eigenfunctions of the spectral problem $A\bar{\upsilon}_j= \lambda_j^1 \bar{\upsilon}_j.$ Additionally, it is known that we can find a family of smooth functions $\bar{\Lambda}_j$, $j=1,2,\ldots$, which forms an orthonormal basis of $\mathbb{L}^2$, with the $\bar{\Lambda}_j$ satisfying the spectral problem $A_1\bar{\Lambda}_j= \lambda_j^2 \bar{\Lambda}_j$.  \newline
Let us fix $n\in \mathbb{N}$ and let 
\begin{align*}
\bar{V}_n= \text{span}\{\bar{\upsilon}_1,\ldots,\bar{\upsilon}_n\} \quad \mbox{and} \quad
\bar{H}_{0n}^1= \text{span}\{\bar{\Lambda}_1,\ldots,\bar{\Lambda}_n\}, 
\end{align*} 
the linear spaces spanned by the vectors $\bar{\upsilon}_1,\ldots,\bar{\upsilon}_n$ and $\bar{\Lambda}_1,\ldots,\bar{\Lambda}_n$, respectively. Let $\mathcal{P}_n^1$ and $\mathcal{P}_n^2$ be the orthogonal projections from $H$ onto $\bar{V}_n$ and from $\mathbb{L}^2$ onto $\bar{H}_{0n}^1$, respectively. \newline
Finally,  we consider the family of eigenfunctions $\{\bar{\psi}_j\}_{j=1}^\infty\cup\{\nabla\bar{\phi}_j\}_{j=1}^\infty$ given by Lemma \ref{lem4}, which is assumed to be an orthonormal basis in $V_1$; and for any integer $n\geq 1$, we will consider the following finite-dimensional spaces 
\begin{align*}
V_{1n}
=&\text{span}\{\bar{\psi}_1,\ldots,\bar{\psi}_n\} \cup \text{span} \{\nabla \bar{\phi}_1,\ldots,\nabla \bar{\phi}_n\}, \quad \mathbb{H}_n = \bar{V}_n \times \bar{H}_{0n}^1 \times V_{1n} \times V_{1n},
\end{align*} 
where we endow the space $\mathbb{H}_n$ with the following norm
$$\|(\bu, \bw, \bM, \bH)\|_{\mathbb{H}_n} = (|\bu|^2 + |\bw|^2 + |\bM|^2 + |\bH|^2)^{\frac 12}, \quad \forall (\bu,\bw,\bM,\bH) \in \mathbb{H}_n.$$
We denote by $\mathcal{P}_n^3$ the orthogonal projections on $V_{1n}$ w.r.t the inner product in $\mathbb{L}^2$. \newline
We look for an approximation to a solution of problem \eqref{Eq4.1}-\eqref{eqt4.3} or problem \eqref{Eq4.3} and the initial condition $(\bu_{0,n},\bw_{0,n},\bM_{0,n}, \bH_{0,n})$ in the following form:
\begin{equation}\label{eqt5.2}
\begin{aligned}
	\bu_n(t) &= \sum_{k=1}^{n} a_k^{n}(t)\bar{\upsilon}_k, \quad \bw_n(t) = \sum_{k=1}^{n} b_k^{n}(t) \bar{\Lambda}_k,
	\\
	\bM_n(t) &= \bM_{n,1}(t) + \bM_{n,2}(t) = \sum_{k=1}^{n} c_k^{n}(t) \bar{\psi}_k + \sum_{k=1}^{n} d_k^{n}(t) \nabla \bar{\phi}_k,
	\\
	\bH_n(t) &= \bH_{n,1}(t) + \bH_{n,2}(t) = \sum_{k=1}^{n} e_k^{n}(t) \bar{\psi}_k - \sum_{k=1}^{n} d_k^{n}(t) \nabla \bar{\phi}_k,
\end{aligned}
\end{equation} 
$$ \bu_{0,n} = \mathcal{P}_n^1 \bu_0 := \bu_{0,n},     ~ 
\bw_{0,n} = \mathcal{P}_n^2 \bw_0 := \bw_{0,n},        ~
\bM_{0,n} = \mathcal{P}_n^3 \bM_0 := \bM_{0,n},  ~
\bH_{0,n} = \mathcal{P}_n^3 \bH_0 := \bH_{0,n},
$$
for every $n \in \mathbb{N}$. As $n \to \infty$ we have 
\begin{equation}\label{eq4.1}
\begin{aligned}
	& \bu_{0,n} \to \bu_0 \quad \text{in} \quad \mathbb{L}^2, \quad \bw_{0,n} \to \bw_0 \quad \text{in} \quad \mathbb{L}^2, \quad \bM_{0,n} \to \bM_0 \quad \text{in} \quad \mathbb{L}^2,
	\\
	& \bH_{0,n} \to \bH_0 \quad \text{in} \quad \mathbb{L}^2.
\end{aligned}
\end{equation}
Hence, from \eqref{eq4.1} we see that 
    \begin{equation}\label{eq5.4}
    	|\bu_{0,n}| \leq |\bu_0|, \quad |\bw_{0,n}| \leq |\bw_0|, \quad |\bM_{0,n}| \leq |\bM_0|, \quad |\bH_{0,n}| \leq |\bH_0|.
    \end{equation}
The decompositions of $\bM_n$ and $\bH_n$ were chosen in such a way that $\bB_n$ satisfies 
\begin{equation*}
\diver \bB_n = 0 \quad \text{in} \quad Q,
\end{equation*}
with $\bB_n(t)= \mu_0(\bM_n(t) + \bH_n(t))= \mu_0 \sum \limits_{k=1}^{n}(c_k^{n}(t) + e_k^{n}(t)) \bar{\psi}_k$. \newline
Let $T >0$ be a fixed positive time. We consider the filtered probability space $(\Omega, \mathcal{F}, \mathbb{F} = (\mathcal{F}_t)_{t\in[0,T]}, \mathbb{P})$ and require that the sequence $(\bu_n,\bw_n,\bM_n,\bH_n)_n$ satisfies the following finite dimensional problem: for all $t \in [0,T]$
\begin{subequations}\label{Eq:Galerkin-App}
\begin{align}
	& \bu_n(t) + \int_0^t (\nu A \bu_n + \mathcal{P}_n^1 [B_0(\bu_n,\bu_n)  - \mu_0 M_0(\bM_n,\bH_n) - \mu_0 R_1(\bH_n, \bH_n) + \alpha R_0(\bu_n, \bw_n)]) \,\d s \notag
	\\
	&\quad = \bu_{0,n} + \int_0^t \mathcal{P}_n^1 F_1(\bu_n(s)) \d \beta^1(s), \label{eq4.2a} 
	\\
	& \bw_n(t) + \int_0^t (\lambda_1 A_1 \bw_n + \mathcal{P}_n^2 [B_1(\bu_n,\bw_n) - (\lambda_1 + \lambda_2) R_5(\bw_n) - 2 \alpha R_2(\bu_n,\bw_n) - \mu_0 R_3(\bM_n,\bH_n)]) \d s \notag
	\\
	&\quad = \bw_{0,n} + \int_0^t \mathcal{P}_n^2 F_2(\bw_n(s)) \, \d \beta^2(s), \label{eq4.2b}
	\\
	& \bM_n(t) + \int_0^t \mathcal{P}_n^3 [B_2(\bu_n,\bM_n) - R_3(\bw_n,\bM_n)]\, \d s  + \lambda \int_0^t R_6(\bM_n)\, \d s - \lambda \int_0^tR_5(\bM_n)\, \d s \notag
	\\
	&\quad= \bM_{0,n} - \frac{1}{\tau} \int_0^t (\bM_n - \chi_0 \bH_n)\, \d s +  \int_0^t \mathcal{P}_n^3 G(\bM_n(s)) \, \d \beta^3(s), \label{eq4.2c}
	\\
	& \bB_n(t) + \int_0^t \frac{1}{\sigma} \mathcal{P}_n^3 R_6(\bH_n) - \mathcal{P}_n^3 \tilde{M}_2(\bu_n,\bB_n)\, \d s = \bB_{0,n} + \int_0^t \mathcal{P}_n^3 F_3(\bH_n) \, \d \beta^4(s). \label{eq4.2d}
\end{align}
\end{subequations}
For each $n$, we consider the following mapping $\Upsilon_n: \mathbb{H}_n \to \mathbb{H}_n$ defined by

$\Upsilon_n(\bu, \bw, \bM, \bH) 
=\begin{pmatrix}
\mathcal{P}_n^1 [B_0(\bu,\bu) + \nu A \bu - \mu_0 M_0(\bM,\bH) - \mu_0 R_1(\bH, \bH) + \alpha R_0(\bu, \bw)]
\\
\\
\mathcal{P}_n^2 [B_1(\bu,\bw) - (\lambda_1 + \lambda_2) R_5(\bw) + \lambda_1 A \bw - 2 \alpha R_2(\bu,\bw) - \mu_0 R_3(\bM,\bH)]
\\
\\
\mathcal{P}_n^3 [B_2(\bu,\bM) - R_3(\bw,\bM)] + \frac{1}{\tau} (\bM - \chi_0 \bH) + \lambda [R_6 (\bM) - R_5(\bM)]
\\
\\
\mathcal{P}_n^3 [\frac{1}{\sigma} R_6(\bH) - \tilde{M}_2(\bu,\bB)] 
\end{pmatrix}$.
In the following lemma, we prove that the mapping $\Upsilon_n$ is locally Lipschitz continuous. This is one of the key tools to show that the solutions of the system \eqref{eq4.2a}-\eqref{eq4.2d} exist. 

\begin{lemma}\label{lem3}
We assume Assumption \ref{Hypo2} is satisfied. Then, for each $n \in \mathbb{N}$, the mapping $\Upsilon_n$ is locally Lipschitz continuous; that is, for every $r > 0$, there exists a positive constant $C$ depending on $r$, $\mu_0$, $\sigma$, $\alpha$, $\chi_0$, $n$, and $\lambda$ such that
\begin{equation}\label{eq4.3}
	\|\Upsilon_n(y_1) - \Upsilon_n(y_2)\|_{\mathbb{H}_n} \leq C \|y_1 - y_2\|_{\mathbb{H}_n},
\end{equation}
for any $y_1 = (\bu_1, \bw_1, \bM_1,\bH_1), \, y_2 = (\bu_2, \bw_2, \bM_2,\bH_2) \in \mathbb{H}_n$, with $\|y_1\|_{\mathbb{H}_n} \leq r$ and  $\|y_2\|_{\mathbb{H}_n} \leq r$.
\end{lemma}
\begin{proof}
Let $ n \in \mathbb{N}$ be fixed. Let $y_1 = (\bu_1, \bw_1, \bM_1,\bH_1), ~ y_2 = (\bu_2, \bw_2, \bM_2,\bH_2) \in \mathbb{H}_n$, and $\by = (\bu, \bw, \bM,\bH) \in \mathbb{H}_n$. 
Since all norms are equivalent on finite dimensional space $\mathbb{H}_n$, and the components of $\Upsilon_n$ are the sums of continuous linear and bilinear maps, we easily see that there exists a constant $C>0$ such that 
\begin{equation*}
	\|\Upsilon_n(y_1) - \Upsilon_n(y_2)\|_{\mathbb{H}_n} \leq C \|y_1 - y_2\|_{\mathbb{H}_n} (1+\lVert y_1\rVert_{\mathbb{H}_n}+\lVert y_1\rVert_{\mathbb{H}_n}),
\end{equation*}
from which we easily conclude the proof of the lemma.
\end{proof}
By Lemma \ref{lem3}, the map $Y_n$ is locally Lipschitz. Moreover, by Proposition \ref{propo-1}, $\mathcal{P}_n^1 F_1(\cdot)$, $\mathcal{P}_n^2 F_2(\cdot)$, $\mathcal{P}n^3 G(\cdot)$, and $\mathcal{P}n^3 F_3(\cdot)$ are Lipschitz continuous. Hence, owing to \cite[Theorem 38, p. 303]{Protter}, there exist an adapted process $y_n = (\bu_n,\bw_n,\bM_n,\bH_n)$, a stopping time $\tau_{\max}$, and a sequence of stopping times $(\bar{\tau}^R_n)_{R\in \mathbb{N}}$ such that 
\begin{enumerate}
	\item  $(\bar{\tau}^R_n)_{R\in \mathbb{N}}$ is increasing almost surely and $\bar{\tau}^R \nearrow \tau_{\max}$ almost surely,
	\item  $ (\bu_n (\cdot \wedge \bar{\tau}_n^R) ,\bw_n(\cdot \wedge \bar{\tau}_n^R),\bM_n(\cdot \wedge \bar{\tau}_n^R),\bH_n(\cdot \wedge \bar{\tau}_n^R))\in C([0,T]; \mathbb{H}_n)$,
	\item for all $t\in [0,T]$, with probability 1, we have 
	\begin{equation*}
		y_n(t\wedge \bar{\tau}^R_n )= y_n(0)+\int_0^{t\wedge \bar{\tau}^R_{n}} \Upsilon_n(y_n(s))\,\d s + \int_0^{t\wedge \bar{\tau}^R_n } Z(y_n(s)) \,\d\mathbf{W},
	\end{equation*}
	where $Z(y_n)=\mathrm{diag}(\mathcal{P}^1_nF_1(\bu_n), \mathcal{P}^2_nF_2(\bw_n), \mathcal{P}^3_nG(\bM_n), \mathcal{P}^3_nF_3(\bH_n))$ and 
	$\mathbf{W}= (\beta^1, \beta^2, \beta^3, \beta^4)$. 
\end{enumerate}
The local process  $(y_n; \tau_{\text{max}})$ is called a maximal solution to the system \eqref{Eq:Galerkin-App}, see \cite{ZB+EH+PR} for more detail on local and maximal local solutions of a stochastic evolution equations. \newline
We point out that in the calculations below, we will use the norm which is the norm inherited from the space $\mathbb{H}$. So, we will use the It\^o Lemma for the Lyapunov functional
\[
\mathcal{E}_{tot}: \mathbb{H}_n \ni (\bu,\bw,\bM,\bH)\mapsto (\lvert \bu \rvert^2 + \lvert \bw \rvert^2 + \lvert \bM \rvert^2 + \mu_0 \lvert \bH \rvert^2) \in [0,\infty).
\]
Next, to prove the global existence, we need to apply the Khasminski non-explosion test. For this purpose, for every $n, \, R \in \mathbb{N}$, we set 
\begin{equation*}
\tau_n^R:= \inf\{t>0: (|\bu_n(t)|^2 + |\bw_n(t)|^2 + |\bM_n(t)|^2 + \mu_0 |\bH_n(t)|^2)^\frac{1}{2} \geq R\} \wedge T.
\end{equation*}
In the next subsection, we will derive crucial uniform estimates for the approximating solutions.  
\subsection*{Step 2: A priori estimates for the approximating solutions}
Hereafter, the symbol $C$ will denote a constant that can change from one term to the other and depends on the problem data, but is independent of $n\in \mathbb{N}$. 
 \dela{
\begin{lemma}\label{lem2}
Let the assumptions of Theorem \ref{theo 1} be satisfied. Then, there exists a positive constant $C$ such that for every $n\in \mathbb{N}$ and for all $t\in [0,T]$,
	\begin{equation}\label{Eq5.8}
		\begin{aligned}
			& \mathbb{E} \sup_{s\in[0,t ]} \mathcal{E}_{tot}(\bu_n(s),\bw_n(s),\bM_n(s),\bH_n(s)) + \mathbb{E} \int_0^{t} (\lvert \nabla \bu_n\rvert^2 + \lvert \nabla \bw_n\rvert^2 +    \lvert \diver\bw_n \rvert^2)\, \d s 
			\\
			& + \mathbb{E} \int_0^{t} (\lvert \bM_n \rvert^2 + \lvert \diver\bM_n \rvert^2 + \lvert \curl\bM_n \rvert^2 + \lvert \curl\bH_n \rvert^2 +  \lvert \curl \bu_n - 2\bw_n \rvert^2)\, \d s  
			\\
			& \leq C\mathcal{E}_{tot}(\bu_0,\bw_0,\bM_0,\bH_0). 
		\end{aligned}
	\end{equation}
As a direct consequence of \eqref{Eq5.8} and \eqref{eq2.1}, we have for every $n\in \mathbb{N}$ and for all $t\in [0,T]$,
	\begin{equation}
		\begin{aligned}
			& \mathbb{E} \int_0^t \Vert \bM_n(s) \Vert_{V_1}^2\, \d s\leq  C[1 + \mathcal{E}_{tot}(\bu_0,\bw_0,\bM_0,\bH_0)], \\
			&\mathbb{E} \int_0^t \Vert \bH_n(s) \Vert_{V_1}^2\, \d s \leq  C[ 1 + \mathcal{E}_{tot}(\bu_0,\bw_0,\bM_0,\bH_0)].   
		\end{aligned}
	\end{equation} 
\end{lemma}
}
\begin{lemma}\label{Lem-2}
Let the assumptions of Theorem \ref{theo 1} be satisfied. Then, there exists a positive constant $C$ such that for every $n\in \mathbb{N}$ and for all $t\in [0,T]$,
\begin{equation}\label{eqt5.29-c}
\begin{aligned}
& \mathbb{E}[|\bu_n(t)|^2 + |\bw_n(t)|^2 + |\bM_n(t)|^2 + \mu_0 |\bH_n(t)|^2] 
\\
& + [2\nu - (2-c_1)] \mathbb{E} \int_{0}^{t}  |\nabla \bu_n(s)|^2\, \d s + [2\lambda_1 - (2-c_2)] \mathbb{E} \int_{0}^{t} |\nabla \bw_n(s)|^2 \, \d s
 \\
& + [\lambda(2 - \sigma \mu_0^2 \lambda) - (\mu_0 + 1) (2-c_3) C_0] \mathbb{E} \int_{0}^{t} \lvert \curl \bM_n(s) \rvert^2\, \d s
\\
& + \left[\sigma^{-1} - (2-c_4)C_0 \mu_0^{-1}\right] \mathbb{E} \int_0^{t} \lvert \curl\bH_n(s) \rvert^2\,\d s + 2\tau^{-1} \mathbb{E} \int_0^{t}  |\bM_n(s)|^2\, \d s
   \\
& + \left[2 (\mu_0 + 1) \lambda - (\mu_0 + 1) (2-c_3) C_0 - (2-c_4)C_0 \mu_0^{-1}\right] \mathbb{E} \int_0^{t} |\diver\bM_n(s)|^2\,\d s 
\\
& + 2(\lambda_1 + \lambda_2)\mathbb{E} \int_{0}^{t} \lvert \diver\bw_n(s) \rvert^2\, \d s + 2 \alpha \mathbb{E} \int_{0}^{t} \lvert \curl\bu_n(s) - 2\bw_n(s) \rvert^2\, \d s
\\
&\leq C[1+ \lvert \bu_{0} \rvert^2 + \lvert \bw_{0} \rvert^2 + \lvert \bM_{0} \rvert^2 + \lvert \bH_{0} \rvert^2].
\end{aligned}
\end{equation}
Moreover, we have for every $n\in \mathbb{N}$ and for all $t\in [0,T]$,
	\begin{equation}\label{eqt5.29-d}
		\begin{aligned}
			& \mathbb{E} \int_0^t \Vert \bM_n(s) \Vert_{V_1}^2\, \d s\leq  C[1 + \mathcal{E}_{tot}(\bu_0,\bw_0,\bM_0,\bH_0)], \\
			&\mathbb{E} \int_0^t \Vert \bH_n(s) \Vert_{V_1}^2\, \d s \leq  C[ 1 + \mathcal{E}_{tot}(\bu_0,\bw_0,\bM_0,\bH_0)].   
		\end{aligned}
	\end{equation} 

\end{lemma}

\begin{proof}
Let $n, R\in \mathbb{N}$, $t\in [0,T]$, and $\mathcal{D}_{t}= (0,t \wedge \tau_n^R)  \times \mathcal{O}$. By applying the It\^o formula to the functional  $\Phi(x) = |x|^2$  and  \eqref{eq4.2a}, we find
	\begin{equation*}
		\begin{aligned}
			&|\bu_n(t\wedge \tau_n^R)|^2 + 2\nu \int_{0}^{t\wedge \tau_n^R} |\nabla \bu_n(s)|^2 \, \d s - 2 \mu_0 \int_{0}^{t\wedge \tau_n^R} M_1(\bM_n(s),\bH_n(s),\bu_n(s)) \, \d s \\
			&\quad - 2\mu_0 \int_{0}^{t\wedge \tau_n^R} \langle R_1(\bH_n(s), \bH_n(s)),\bu_n(s) \rangle \, \d s + 2\alpha \int_{0}^{t\wedge \tau_n^R} \langle R_0(\bu_n(s), \bw_n(s)),\bu_n(s) \rangle \,\d s \\
			&\quad= |\bu_{0,n}|^2 + \int_{0}^{t\wedge \tau_n^R} \|F_1(\bu_n(s))\|_{L_2(U_1,H)}^2 \, \d s + 2\int_{0}^{t\wedge \tau_n^R}  \left( \bu_n(s), F_1(\bu_n(s)) \, \d \beta^1(s)\right),
		\end{aligned}
	\end{equation*}
where $\langle \cdot, \cdot\rangle$ denotes the dual pairing between $V'$ and $V$. Here we used the fact that $\langle B_0(\bu_n,\bu_n),\bu_n \rangle= 0$, due to \eqref{eq2.6}. Next, thanks to \eqref{eqt2.2}, we have
	\begin{equation}\label{eqt5.8}
		\begin{aligned}
			& -2\mu_0 M_1(\bM_n,\bH_n,\bu_n) - 2\mu_0 \langle R_1(\bH_n, \bH_n),\bu_n \rangle 
			\\
			& \quad = -2\int_{\mathcal{O}}(\curl\bH_n \times \bB_n) \cdot \bu_n \, \d x + 2\mu_0 \int_{\mathcal{O}}(\bu_n \cdot \nabla) \bM_n \cdot \bH_n \, \d x.
		\end{aligned}
	\end{equation}		
Thus, 
	\begin{equation}\label{eq5.6}
		\begin{aligned}
			&|\bu_n(t\wedge \tau_n^R)|^2 + 2 \nu \int_{0}^{t\wedge \tau_n^R} | \nabla \bu_n(s)|^2 \, \d s - 2 \int_{\mathcal{D}_{t}} (\curl\bH_n(s) \times \bB_n(s)) \cdot \bu_n(s) \, \d x \, \d s \\
			& \quad + 2\mu_0 \int_{\mathcal{D}_{t}} (\bu_n(s) \cdot \nabla) \bM_n(s) \cdot \bH_n(s) \, \d x \, \d s + 2\alpha \int_{0}^{t\wedge \tau_n^R} \langle R_0(\bu_n(s), \bw_n(s)),\bu_n(s) \rangle \, \d s \\
			&\quad= |\bu_{0,n}|^2 + \int_{0}^{t\wedge \tau_n^R} \|F_1(\bu_n(s))\|_{L_2(U_1,H)}^2 \, \d s + 2 \int_{0}^{t\wedge \tau_n^R}  \left( \bu_n(s), F_1(\bu_n(s)) \, \d \beta^1(s)\right).
		\end{aligned}
	\end{equation}	
Applying the It\^o formula to the functional $\Phi(x)= |x|^2$ and \eqref{eq4.2b}, using the fact that $\langle B_1(u_n,\bw_n), \bw_n \rangle = 0$ and
$4 \alpha \langle  R_3(\bu_n,\bw_n),\bw_n \rangle
= -2\alpha |\curl\bu_n - 2\bw_n|^2 + 2 \alpha \langle R_0(\bu_n, \bw_n), \bu_n\rangle$,	
we obtain
	\begin{align}\label{eq5.8}
		& |\bw_n(t\wedge \tau_n^R)|^2 + 2 \int_{0}^{t\wedge \tau_n^R} [\lambda_1 |\nabla \bw_n|^2 + (\lambda_1 + \lambda_2) |\diver\bw_n|^2 + \alpha |\curl \bu_n - 2\bw_n|^2] \, \d s \notag
		\\
		&\quad= |\bw_{0,n}|^2 + 2 \int_{0}^{t\wedge \tau_n^R} (\mu_0 \langle R_3(\bM_n,\bH_n),\bw_n \rangle_{\mathbb{H}^{-1}, \mathbb{H}_0^1} +  \alpha \langle R_0(\bu_n, \bw_n), \bu_n \rangle_{V',V}) \,\d s 
		\\
		&\qquad  + \int_{0}^{t\wedge \tau_n^R} \|F_2(\bw_n)\|_{L_2(U_1,\mathbb{L}^2)}^2  \, \d s
		+ 2 \int_{0}^{t\wedge \tau_n^R}  \left(\bw_n(s), \ F_2(\bw_n(s)) \, \d \beta^2(s)\right). \notag
	\end{align}		
Applying the It\^o formula to the functional $\Phi(x) = |x|^2$ and \eqref{eq4.2c}, using the fact that $$ \langle B_2(u_n,\bM_n),\bM_n \rangle_{V'_1, V_1} =0 \quad \text{and} \quad \langle R_5(\bw_n,\bM_n),\bM_n \rangle_{V'_1, V_1} =0,$$ we find
		\begin{equation*}
		\begin{aligned}
			& |\bM_n(t\wedge \tau_n^R)|^2 - |\bM_{0,n}|^2 + 2\tau^{-1} \int_{0}^{t\wedge \tau_n^R} |\bM_n|^2 \, \d s + 2\lambda \int_{\mathcal{D}_{t}} (\curl \curl \bM_n - \nabla \diver\bM_n) \cdot \bM_n \, \d x \, \d s  
			\\
			&\quad= 2\chi_0 \tau^{-1} \int_{\mathcal{D}_{t}} \bM_n \cdot \bH_n \, \d x \, \d s + \int_{0}^{t\wedge \tau_n^R} \|G(\bM_n)\|_{L_2(U_1,\mathbb{L}^2)}^2 \, \d s
			+ 2 \int_{0}^{t\wedge \tau_n^R} \left(\bM_n, G(\bM_n) \, \d \beta^3(s)\right).
		\end{aligned}
	\end{equation*}	
Besides, using an integration by parts along with the boundary conditions \eqref{eq4.2}, we have
	\begin{equation*}
		2\lambda \int_{\mathcal{O}} (\curl \curl \bM_n - \nabla \diver\bM_n)\cdot \bM_n \,\d x  
		= 2\lambda (|\curl\bM_n|^2 + |\diver\bM_n|^2).
	\end{equation*}
Hence, 
	\begin{equation}\label{eq4.10}
		\begin{aligned}
			& |\bM_n(t\wedge \tau_n^R)|^2 + 2\tau^{-1} \int_{0}^{t\wedge \tau_n^R}  |\bM_n(s)|^2 \, \d s + 2\lambda \int_{0}^{t\wedge \tau_n^R} (|\curl\bM_n(s)|^2 + |\diver\bM_n(s)|^2) \, \d s 
			\\
			&\quad= |\bM_{0,n}|^2 + 2\chi_0 \tau^{-1} \int_{\mathcal{D}_{t}} \bM_n(s) \cdot \bH_n(s) \, \d x \, \d s + \int_{0}^{t\wedge \tau_n^R}  \|G(\bM_n(s))\|_{L_2(U_1,\mathbb{L}^2)}^2 \, \d s,
		\end{aligned}
	\end{equation} 
where we have also used the fact that
	\begin{equation*}
		\int_{0}^{t\wedge \tau_n^R} \left(\bM_n(s), G(\bM_n(s))\, \d \beta^3(s) \right)= \int_{0}^{t\wedge \tau_n^R} \int_{\mathcal{O}} \sum_{k=1}^{\infty} \sum_{j=1}^{3} \sigma_k^{(i)}(x) \frac{\partial}{\partial x_i}\left( \frac 1 2 |(\bM_n(s))_j|^2 \right) \, \d x \, \d\beta_k^3(s)= 0.
	\end{equation*} 
Note that this is true since $\diver\sigma_k =0$ in $\mathcal{O}$ and $\sigma_k|_{\Gamma}= 0$ on $\Gamma=\partial \mathcal{O}$. Here $(\bM_n(s))_j$ denotes the $j$-entry of the vector fields $\bM_n(s)$.  \newline
Now, observe that 
\begin{equation}\label{eq5.22}
\begin{aligned}
\d |\bM_n|^2 &=2(\bM_n,\d \bM_n) + \d [\bM_n,\bM_n]=2(\bM_n,\d \bM_n) + \|G(\bM_n)\|_{L_2(U_1,\mathbb{L}^2)}^2 \d t, 
\\
\d |\bH_n|^2 &= 2(\bH_n,\d \bH_n) + \d [\bH_n,\bH_n]=2(\bH_n,\d \bH_n) + \|G(\bM_n)\|_{L_2(U_1,\mathbb{L}^2)}^2 \d t \\
&\quad \hspace{6.7cm} + \mu_0^{-2} \|F_3(\bH_n)\|_{L_2(U_2,\mathbb{L}^2)}^2 \d t, 
   \\
\d |\bB_n|^2 &=2(\bB_n,\d \bB_n) + \d [\bB_n,\bB_n]=2(\bB_n,\d \bB_n) + \|F_3(\bH_n)\|_{L_2(U_2,\mathbb{L}^2)}^2 \d t,
\end{aligned}
\end{equation} 
where $[X,X]$ is the quadratic variation of the process $X$. Thus, 
	\begin{align*}
		&(\bM_n,\d \bH_n) + (\bH_n,\d \bM_n)
		= \mu_0^{-2} (\bB_n,\d \bB_n) - (\bH_n,\d \bH_n) - (\bM_n,\d \bM_n) \notag \\
		&= \frac{1}{2\mu_0^2} \d |\bB_n|^2 - \frac{1}{2\mu_0^2} \|F_3(\bH_n)\|_{L_2(U_2,\mathbb{L}^2)}^2 \d t - \frac 1 2 \d |\bH_n|^2 + \frac{1}{2\mu_0^2} \|F_3(\bH_n)\|_{L_2(U_2,\mathbb{L}^2)}^2 \d t \notag \\
		&\qquad + \frac 1 2 \|G(\bM_n)\|_{L_2(U_1,\mathbb{L}^2)}^2 \d t - \frac 1 2 \d |\bM_n|^2 + \frac 1 2 \|G(\bM_n)\|_{L_2(U_2,\mathbb{L}^2)}^2 \d t \\
		&= \frac{1}{2\mu_0^2} \d |\bB_n|^2 - \frac 1 2 \d |\bH_n|^2 - \frac 1 2 \d |\bM_n|^2 + \|G(\bM_n)\|_{L_2(U_1,\mathbb{L}^2)}^2 \d t, \notag
	\end{align*}
and in turn, we find that
 \begin{equation}\label{eqt5.23}
  \d |\bH_n|^2 
  = \mu_0^{-2} \d |\bB_n|^2 - \d |\bM_n|^2 + 2 \|G(\bM_n)\|_{L_2(U_1,\mathbb{L}^2)}^2 \d t - 2 (\bM_n,\d \bH_n) - 2 (\bH_n,\d \bM_n).
  \end{equation}
Integrating \eqref{eq5.22}$_3$ in time over $[0,t\wedge \tau_n^R]$ and using \eqref{eq4.2d}, we deduce that
		\begin{align*}
			|\bB_n(t\wedge \tau_n^R)|^2  
			&= |\bB_{0,n}|^2 - 2 \sigma^{-1} \int_{0}^{t\wedge \tau_n^R} \langle R_6(\bH_n),\bB_n\rangle_{V'_2,V_2} \d s + 2\int_{0}^{t\wedge \tau_n^R} \langle \tilde{M}_2(\bu_n,\bB_n),\bB_n\rangle_{V'_2, V_2} \d s \\
			&\quad + 2 \int_{0}^{t\wedge \tau_n^R} \left(\bB_n(s), F_3(\bH_n(s)) \, \d \beta^4(s)\right) + \int_{0}^{t\wedge \tau_n^R} \|F_3(\bH_n(s))\|_{L_2(U_2,\mathbb{L}^2)}^2 \d s.
	\end{align*}  
Besides, we observe that
\begin{align*}
\langle R_6(\bH_n),\bB_n\rangle_{V'_2,V_2}
&= (\curl \bH_n,\curl\bB_n) = \mu_0 |\curl\bH_n|^2  + \mu_0 (\curl\bH_n, \curl\bM_n), \\
\langle \tilde{M}_2(\bu_n,\bB_n),\bB_n\rangle_{V'_2, V_2}  
&=  (\curl (\bu_n \times \bB_n), \bB_n)= \mu_0 (\curl (\bu_n \times \bB_n), \bM_n) -  \mu_0 (\curl\bH_n \times \bB_n, \bu_n).       
   \end{align*}
Consequently, 
	\begin{equation}\label{eqt5.24}
		\begin{aligned}
			&|\bB_n(t\wedge \tau_n^R)|^2 - |\bB_{0,n}|^2 + 2\mu_0 \sigma^{-1} \int_{0}^{t\wedge \tau_n^R} |\curl \bH_n|^2 \, \d s -  \int_{0}^{t\wedge \tau_n^R} \|F_3(\bH_n)\|_{L_2(U_2,\mathbb{L}^2)}^2 \d s \\
			&=  -2\mu_0 \sigma^{-1} \int_{\mathcal{D}_{t}} \curl \bH_n \cdot \curl \bM_n \, \d x \, \d s  + 2 \mu_0 \int_{\mathcal{D}_{t}} \curl (\bu_n \times \bB_n) \cdot \bM_n \, \d x \, \d s  \\
			&\quad - 2 \mu_0 \int_{\mathcal{D}_{t}} (\curl \bH_n \times \bB_n) \cdot \bu_n \, \d x \, \d s + 2 \int_{0}^{t\wedge \tau_n^R}\left(\bB_n, F_3(\bH_n) \, \d \beta^4(s)\right) .
		\end{aligned}
	\end{equation}
Integrating \eqref{eqt5.23} in time over $[0,t\wedge \tau_n^R]$ with $t \in[0,T]$, we infer that
	\begin{align}\label{eqt5.25}
		\mu_0 |\bH_n(t\wedge \tau_n^R)|^2 - \mu_0 |\bH_{0,n}|^2 
		&=  \mu_0^{-1} [|\bB_n(t\wedge \tau_n^R)|^2 - |\bB_{0,n}|^2] - \mu_0 [|\bM_n(t\wedge \tau_n^R)|^2 - |\bM_{0,n}|^2] \notag \\
		&\quad + 2\mu_0 \int_{0}^{t\wedge \tau_n^R} \|G(\bM_n)\|_{L_2(U_1,\mathbb{L}^2)}^2 \d s  \\
		&\quad - 2\mu_0 \int_{\mathcal{D}_{t}} \bM_n \cdot \partial_t \bH_n \, \d x \, \d s - 2\mu_0 \int_{\mathcal{D}_{t}} \bH_n \cdot \partial_t \bM_n \, \d x \, \d s. \notag
	\end{align}
On the other hand, by integration by parts, \eqref{eq4.2} and the fact $\diver\bB_n=0$ in $\mathcal{O}$, we have
	\begin{equation*}
		\begin{aligned}
			- 2\mu_0 \lambda \int_{\mathcal{O}} [\nabla \diver \bM_n - \curl\curl\bM_n] \cdot \bH_n \,\d x 
			&= 2\mu_0 \lambda (\curl\bM_n, \curl \bH_n) - 2\mu_0  \lambda |\diver\bM_n|^2, \\
			(\nabla \diver\bM_n, \bM_n)
			&= - |\diver\bM_n|^2,
		\end{aligned}
	\end{equation*}
and in turn, exploiting \eqref{eq4.2c} and \eqref{eq4.2d}, we have 
	\begin{align}\label{eqt5.26}
		- 2\mu_0 \int_{\mathcal{D}_{t}} \bM_n \cdot \partial_t \bH_n \, \d x \, \d s 
		&= -2 \mu_0  \int_{0}^{t\wedge \tau_n^R} [\lambda (|\curl\bM_n|^2 + |\diver\bM_n|^2) + \tau^{-1} |\bM_n|^2]  \, \d s  \notag  \\
		&\quad +  2\int_{\mathcal{D}_{t}} (\sigma^{-1} \curl\curl \bH_n - \curl (\bu_n \times \bB_n)) \cdot \bM_n \, \d x \, \d s \\
		&\quad + 2\mu_0 \chi_0 \tau^{-1} \int_{\mathcal{D}_{t}} \bH_n \cdot \bM_n \, \d x \, \d s -2 \int_{0}^{t\wedge \tau_n^R}\left(\bM_n, F_3(\bH_n) \d \beta^4(s)\right),  \notag
	\end{align}
\begin{align}\label{eqt5.27}
- 2\mu_0 \int_{\mathcal{D}_{t}} \bH_n \cdot \partial_t \bM_n \, \d x \, \d s 
&= 2\mu_0 \int_{\mathcal{D}_{t}} [(\bu_n \cdot \nabla) \bM_n \cdot \bH_n + \lambda \curl \bM_n \cdot \curl\bH_n]  \, \d x \, \d s  \notag 
\\
&\quad - 2 \mu_0 \int_{0}^{t\wedge \tau_n^R}  (\lambda |\diver\bM_n|^2 + \chi_0 \tau^{-1} |\bH_n|^2) \, \d s 
   \\ 
&\quad -2\mu_0 \int_{\mathcal{D}_{t}} [(\bw_n \times \bM_n) \cdot \bH_n - \tau^{-1} \bM_n \cdot \bH_n] \, \d x \, \d s \notag 
      \\
&\quad    - 2\mu_0 \int_{0}^{t\wedge \tau_n^R}\left(\bH_n, G(\bM_n) \d \beta^3(s)\right). \notag 
\end{align}
Owing to \eqref{eq4.10}, \eqref{eqt5.24}-\eqref{eqt5.27}, and the fact that $(\bw_n \times \bM_n) \cdot \bH_n = (\bM_n\times \bH_n) \cdot \bw_n$, we obtain 
	\begin{equation}\label{eqt5.28}
		\begin{split}
			& \mu_0 |\bH_n(t\wedge \tau_n^R)|^2  + \int_{0}^{t\wedge \tau_n^R} (\sigma^{-1} |\curl\bH_n|^2 + 2\mu_0 \lambda |\diver\bM_n|^2 + 2\mu_0 \chi_0 \tau^{-1} |\bH_n|^2) \d s 
			\\
			&= \mu_0 |\bH_{0,n}|^2 - 2 \int_{\mathcal{D}_{t}} (\curl\bH_n \times \bB_n) \cdot \bu_n \, \d x \, \d s + 2 \mu_0 \int_{\mathcal{D}_{t}} (\bu_n \cdot \nabla) \bM_n \cdot \bH_n \, \d x \, \d s
			\\
			&\quad  +  \int_{\mathcal{D}_{t}} (2\mu_0 \tau^{-1} \bH_n \cdot \bM_n   - 2\mu_0 (\bM_n \times \bH_n) \cdot \bw_n + 2\mu_0 \lambda \curl\bM_n \cdot\curl \bH_n) \d x \, \d s 
			\\
			&\quad + \mu_0 \int_{0}^{t\wedge \tau_n^R} \|G(\bM_n)\|_{L_2(U_1,\mathbb{L}^2)}^2 \d s + \mu_0^{-1} \int_{0}^{t\wedge \tau_n^R} \|F_3(\bH_n)\|_{L_2(U_2,\mathbb{L}^2)}^2 \d s 
			\\
			&\quad - 2\mu_0 \int_{0}^{t\wedge \tau_n^R}\left(\bH_n, G(\bM_n) \d \beta^3(s)\right) + 2\int_{0}^{t\wedge \tau_n^R}\left(\bH_n, F_3(\bH_n) \, \d \beta^4(s)\right).
		\end{split}
	\end{equation}
Adding \eqref{eq5.6}, \eqref{eq5.8}, \eqref{eq4.10} and \eqref{eqt5.28} side by side, we obtain for all $t \in[0,T]$,
	\begin{equation}\label{eqt5.29}
		\begin{aligned}
			& |\bu_n(t\wedge \tau_n^R)|^2 + |\bw_n(t\wedge \tau_n^R)|^2 + |\bM_n(t\wedge \tau_n^R)|^2 + \mu_0 |\bH_n(t\wedge \tau_n^R)|^2 
			\\
			&\quad + 2 \int_{0}^{t\wedge \tau_n^R} (\nu |\nabla \bu_n(s)|^2 + \lambda_1 |\nabla \bw_n(s)|^2 + \lambda |\curl \bM_n(s)|^2 + \lambda |\diver\bM_n(s)|^2) \, \d s \\
			&\quad + 2 \int_0^{t\wedge \tau_n^R}  \left(\mu_0 \lambda |\diver\bM_n(s)|^2 +  \frac{1}{\sigma} |\curl\bH_n(s)|^2 + \tau^{-1} |\bM_n(s)|^2 + \frac{\mu_0\chi_0}{\tau} |\bH_n(s)|^2\right) \d s 
			\\
			&\quad  + 2(\lambda_1 + \lambda_2)\int_{0}^{t\wedge \tau_n^R} |\diver\bw_n(s)|^2 \, \d s + 2 \alpha \int_{0}^{t\wedge \tau_n^R} |\curl\bu_n(s) - 2\bw_n(s)|^2 \, \d s
			\\
			&= |\bu_{0,n}|^2 + |\bw_{0,n}|^2 + |\bM_{0,n}|^2 + \mu_0 |\bH_{0,n}|^2 + \frac{2(\chi_0+\mu_0)}{\tau} \int_{\mathcal{D}_{t}} \bM_n(s) \cdot \bH_n(s) \, \d x \, \d s 
			\\
			&\quad + 2\mu_0 \lambda \int_{\mathcal{D}_{t}} \curl\bM_n(s) \cdot \curl\bH_n(s) \, \d x \d s + (\mu_0 + 1) \int_{0}^{t\wedge \tau_n^R}  \|G(\bM_n(s))\|_{L_2(U_1,\mathbb{L}^2)}^2 \, \d s
			\\
			&\quad  + \int_{0}^{t\wedge \tau_n^R} \|F_1(\bu_n(s))\|_{L_2(U_1,H)}^2 + \|F_2(\bw_n(s))\|_{L_2(U_1,\mathbb{L}^2)}^2  + \frac{1}{\mu_0} \|F_3(\bH_n(s))\|_{L_2(U_2,\mathbb{L}^2)}^2 \, \d s 
			\\
			&\quad + 2 \int_{0}^{t\wedge \tau_n^R} \left(\bu_n(s), F_1(\bu_n(s)) \, \d \beta^1(s)\right) + 2 \int_{0}^{t\wedge \tau_n^R}\left(\bw_n(s), F_2(\bw_n(s)) \, \d \beta^2(s)\right) 
			\\
			&\quad - 2\mu_0\int_0^{t\wedge \tau_n^R} \left(\bH_n(s),G(\bM_n(s)) \, \d \beta^3(s)\right) + 2 \int_0^{t\wedge \tau_n^R} \left( \bH_n(s) , F_3(\bH_n(s)) \, \d \beta^4(s)\right).
		\end{aligned}
	\end{equation}
Observe that
		\begin{equation}\label{eq5.28}
			\begin{aligned}
				2\tau^{-1} (\mu_0 + \chi_0) (\bM_n,\bH_n) 
				& \leq  \lvert \bM_n \rvert^2 + \tau^{-1}(\mu_0 + \chi_0)^2 \lvert \bH_n \rvert^2), \\
				2\mu_0 \lambda (\curl \bM_n, \curl \bH_n)  
				&\leq  (\sigma \mu_0^2 \lambda^2 \lvert \curl \bM_n \rvert^2 + \sigma^{-1} \lvert \curl \bH_n \rvert^2).
			\end{aligned}
		\end{equation} 
Hence, taking the expectation in \eqref{eqt5.29}, using \eqref{eq2.1}, \eqref{eq3.9}, and \eqref{eq5.28}, we deduce that for every $n,\, R \in \mathbb{N}$,
\begin{equation}\label{eqt5.29-a}
\begin{aligned}
& \mathbb{E}[|\bu_n(t\wedge \tau_n^R)|^2 + |\bw_n(t\wedge \tau_n^R)|^2 + |\bM_n(t\wedge \tau_n^R)|^2 + \mu_0 |\bH_n(t\wedge \tau_n^R)|^2] 
\\
& + [2\nu - (2-c_1)] \mathbb{E} \int_{0}^{t\wedge \tau_n^R}  |\nabla \bu_n(s)|^2\, \d s + [2\lambda_1 - (2-c_2)] \mathbb{E} \int_{0}^{t\wedge \tau_n^R} |\nabla \bw_n(s)|^2 \, \d s
 \\
& + [\lambda(2 - \sigma \mu_0^2 \lambda) - (\mu_0 + 1) (2-c_3) C_0] \mathbb{E} \int_{0}^{t\wedge \tau_n^R} \lvert \curl \bM_n(s) \rvert^2\, \d s
\\
& + \left[\sigma^{-1} - (2-c_4)C_0 \mu_0^{-1}\right] \mathbb{E} \int_0^{t\wedge \tau_n^R} \lvert \curl\bH_n(s) \rvert^2\,\d s + 2\tau^{-1} \mathbb{E} \int_0^{t\wedge \tau_n^R}  |\bM_n(s)|^2\, \d s
   \\
& + \left[2 (\mu_0 + 1) \lambda - (\mu_0 + 1) (2-c_3) C_0 - (2-c_4)C_0 \mu_0^{-1}\right] \mathbb{E} \int_0^{t\wedge \tau_n^R} |\diver\bM_n(s)|^2\,\d s 
\\
& + 2(\lambda_1 + \lambda_2)\mathbb{E} \int_{0}^{t\wedge \tau_n^R} \lvert \diver\bw_n(s) \rvert^2\, \d s + 2 \alpha \mathbb{E} \int_{0}^{t\wedge \tau_n^R} \lvert \curl\bu_n(s) - 2\bw_n(s) \rvert^2\, \d s
\\
&\leq |\bu_{0}|^2 + |\bw_{0}|^2 + |\bM_{0}|^2 + \mu_0 |\bH_{0}|^2 + [1 + (\mu_0 + 1) (2-c_3) C_0] \mathbb{E} \int_{0}^{t\wedge \tau_n^R} \lvert \bM_n \rvert^2\,\d s 
    \\
&\quad + [(2-c_4)C_0 \mu_0^{-1} + \tau^{-1}(\mu_0^2 + \chi_0^2)] \mathbb{E} \int_{0}^{t\wedge \tau_n^R} \lvert \bH_n \rvert^2\, \d s.
\end{aligned}
\end{equation}
Now, under the assumptions of Theorem \ref{theo 1}, we can apply the Gronwall lemma and deduce that there exists a positive constant $C$ depending only on $T,\,\mu_0,\,\chi_0,\,C_0,\,\tau,\,c_3$, and $c_4$ such that
\begin{equation}\label{eqt5.29-b}
\begin{aligned}
& \mathbb{E}[|\bu_n(t\wedge \tau_n^R)|^2 + |\bw_n(t\wedge \tau_n^R)|^2 + |\bM_n(t\wedge \tau_n^R)|^2 + \mu_0 |\bH_n(t\wedge \tau_n^R)|^2] 
\\
& + [2\nu - (2-c_1)] \mathbb{E} \int_{0}^{t\wedge \tau_n^R}  |\nabla \bu_n(s)|^2\, \d s + [2\lambda_1 - (2-c_2)] \mathbb{E} \int_{0}^{t\wedge \tau_n^R} |\nabla \bw_n(s)|^2 \, \d s
 \\
& + [\lambda(2 - \sigma \mu_0^2 \lambda) - (\mu_0 + 1) (2-c_3) C_0] \mathbb{E} \int_{0}^{t\wedge \tau_n^R} \lvert \curl \bM_n(s) \rvert^2\, \d s
\\
& + \left[\sigma^{-1} - (2-c_4)C_0 \mu_0^{-1}\right] \mathbb{E} \int_0^{t\wedge \tau_n^R} \lvert \curl\bH_n(s) \rvert^2\,\d s + 2\tau^{-1} \mathbb{E} \int_0^{t\wedge \tau_n^R}  |\bM_n(s)|^2\, \d s
   \\
& + \left[2 (\mu_0 + 1) \lambda - (\mu_0 + 1) (2-c_3) C_0 - (2-c_4)C_0 \mu_0^{-1}\right] \mathbb{E} \int_0^{t\wedge \tau_n^R} |\diver\bM_n(s)|^2\,\d s 
\\
& + 2(\lambda_1 + \lambda_2)\mathbb{E} \int_{0}^{t\wedge \tau_n^R} \lvert \diver\bw_n(s) \rvert^2\, \d s + 2 \alpha \mathbb{E} \int_{0}^{t\wedge \tau_n^R} \lvert \curl\bu_n(s) - 2\bw_n(s) \rvert^2\, \d s
\\
&\leq C[1+ \lvert \bu_{0} \rvert^2 + \lvert \bw_{0} \rvert^2 + \lvert \bM_{0} \rvert^2 + \lvert \bH_{0} \rvert^2],
\end{aligned}
\end{equation}
from which we infer that $\tau_n^R \to \infty$ in probability as $R \to \infty$. Thus, there exists a subsequence $(\tau_n^{R_k})_{k \in \mathbb{N}}$ of $(\tau_n^R)_{R \in \mathbb{N}}$ such that $\tau_n^{R_k} \to \infty$, a.s. Furthermore, since the sequence $(\tau_n^R)_R$ is increasing, we infer $\tau_n^R \to \infty$ a.s. as $R \to \infty$. \newline
Now, because $t\wedge\tau_n^R \to t$ as $R \to \infty$, by passing to the limit in \eqref{eqt5.29-b} and using the Fatou Lemma, we deduce \eqref{eqt5.29-c}
 \newline
Finally, from \eqref{eq2.1} and \eqref{eqt5.29-b}, we deduce \eqref{eqt5.29-c}. This completes the proof of the lemma.
\end{proof}
\begin{lemma}\label{lem6.4}
Let $p \in [2, 4]$ be fixed and let the assumptions of Lemma \ref{eqt5.29-c} be satisfied. Then,  there exists $C>0$ such that for all $t \in[0,T]$ and every $n \in \mathbb{N}$, 
		\begin{equation}\label{eq5.42}
			\begin{aligned}
				& \mathbb{E}\sup_{s \in [0,t]}[\mathcal{E}_{tot}(\bu_n(s),\bw_n(s),\bM_n(s),\bH_n(s))]^p + \mathbb{E} \left( \int_{0}^{t} |\nabla \bu_n(s)|^2 \d s\right)^p 
				\\
				&\quad +  \mathbb{E} \left(\int_{0}^{t} |\nabla  \bw_n(s)|^2 \d s\right)^p +  \mathbb{E} \left(\int_{0}^{t} |\curl\bM_n(s)|^2 \, \d s\right)^p +  \mathbb{E} \left(\int_0^{t} |\diver\bM_n(s)|^2 \d s\right)^p
				\\
				&\quad + \mathbb{E} \left(\int_0^{t} |\curl\bH_n(s)|^2 \, \d s\right)^p + \mathbb{E} \left(\int_{0}^{t}  |\bM_n(s)|^2 \, \d s\right)^p + \mathbb{E} \left(\int_0^{t} |\bH_n(s)|^2 \d s\right)^p
				\\
				&\quad  + \mathbb{E} \left(\int_{0}^{t} |\diver\bw_n(s)|^2 \d s\right)^p + \mathbb{E} \left(\int_{0}^{t} |\curl\bu_n(s) - 2\bw_n(s)|^2 \d s\right)^p
				\\
				&\leq C[1+ \mathcal{E}_{tot}(\bu_0,\bw_0,\bM_0,\bH_0)]^p
			\end{aligned}
		\end{equation}
and  
\begin{equation}\label{eqt5.43}
\mathbb{E} \left(\int_0^t \|\bM_n(s)\|_{V_1}^2 \, \d s \right)^p + \mathbb{E} \left(\int_0^t \|\bH_n(s)\|_{V_1}^2 \, \d s \right)^p \leq  C [ 1+ \mathcal{E}_{tot}(\bu_0,\bw_0,\bM_0,\bH_0)]^p.  
\end{equation}  
\end{lemma}
\begin{proof}
\dela{
Throughout this proof, $ C(p)>0$ will denote a positive constant independent of $n$ and $\lambda$, but it may depend on the problem data and may vary from one term to the other. \newline
The proof of \eqref{eqt5.43} follows directly from \eqref{eq5.42} and \eqref{eq2.1}. Now, let us proceed to the proof of \eqref{eq5.42}. Once we have proved Lemma \ref{lem2}, which is the analogue of Lemma \ref{lem6.4} for the case $p=1$, each term of  \eqref{eqt5.29} is well-defined, and therefore we do not need to use the stopping times anymore. 
}
Obviously, the proof of \eqref{eqt5.43} follows directly from \eqref{eq5.42} and \eqref{eq2.1}. 
\newline
Now, let us proceed to the proof of \eqref{eq5.42}.  First, as a direct consequence of \eqref{eqt5.29-c}, we deduce that each term of \eqref{eqt5.29} is well defined, and therefore we no longer need to use stopping times. Thus, in view of \eqref{eqt5.29}, \eqref{eq3.9} and \eqref{eq5.4}, we have for all $t \in[0,T]$, 
		\begin{align*}
			& \mathcal{E}_{tot}(\bu_n(t),\bw_n(t),\bM_n(t),\bH_n(t)) +  \int_{0}^{t} [(2 \nu - 2 + c_1) |\nabla \bu_n(s)|^2 +  (2\lambda_1 - 2 + c_2)|\nabla  \bw_n(s)|^2] \d s
			\\
			&  + \int_{0}^{t} [2\lambda \lvert \curl\bM_n(s) \rvert^2 + 2(\mu_0+1) \lambda \lvert \diver\bM_n(s) \rvert^2 + 2\sigma^{-1}  \lvert \curl\bH_n(s) \rvert^2 + 2\tau^{-1} \lvert \bM_n(s) \rvert^2] \d s
			\\
			&  +  \int_0^{t} [2\mu_0 \chi_0 \tau^{-1} |\bH_n(s)|^2  + 2(\lambda_1 + \lambda_2) |\diver\bw_n(s)|^2 + 2\alpha |\curl\bu_n(s) - 2\bw_n(s)|^2] \, \d s
			\\ 
			&\leq \mathcal{E}_{tot}(\bu_0,\bw_0,\bM_0,\bH_0) + \int_{0}^{t}  [2(\chi_0+\mu_0)\tau^{-1} |\bM_n(s)| |\bH_n(s)| + 2\mu_0 \lambda |\curl \bM_n(s)| |\curl\bH_n(s)|] \d s 
			\\
			&\quad + (\mu_0 + 1) \int_{0}^{t}  \|G(\bM_n(s))\|_{L_2(U_1,\mathbb{L}^2)}^2 \, \d s + \mu_0^{-1} \int_0^{t} \|F_3(\bH_n(s))\|_{L_2(U_2,\mathbb{L}^2)}^2 \d s
			\\ 
			&\quad   + 2 \int_{0}^{t}\left( \bu_n(s), F_1(\bu_n(s)) \d \beta^1(s)\right) + 2  \int_{0}^{t}\left( \bw_n(s), F_2(\bw_n(s)) \d \beta^2(s)\right)  
			\\
			&\quad - 2\mu_0 \int_{0}^{t}\left( \bH_n(s), G(\bM_n(s)) \d \beta^3(s)\right)  + 2 \int_{0}^{t}\left( \bH_n(s), F_3(\bH_n(s)) \d \beta^4(s)\right).
		\end{align*}
Next, we observe that it is enough to prove the Proposition for $p=4$, the case $2\leq p \leq 3$ follows from H\"older inequality and the estimates for $p= 4$. \newline
 Now, we take $\mathbb{E} \left[\sup_{s \in [0,t]} \lvert \cdot \rvert^4 \right]$. Hence,
\begin{align}\label{eq5.47}
& \mathbb{E}\sup_{s \in [0,t]}[\mathcal{E}_{tot}(\bu_n(s),\bw_n(s),\bM_n(s),\bH_n(s))]^4 \notag 
\\
& + (2 \nu - 2 + c_1)^4 \mathbb{E} \left(\int_{0}^{t} \lvert \nabla \bu_n \rvert^2 \d s\right)^4 + (2\lambda_1 - 2 + c_2)^4 \mathbb{E} \left(\int_{0}^{t} \lvert \nabla  \bw_n \rvert^2 \d s\right)^4 \notag 
  \\
& + 2^4 \lambda^4 \mathbb{E} \left(\int_{0}^{t} \lvert \curl\bM_n \rvert^2\, \d s\right)^4 + 2^4 \lambda^4 (\mu_0 +1)^4 \mathbb{E} \left(\int_0^{t} \lvert \diver\bM_n \rvert^2 \d s\right)^4 \notag 
    \\
& + \frac{2^4}{\sigma^4} \mathbb{E} \left(\int_0^{t} \lvert \curl\bH_n \rvert^2\, \d s\right)^4 + \frac{2^4}{\tau^4} \mathbb{E} \left(\int_{0}^{t} \lvert \bM_n \rvert^2\, \d s\right)^4 + \frac{2^4 \mu_0^4 \chi_0^4}{\tau^4} \mathbb{E} \left(\int_0^{t} \lvert \bH_n \rvert^2\, \d s\right)^4 \notag 
\\
& + 2^4 (\lambda_1 + \lambda_2)^4 \mathbb{E} \left(\int_{0}^{t} \lvert \diver\bw_n \rvert^2\, \d s\right)^4 + 2^4 \alpha^4 \mathbb{E} \left(\int_{0}^{t} \lvert \curl\bu_n - 2\bw_n \rvert^2\, \d s\right)^4            
   \\
& \leq 9^3 \bigg[ [\mathcal{E}_{tot}(\bu_0,\bw_0,\bM_0,\bH_0)]^4 + (\chi_0+\mu_0)^4 \tau^{-4} \mathbb{E} \left(\int_{0}^{t} \lvert \bM_n \rvert \lvert \bH_n \rvert\, \d s\right)^4 \notag 
     \\
&\qquad + \mu_0^4 \lambda^4 \mathbb{E} \left(\int_0^{t} \lvert \curl \bM_n \rvert \lvert \curl\bH_n \rvert\,\d s\right)^4 + (\mu_0 + 1)^4 \mathbb{E} \left(\int_{0}^{t}  \Vert G(\bM_n) \vert_{L_2(U_1,\mathbb{L}^2)}^2\, \d s \right)^4 \notag 
\\
&\qquad + \mu_0^{-4} \mathbb{E} \left(\int_0^{t} \Vert F_3(\bH_n) \Vert_{L_2(U_2,\mathbb{L}^2)}^2\, \d s\right)^4 + \mathbb{E} \sup_{s\in [0,t]} \left\lvert \int_{0}^{s}\left(\bu_n(s), F_1(\bu_n(s))\d \beta^1(s)\right)\right\rvert^4 \notag 
   \\
&\qquad + \mathbb{E} \sup_{s\in [0,t]} \left \lvert \int_{0}^{s}\left(\bw_n(s), F_2(\bw_n(s))\d \beta^2(s)\right) \right\rvert^4 + \mu_0^4 \mathbb{E} \sup_{s\in [0,t]} \left\lvert \int_{0}^{s}\left( \bH_n(s), G(\bM_n(s))\d \beta^3(s)\right)\right\rvert^4 \notag  
\\
&\qquad + \mathbb{E} \sup_{s\in [0,t]} \left \lvert \int_{0}^{s}\left( \bH_n(s), F_3(\bH_n(s))\d \beta^4(s)\right)\right \rvert^4\bigg]. \notag 
\end{align}
We proceed to estimate the terms on the RHS of \eqref{eq5.47}. First, by H\"older's and Young's inequalities, we infer that
\begin{equation}
\begin{aligned}
& 9^3 \mu_0^4 \lambda^4 \mathbb{E} \left(\int_0^{t} \lvert \curl\bM_n(s) \rvert \lvert \curl\bH_n(s) \rvert\, \d s \right)^4
\\
& \leq \frac{2^{3}}{\sigma^{4}} \mathbb{E} \left(\int_0^{t} \lvert \curl\bH_n(s) \rvert^2\, \d s \right)^4 + \frac{9^6 \mu_0^8 \lambda^8 \sigma^4}{32} \mathbb{E} \left(\int_0^{t} \lvert \curl \bM_n(s) \rvert^2\, \d s \right)^4
\end{aligned}
\end{equation}
and
\begin{equation}
\begin{aligned}
& 9^3 (\mu_0 + \chi_0)^4 \tau^{-4} \mathbb{E} \left(\int_0^{t} \lvert \bM_n(s) \rvert \lvert \bH_n(s) \rvert\, \d s \right)^4 
\\
& \leq 9^3 (\mu_0 + \chi_0)^4 \tau^{-4} \mathbb{E} \left[\left(\int_0^{t} \lvert \bM_n(s) \rvert^2\, \d s \right)^2 \left(\int_0^{t} \lvert \bH_n(s) \rvert^2\, \d s \right)^2\right] 
  \\
& \leq \frac{2^{3}}{\sigma^{4}} \mathbb{E} \left(\int_0^{t} \lvert \bM_n(s) \rvert^2\, \d s \right)^4 + \frac{9^6 (\mu_0 + \chi_0)^8 \tau^{-8}}{32} \mathbb{E} \left(\int_0^{t} \lvert \bH_n(s) \rvert^2\, \d s \right)^4 
     \\
& \leq \frac{2^{3}}{\sigma^{4}} \mathbb{E} \left(\int_0^{t} \lvert \bM_n(s) \rvert^2\, \d s \right)^4 + \frac{9^6 (\mu_0 + \chi_0)^8 \tau^{-8}}{32} t^{3} \mathbb{E} \int_0^{t} \lvert \bH_n(s) \rvert^{8}\, \d s. 
\end{aligned}
\end{equation}
Next, in view of \eqref{eq2.1} and \eqref{eq3.9}, we infer that
\begin{align*}
& 9^3 (\mu_0 + 1)^4 \mathbb{E} \left(\int_0^{t}  \Vert G(\bM_n(s)) \Vert_{L_2(U_1,\mathbb{L}^2)}^2\,\d s \right)^4 
\leq 9^3 (2-c_3)^4 (\mu_0 + 1)^4 \mathbb{E} \left(\int_0^{t} \lvert \nabla \bM_n(s) \rvert^2\,\d s \right)^4
\\
& \leq 9^3 (2-c_3)^4 C_0^4 (\mu_0 + 1)^4 \mathbb{E} \left(\int_0^{t} (\lvert \bM_n(s) \rvert^2 + \lvert \curl\bM_n(s) \rvert^2 + \lvert \diver \bM_n(s) \rvert^2)\,\d s \right)^4,
\end{align*} 
and in turn, applying the H\"older inequality, we deduce that  
\begin{equation}\label{eq5.50}
\begin{aligned}
& 9^3 (\mu_0 + 1)^4 \mathbb{E} \left(\int_0^{t}  \Vert G(\bM_n(s)) \Vert_{L_2(U_1,\mathbb{L}^2)}^2\,\d s \right)^4 
\leq 3^{9} (2-c_3)^4 C_0^4 (\mu_0 + 1)^4 \mathbb{E} \left[ t^{3} \int_0^{t} \lvert \bM_n(s) \rvert^{8}\, \d s \right.
\\
& \left. + \left(\int_{0}^{t} \lvert \curl\bM_n(s) \rvert^2\, \d s\right)^4 + \left(\int_{0}^{t} \lvert \diver\bM_n(s) \rvert^{2}\, \d s\right)^4 \right].
\end{aligned}
\end{equation} 
By making use of the Burholder-Davis-Gundy and Cauchy-Schwarz inequalities, we find that
		\begin{equation}\label{eq5.51}
			\begin{aligned}
				& 9^3 \mu_0^4 \mathbb{E} \sup_{s\in[0,t]} \left\lvert \int_0^s ( \bH_n(\tau), G(\bM_n(\tau))\d \beta^3(\tau))\right\rvert^4 
                    \leq 9^3 \mu_0^4 C(4) \mathbb{E} \left(\int_0^{t} \Vert G(\bM_n(s)) \Vert_{L_2(U_1,\mathbb{L}^2)}^2 \lvert \bH_n(s) \rvert^2\, \d s \right)^2 
                      \\
				&\leq 9^3 \mu_0^4 C(4) \mathbb{E} \left[\sup_{s\in [0,t]} \lvert \bH_n(s) \rvert^4 \left(\int_0^{t} \Vert G(\bM_n(s)) \Vert_{L_2(U_1,\mathbb{L}^2)}^2\, \d s \right)^2\right] 
                        \\
				&\leq 9^3 \mu_0^4 C(4) \mathbb{E} [\sup_{s\in [0,t]} \lvert \bH_n(s) \rvert^{8}]^\frac{1}{2} \left[\mathbb{E} \left(\int_0^{t} \Vert G(\bM_n(s)) \Vert_{L_2(U_1,\mathbb{L}^2)}^2\, \d s \right)^4\right]^\frac{1}{2} 
                          \\
				&\leq \frac{\mu_0^4}{12} \mathbb{E} \left[\sup_{s\in [0,t]} \lvert \bH_n(s) \rvert^{8}\right] + 3^{13} \mu_0^8 [C(4)]^2 \mathbb{E} \left(\int_0^{t}  \Vert G(\bM_n(s)) \Vert_{L_2(U_1,\mathbb{L}^2)}^2\, \d s \right)^4.
			\end{aligned}
		\end{equation}
In view of \eqref{eq5.50} and \eqref{eq5.51}, we infer that
\begin{equation}
\begin{aligned}
& 9^3 \mu_0^4 \mathbb{E} \sup_{s\in[0,t]} \left\lvert \int_0^s ( \bH_n(\tau), G(\bM_n(\tau))\d \beta^3(\tau))\right\rvert^4 
\\	
&\leq \frac{\mu_0^4}{12} \mathbb{E} \sup_{s\in [0,t]} \lvert \bH_n(s) \rvert^{8} + 3^{16} (2-c_3)^4 C_0^4 \mu_0^4 [C(4)]^2 \mathbb{E} \left[ t^{3} \int_0^{t} \lvert \bM_n(s) \rvert^{8}\, \d s \right.
\\
& \left. + \left(\int_{0}^{t} \lvert \curl\bM_n(s) \rvert^2\, \d s\right)^4 + \left(\int_{0}^{t} \lvert \diver\bM_n(s) \rvert^{2}\, \d s\right)^4 \right].
\end{aligned}
\end{equation}
Next, using the Burholder-Davis-Gundy inequality and the H\"older inequality, we infer that for every $n \in \mathbb{N}$,
\begin{align*}
9^3 \mathbb{E} \sup_{s\in[0,t]} \left\lvert \int_0^s (\bu_n, F_1(\bu_n)\d \beta^1(\tau))\right \rvert^4 
&\leq 9^3 C(4) \mathbb{E} \left(\int_0^{t} \Vert F_1(\bu_n(s)) \Vert_{L_2(U_1,H)}^2 \lvert \bu_n(s) \rvert^2\, \d s \right)^2
\\
& \leq 9^3 C(4) \left[\mathbb{E} \sup_{s \in [0,t]} \lvert \bu_n(s) \rvert^4 \left(\int_0^{t} \Vert F_1(\bu_n(s)) \Vert_{L_2(U_1,H)}^2\, \d s \right)^2 \right].
\end{align*}
On the other hand, due to the Young inequality and \eqref{eq3.9}, we further obtain
\begin{equation}
\begin{aligned}
& 9^3 \mathbb{E} \sup_{s\in[0,t]} \left\lvert \int_0^s (\bu_n, F_1(\bu_n)\d \beta^1(\tau))\right \rvert^4
\\
&\leq \frac16 \mathbb{E} \sup_{s \in [0,t]} \lvert \bu_n(s) \rvert^8 + \frac{3^{13} [C(4)]^2}{2} \mathbb{E} \left(\int_0^{t} \Vert F_1(\bu_n(s)) \Vert_{L_2(U_1,H)}^2\, \d s \right)^4 
\\
&\leq \frac16 \mathbb{E} \sup_{s \in [0,t]} \lvert \bu_n(s) \rvert^8 + \frac{3^{13} [C(4)]^2 (2-c_1)^4}{2} \mathbb{E} \left(\int_0^{t} \lvert \nabla \bu_n(s) \rvert^2\, \d s\right)^4.
\end{aligned}
\end{equation}
Analogously,
		\begin{equation}
			\begin{aligned}
				&  9^3 \mathbb{E} \sup_{s\in[0,t]} \left \lvert \int_0^s (\bw_n(\tau), F_2(\bw_n(\tau))\d \beta^2(\tau)) \right \rvert^4 
				\\
				&\leq \frac16 \mathbb{E} \sup_{s \in [0,t]} \lvert \bw_n(s) \rvert^{8} + \frac{3^{13} [C(4)]^2 (2-c_2)^4}{2} \mathbb{E} \left(\int_0^{t} \lvert \nabla \bw_n(s) \rvert^2\, \d s\right)^4.
			\end{aligned}
		\end{equation}
Next, arguing as in \eqref{eq5.50}, we deduce that
\begin{equation}
\begin{aligned}
& \mu_0^{-4} \mathbb{E} \left(\int_0^{t} \Vert F_3(\bH_n(s))\Vert_{L_2(U_2,\mathbb{L}^2)}^2\, \d s\right)^4 
\leq 3^{3} (2-c_4)^4 C_0^4 \mu_0^{-4} \mathbb{E} \left[ t^{3} \int_0^{t} \lvert \bH_n(s) \rvert^{8}\, \d s \right.
\\
& \left. + \left(\int_{0}^{t} \lvert \curl\bH_n(s) \rvert^2\, \d s\right)^4 + \left(\int_{0}^{t} \lvert \diver\bH_n(s) \rvert^{2}\, \d s\right)^4 \right].
\end{aligned}
\end{equation} 
Now, let us point out that 
\begin{align*}
(\bH_n, F_3(\bH_n)\d \beta^4)
&=  \int_{\mathcal{O}} \sum_{k=1}^{N} [(j_k(x) \cdot \nabla) \bH_n(x) \cdot \bH_n(x) + q_k(x) \bH_n(x) \cdot \bH_n(x)]\, \d x \, \d \beta_{k}^4,
\\
\int_{\mathcal{O}} j_k^{(i)} \frac{\partial \bH_j}{\partial x_i} \bH_j\, \d x 
&= \int_{\mathcal{O}} j_k^{(i)} \frac{\partial}{\partial x_i} \left(\frac{1}{2} |\bH^j|^2\right)\, \d x= - \frac{1}{2} \int_{\mathcal{O}} \diver j_k |\bH^j|^2 \, \d x=0,
\end{align*}
where for the last identity we used an integration by parts, the assumption (H2) for $j_k$ and the summation convention over repeated indices. Hence,
\begin{equation*}
\begin{aligned}
(\bH_n, F_3(\bH_n)\d \beta^4)
&= \sum_{k=1}^{N} ( \bH_n, q_k\bH_n \, \d \beta_{k}^4),
\end{aligned}
\end{equation*}
and in turn, applying the Burkholder-Davis-Gundy inequality and using \eqref{eq3.3}$_3$, we infer that for every $n \in \mathbb{N}$,
 \begin{align*}
& 9^3 \mathbb{E} \sup_{s \in [0,t]}\left \lvert \int_0^s (\bH_n(\tau), F_3(\bH_n(\tau))\d \beta^4(\tau)) \right \rvert^4 
\leq 9^3 C(4) \mathbb{E} \left[\int_0^{t} \left(\sum_{k=1}^{N} q_k \bH_n(s), \bH_n(s)\right)^2\, \d s \right]^2
\\
&\leq 9^3 C(4) \mathbb{E} \left[\int_0^{t} \left(\sum_{k=1}^{N} \Vert q_k \Vert_{L^\infty(\mathcal{O})} \lvert \bH_n \rvert^2 \right)^2\, \d s \right]^2
\leq 9^3 C(4) (NC_7)^2 \mathbb{E}  \left(\int_0^{t} \lvert \bH_n \rvert^4\, \d s\right)^2 
\\
&\leq 9^3 C(4) (NC_7)^2 \mathbb{E} \left[\sup_{s\in[0,t]} \lvert \bH_n \rvert^4 \left(\int_0^{t} \lvert \bH_n \rvert^2\, \d s\right)^2 \right] 
  \\
&\leq \frac{\mu_0^4}{12} \mathbb{E} \sup_{s\in[0,t]} \lvert \bH_n(s) \rvert^{8} + \frac{3^{13} [C(4)]^2 (NC_7)^4}{\mu_0^4} t^{3} \mathbb{E} \int_0^{t} \lvert \bH_n \rvert^{8}\, \d s.
\end{align*}
Plugging the previous estimates into the RHS of \eqref{eq5.47}, we deduce that for every $n \in \mathbb{N}$,
\begin{equation*}
\begin{aligned}
& \frac12 \mathbb{E}\sup_{s \in [0,t]}[\mathcal{E}_{tot}(\bu_n(s),\bw_n(s),\bM_n(s),\bH_n(s))]^4 
\\
& + \left[(2 \nu - 2 + c_1)^4 - \frac{3^{13} [C(4)]^2}{2} (2-c_1)^4  \right] \mathbb{E} \left(\int_{0}^{t} \lvert \nabla \bu_n \rvert^2 \d s\right)^4 
  \\
& + \left[(2\lambda_1 - 2 + c_2)^4 - \frac{3^{13} [C(4)]^2}{2} (2-c_2)^4 \right] \mathbb{E} \left(\int_{0}^{t} \lvert \nabla  \bw_n \rvert^2 \d s\right)^4 
    \\
& + 2^4 \lambda^4  \mathbb{E} \left(\int_{0}^{t} \lvert \curl\bM_n \rvert^2\, \d s\right)^4 + 2^4 \lambda^4 (\mu_0 +1)^4 \mathbb{E} \left(\int_0^{t} \lvert \diver\bM_n \rvert^2 \d s\right)^4 
\\
& + \frac{2^4}{\sigma^4} \mathbb{E} \left(\int_0^{t} \lvert \curl\bH_n \rvert^2\, \d s\right)^4 + \frac{2^4}{\tau^4} \mathbb{E} \left(\int_{0}^{t} \lvert \bM_n \rvert^2\, \d s\right)^4
+ \frac{2^4 \mu_0^4 \chi_0^4}{\tau^4} \mathbb{E} \left(\int_0^{t} \lvert \bH_n \rvert^2\, \d s\right)^4 
\\
& + 2^4 (\lambda_1 + \lambda_2)^4 \mathbb{E} \left(\int_{0}^{t} \lvert \diver\bw_n \rvert^2\, \d s\right)^4 + 2^4 \alpha^4 \mathbb{E} \left(\int_{0}^{t} \lvert \curl\bu_n - 2\bw_n \rvert^2\, \d s\right)^4            
   \\
& \leq 9^3 [\mathcal{E}_{tot}(\bu_0,\bw_0,\bM_0,\bH_0)]^4 + \left[\frac{2^{3}}{\sigma^{4}} + 3^{3} (2-c_4)^4 C_0^4 \mu_0^{-4} \right] \mathbb{E} \left(\int_0^{t} \lvert \curl\bH_n(s) \rvert^2\, \d s \right)^4  
\\
& + \left[\frac{9^6 \mu_0^8 \lambda^8\sigma^4}{32} + 3^{9} (2-c_3)^4 C_0^4 (\mu_0 + 1)^4 +  3^{16} (2-c_3)^4 C_0^4 \mu_0^4 [C(4)]^2\right] \mathbb{E} \left(\int_0^{t} \lvert \curl \bM_n(s) \rvert^2\, \d s \right)^4 
\\
& + \left[\frac{2^{3}}{\sigma^{4}} + 3^{9} (2-c_3)^4 C_0^4 (\mu_0 + 1)^4 + 3^{16} (2-c_3)^4 C_0^4 \mu_0^4 [C(4)]^2 \right] t^{3} \mathbb{E}  \int_0^{t} \lvert \bM_n(s) \rvert^{8}\, \d s  
\\
& + \left[3^{9} (2-c_3)^4 C_0^4 (\mu_0 + 1)^4 + 3^{3} (2-c_4)^4 C_0^4 \mu_0^{-4} + 3^{16} (2-c_3)^4 C_0^4 \mu_0^4 [C(4)]^2 \right] \mathbb{E} \left(\int_{0}^{t} \lvert \diver\bM_n(s) \rvert^{2}\, \d s\right)^4 
\end{aligned}
\end{equation*}
\begin{equation*}
\begin{aligned}
& + \left[ 3^{3} (2-c_4)^4 C_0^4 \mu_0^{-4} + \frac{3^{13} [C(4)]^2 (NC_7)^4}{\mu_0^4} + \frac{9^6 (\mu_0 + \chi_0)^8 \tau^{-8}}{32} \right] t^{3} \mathbb{E} \int_0^{t} \lvert \bH_n \rvert^{8}\, \d s.
\end{aligned}
\end{equation*}
From this previous inequality, we infer that for every $n \in \mathbb{N}$ and for all $t \in [0,T]$,
\begin{equation}\label{eq5.58-a}
\begin{aligned}
& \frac12 \mathbb{E}\sup_{s \in [0,t]}[\mathcal{E}_{tot}(\bu_n(s),\bw_n(s),\bM_n(s),\bH_n(s))]^4 
\\
& + \left[(2 \nu - 2 + c_1)^4 - 2\cdot 3^{16} [C(4)]^2 (2-c_1)^4  \right] \mathbb{E} \left(\int_{0}^{t} \lvert \nabla \bu_n \rvert^2 \d s\right)^4 
  \\
& + \left[(2\lambda_1 - 2 + c_2)^4 - 2\cdot 3^{16} [C(4)]^2 (2-c_2)^4 \right] \mathbb{E} \left(\int_{0}^{t} \lvert \nabla  \bw_n \rvert^2 \d s\right)^4 
    \\
& + \left[2^4 \lambda^4 - 2\cdot 3^{16} [C(4)]^2 \left(\mu_0^2 \lambda^2\sigma + (2-c_3) C_0 (\mu_0+1)\right)^4 \right] \mathbb{E} \left(\int_{0}^{t} \lvert \curl\bM_n \rvert^2\, \d s\right)^4 
\\
& + \left[2^4 \lambda^4 (\mu_0 +1)^4 - 2 \cdot 3^{16} [C(4)]^2 \left((2-c_3) C_0 (\mu_0 + 1) + (2-c_4) C_0 \mu_0^{-1} \right)^4 \right] \mathbb{E} \left(\int_0^{t} \lvert \diver\bM_n \rvert^2 \d s\right)^4 
\\
& + \frac{1}{2} \left[\frac{2^4}{\sigma^4} - 2 \cdot 3^{16} [C(4)]^2 (2-c_4)^4 C_0^4 \mu_0^{-4} \right] \mathbb{E} \left(\int_0^{t} \lvert \curl\bH_n \rvert^2\, \d s\right)^4 
\\
& + \frac{2^4}{\tau^4} \mathbb{E} \left(\int_{0}^{t} \lvert \bM_n \rvert^2\, \d s\right)^4
+ \frac{2^4 \mu_0^4 \chi_0^4}{\tau^4} \mathbb{E} \left(\int_0^{t} \lvert \bH_n \rvert^2\, \d s\right)^4 
\\
& + 2^4 (\lambda_1 + \lambda_2)^4 \mathbb{E} \left(\int_{0}^{t} \lvert \diver\bw_n \rvert^2\, \d s\right)^4 + 2^4 \alpha^4 \mathbb{E} \left(\int_{0}^{t} \lvert \curl\bu_n - 2\bw_n \rvert^2\, \d s\right)^4            
   \\
& \leq 9^3 [\mathcal{E}_{tot}(\bu_0,\bw_0,\bM_0,\bH_0)]^4  + 2 \kappa(t,4), \; \; t \in [0,T],
\end{aligned}
\end{equation}
having set
\begin{equation*}
\begin{aligned}
\kappa(t,4) \coloneq
&\left[\frac{2^{3}}{\sigma^{4}} +  3^{9} (2-c_3)^4 C_0^4 (\mu_0 + 1)^4 + 3^{16} (2-c_3)^4 C_0^4 \mu_0^4 [C(4)]^2 \right] t^{3} \\
& + \left[ 3^{3} (2-c_4)^4 C_0^4 \mu_0^{-4} + \frac{3^{13} [C(4)]^2 (NC_7)^4}{\mu_0^4} + \frac{9^6 (\mu_0 + \chi_0)^8 \tau^{-8}}{32} \right] \mu_0^{-4} t^{3}, \; \; t \in [0,T].
\end{aligned}
\end{equation*}
Notice that in view of the assumptions of Theorem \ref{theo 1}, we hae
\begin{align*}
(2 \nu - 2 + c_1)^4 - 2\cdot 3^{16} [C(4)]^2 (2-c_1)^4&>0,  
\\
(2\lambda_1 - 2 + c_2)^4 - 2\cdot 3^{16} [C(4)]^2 (2-c_2)^4&>0,  
\\
2^4 \lambda^4 - 2\cdot 3^{16} [C(4)]^2 (\mu_0^2 \lambda^2\sigma + (2-c_3) C_0 (\mu_0+1))^4&>0,  
  \\
2^4 \lambda^4 (\mu_0 +1)^4 - 2 \cdot 3^{16} [C(4)]^2 ((2-c_3) C_0 (\mu_0 + 1) + (2-c_4) C_0 \mu_0^{-1})^4&>0, 
    \\
\frac{2^4}{\sigma^4} - 2 \cdot 3^{16} [C(4)]^2 (2-c_4)^4 C_0^4 \mu_0^{-4}&>0.  
\end{align*}
Therefore, \eqref{eq5.42} follows directly from \eqref{eq5.58-a} through an application of the Gronwall lemma.
\end{proof}
\subsection*{Step 3: Translation estimates, tightness property of Galerkin solutions and Construction of new probability space}
Since we  are not able to estimate time derivatives of $\bu_n,\, \bw_n,\, \bM_n,\, \bH_n$ and $\bB_n$, and hence to use the stochastic version of the Aubin-Lions compactness theorem, we will estimate translation differences of $\bu_n,\, \bw_n,\, \bM_n,\, \bH_n$ and $\bB_n$, namely $\bu_n(t+\theta) -\bu_n(t)$, $\bw_n(t+\theta) -\bw_n(t)$, $\bM_n(t+\theta) -\bM_n(t)$, $\bH_n(t+\theta) -\bH_n(t)$ and $\bB_n(t+\theta) -\bB_n(t)$. This type of estimates has been used in \cite{Aris,Motyl1,Motyl}. For this purpose, we  state and prove the following lemma and its corollary.
\begin{lemma}\label{lem5}	
Let the assumptions of Lemma \ref{eqt5.29-c} be satisfied. Then, there exists a positive constant $C_9$ independent of $n$ such that
			\begin{equation}\label{eq5.59}
				\begin{aligned}
					& \mathbb{E} \int_0^t (\|A\bu_n(s)\|_{V'}^2 + \|\mathcal{P}_n^1 B_0(\bu_n(s),\bu_n(s))\|_{V'}^{4/3} + \|\mathcal{P}_n^1 M_0(\bM_n(s),\bH_n(s))\|_{V'}^{8/7}) \, \d s \leq C_9, \\
					& \mathbb{E} \int_0^t (\|\mathcal{P}_n^1 R_1(\bH_n(s), \bH_n(s))\|^{8/7}_{V'} +  \|\mathcal{P}_n^1 R_0(\bu_n(s), \bw_n(s))\|_{V'}^2) \, \d s \leq C_9, \\
					& \mathbb{E} \int_0^t (\|A_1 \bw_n \|_{\mathbb{H}^{-1}(\mathcal{O})}^2 + \|\mathcal{P}_n^2 B_1(\bu_n,\bw_n) \|_{\mathbb{H}^{-1}}^{4/3} +\|\mathcal{P}_n^2 R_5(\bw_n)\|_{\mathbb{H}^{-1}}^2) \, \d s \leq C_9, \\
					& \mathbb{E} \int_0^t (\|\mathcal{P}_n^2 R_2(\bu_n(s),\bw_n(s))\|_{\mathbb{H}^{-1}}^2 + \|\mathcal{P}_n^2 R_3(\bM_n(s),\bH_n(s))\|_{\mathbb{H}^{-1}}^{4/3}) \, \d s \leq C_9,\\
					& \mathbb{E} \int_0^t (\|\mathcal{P}_n^3  R_6(\bM_n(s))\|_{V'_1}^2 + \|\mathcal{P}_n^3  R_5(\bM_n(s))  \|_{V'_1}^2)  \, \d s \leq C_9, \\
					&\mathbb{E} \int_0^t  (\|\mathcal{P}_n^3 B_2(\bu_n(s),\bM_n(s))\|^{4/3}_{V'_1} + \|\mathcal{P}_n^3 R_3(\bw_n(s),\bM_n(s))\|_{V'_1}^2) \, \d s, \\
					& \mathbb{E} \int_0^t (\|\mathcal{P}_n^3  R_6(\bH_n(s))\|_{V'_2}^2 + \|\mathcal{P}_n^3 \tilde{M}_2(\bu_n(s),\bB_n(s))\|^{4/3}_{V'_2}) \, \d s \leq C_9, \;\; t \in [0,T].
				\end{aligned}
		\end{equation} 
	\end{lemma}
\begin{proof}
Observe that there exists a positive constant $C>0$ such that
			\begin{align*}
				\|\bB_n\|_{V_2}^2 
				&\leq C [\bB_n,\bB_n]= C(|\bB_n|^2 + |\curl\bB_n|^2 + |\diver\bB_n|^2)= C(|\bB_n|^2 + |\curl\bB_n|^2) \\
				&\leq 2C\mu_0^2 (|\bM_n|^2 + |\bH_n|^2) + 2C\mu_0^2(|\curl\bM_n|^2 + |\curl\bH_n|^2),
			\end{align*}
where we used Cauchy-Schwarz's and Young's inequalities and the fact that $\diver\bB_n = 0$ in $Q_T$. Thus, in light of \eqref{eq2.9} and \eqref{eq5.42}, there exists $c_8>0$ such that for all $n\in \mathbb{N}$, 
\begin{align}\label{eq5.64}
& \mathbb{E} \int_0^t \|\mathcal{P}_n^3 \tilde{M}_2(\bu_n(s),\bB_n(s))\|_{V'_2}^\frac{4}{3} \, \d s \notag 
\leq  C_2^\frac{4}{7} \mathbb{E} \int_0^t |\bB_n(s)|^\frac{1}{3} \|\bB_n(s)\|_{V_2} \|\bu_n(s)\|^\frac{1}{3} |\nabla \bu_n(s)| \, \d s \notag 
\\
&\leq C_2^\frac{4}{7} \mathbb{E} \left[\sup_{ s\in [0,t]} |\bB_n(s)|^\frac{1}{3} \sup_{s\in [0,t]} |\bu_n(s)|^\frac{1}{3} \left(\int_0^t \|\bB_n(s)\|_{V_2}^2 \, \d s \right)^\frac{1}{2} \left(\int_0^t |\nabla \bu_n(s)|^2 \, \d s\right)^\frac{1}{2} \right] 
\\
&\leq C_2^\frac{4}{7} [\mathbb{E} \sup_{s\in [0,t]} |\bB_n(s)|^\frac{4}{3}]^\frac{1}{4} [\mathbb{E} \sup_{ s\in [0,t]} |\bu_n(s)|^\frac{4}{3}]^\frac{1}{4} \left[\mathbb{E} \left(\int_0^t \|\bB_n(s)\|_{V_2}^2 \d s \right)^2 \right]^\frac{1}{4} \left[\mathbb{E} \left(\int_0^t |\nabla \bu_n(s)|^2 \d s \right)^2 \right]^\frac{1}{4} \notag   
\\
& \leq c_8. \notag 
\end{align}
Thanks to \eqref{eq2.6a}, \eqref{eq2.11}, \eqref{Eq:Ineq-R3}, and Lemma \ref{lem6.4}, along with an argument similar to the proof of \eqref{eq5.64}, we show that there exists $c_8>0$ such that for all $n\in \mathbb{N}$ 
				\begin{align*}
					& \mathbb{E} \int_0^t \biggl(\|\mathcal{P}_n^1 B_0(\bu_n(s),\bu_n(s))\|_{V'}^{\frac43} +   \|\mathcal{P}_n^2 B_1(\bu_n(s),\bw_n(s))\|_{\mathbb{H}^{-1}}^{\frac43}+ \|\mathcal{P}_n^3 B_2(\bu_n(s),\bM_n(s))\|_{V_1'}^{\frac43}\biggr) \d s \le c_8,\\
					&\mathbb{E} \int_0^t \biggl(\|\mathcal{P}_n^3 R_3(\bu_n(s),\bw_n(s))\|_{V_1'}^{\frac43}+   \|\mathcal{P}_n^2 R_3(\bM_n(s),\bH_n(s))\|_{\mathbb{H}^{-1}}^\frac43\biggr)\d s \leq c_8. 
			\end{align*} 
From \eqref{eq2.3}, \eqref{eq5.42}, and \eqref{eqt5.43}, along with the embedding of $\mathbb{H}^1$ in $\mathbb{L}^4$, we conclude that there exists $c_8>0$ such that for all $n\in \mathbb{N}$,
			\begin{align*}
				&\mathbb{E} \int_0^t \|\mathcal{P}_n^1 M_0(\bM_n(s),\bH_n(s))\|_{V'}^\frac{8}{7}\, \d s 
				\leq C \mathbb{E} \int_0^t \|\bM_n(s)\|^\frac{2}{7} \|\bM_n(s)\|_{V_1}^\frac{6}{7} \|\bH_n\|_{V_1}^\frac{8}{7} \, \d s
                      \\
				&\leq C \mathbb{E} \left[\sup_{s \in [0,t]} |\bM_n(s)|^\frac{2}{7} \left(\int_0^t \|\bM_n(s)\|_{V_1}^2 \, \d s \right)^\frac{3}{7} \left(\int_0^t \|\bH_n(s)\|_{V_1}^2 \, \d s \right)^\frac{4}{7} \right] \\
				& \leq C [\mathbb{E} \sup_{s \in [0,t]} |\bM_n(s)|^2]^\frac{1}{7} \left(\mathbb{E} \int_0^t \|\bM_n(s)\|_{V_1}^2 \, \d s \right)^\frac{3}{7} \left[\mathbb{E} \left(\int_0^t \|\bH_n(s)\|_{V_1}^2 \, \d s \right)^\frac{4}{3} \right]^\frac{3}{7}  \\
				& \leq C [\mathbb{E} \sup_{s \in [0,t]} |\bM_n(s)|^2]^\frac{1}{7} \left(\mathbb{E} \int_0^t \|\bM_n(s)\|_{V_1}^2 \, \d s \right)^\frac{3}{7} \left[\mathbb{E} \left(\int_0^t \|\bH_n(s)\|_{V_1}^2 \, \d s \right)^2 \right]^\frac{2}{7} \leq c_8.
			\end{align*}
In a similar manner, we have
			\begin{equation*}
				\begin{aligned}
					&\mathbb{E} \int_0^t \|\mathcal{P}_n^1 R_1(\bH_n(s), \bH_n(s))\|_{V'}^\frac{8}{7}\, \d s
					\leq C \mathbb{E} \int_0^t |\bH_n(s)|^\frac{2}{7} \|\bH_n(s)\|_{V_1}^\frac{6}{7} |\curl\bH_n(s)|^\frac{8}{7}\, \d s \\
					& \leq C [\mathbb{E} \sup_{s \in [0,t]} |\bH_n(s)|^2]^\frac{1}{7} \left(\mathbb{E} \int_0^t \|\bH_n(s)\|_{V_1}^2\, \d s \right)^\frac{3}{7} \left[\mathbb{E} \left(\int_0^t |\curl\bH_n(s)|^2 \, \d s \right)^2 \right]^\frac{2}{7} \leq c_8.
				\end{aligned}
			\end{equation*}
By Lemma \ref{lem:Bilinear-R0}, \eqref{Eq:Ineq-R2}, \eqref{Eq:Ineq-R5}, and \eqref{Eq:Ineq-R6}, the boundedness of the linear maps $A$ and $A_1$, and Lemma \ref{lem6.4}, we can easily show that there exists $c_8 > 0$ such that for all $n \in \mathbb{N}$
			\begin{align*}
				&\mathbb{E} \int_0^t \biggl(\lVert A\bu_n\rVert^2_{V^\prime}+\|\mathcal{P}_n^1 R_0(\bu_n, \bw_n)\|_{V'}^2+ \lVert A\bw_n\rVert^2_{\mathbb{H}^{-1}(\mathcal{O})}+  \|\mathcal{P}_n^2 R_5(\bw_n)\|_{\mathbb{H}^{-1}}^2\biggr) \d s \leq c_8, \\
				&\mathbb{E} \int_0^t \biggl(\|\mathcal{P}_n^2 R_2(\bu_n,\bw_n)\|_{\mathbb{H}^{-1}}^2+\|\mathcal{P}_n^3  R_6(\bH_n)\|_{V'_2}^2 + \|\mathcal{P}_n^3  R_6(\bM_n)\|_{V'_1}^2 + \|\mathcal{P}_n^3  R_5(\bM_n)\|_{V'_1}^2 \biggl) \d s \leq c_8.
				\end{align*} 
By collecting all the above estimates, we can easily complete the proof of the lemma. 
\end{proof}
\begin{cor}
There exists a constant $C>0$ such that for every $(\tau_{n})_{n\in\mathbb{N}}$ of $\mathbb{F}$-stoppings times with $\tau_{n}\leq T$, and for every $n\in\mathbb{N}$ and $\theta\geq 0$, we have  
				\begin{equation}\label{Eq:Aldous-un}
					\begin{aligned}
						&\mathbb{E} \|\bu_{n}(\tau_{n}+\theta)-\bu_{n}(\tau_{n}) \|_{V^\prime}^{\frac87}\leq C(\theta^{\frac{4}{7}} + \theta^{\frac{2}{7}} +  \theta^{\frac17}),
						\\
						&\mathbb{E} \|\bw_{n}(\tau_{n}+\theta)-\bw_{n}(\tau_{n}) \|_{\mathbb{H}^{-1}(\mathcal{O})}^{\frac43}\leq C (\theta^{\frac23} + \theta^{\frac13}),
						\\
						&\mathbb{E} \|\bM_{n}(\tau_{n}+\theta)-\bM_{n}(\tau_{n}) \|_{V_2^\prime}^{\frac43}\leq C (\theta^{\frac23} + \theta^{\frac13}),
						\\
						&\mathbb{E} \|\bH_{n}(\tau_{n}+\theta)-\bH_{n}(\tau_{n}) \|_{V_1^\prime}^{\frac43}\leq C (\theta^{\frac23} + \theta^{\frac13}),
						\\
						&\mathbb{E} \|\bB_{n}(\tau_{n}+\theta)-\bB_{n}(\tau_{n}) \|_{V_1^\prime}^{\frac43}\leq C (\theta^{\frac23} + \theta^{\frac13}).
					\end{aligned}
					\end{equation}  
\end{cor}
\begin{proof}
Note that there exists a constant $C>0$ such that for all $n\in \mathbb{N}$
				\begin{equation}\label{eq6.57}
					\begin{aligned}
						\lVert \bu_n(\tau_{n} + \theta) - \bu_n(\tau_{n}) \rVert_{V^\prime}^{\frac87} \le 
						C \biggl\lVert  \int_{\tau_{n}}^{\tau_{n} + \theta}\biggl( - \nu A \bu_n(s)-\mathcal{P}_n^1 B_0(\bu_n(s),\bu_n(s))-\alpha \mathcal{P}_n^1 R_0(\bu_n(s), \bw_n(s))  \,\biggr) \d s  \\
						\quad  + \mu_0  \int_{\tau_{n}}^{\tau_{n} + \theta}   [\mathcal{P}_n^1 M_0(\bM_n(s),\bH_n(s))  
						+ \mathcal{P}_n^1 R_1(\bH_n(s), \bH_n(s))] \, \d s \biggr\lVert_{V^\prime}^{\frac87}\\
						+C  \biggl\lVert \int_{\tau_{n}}^{\tau_{n} + \theta} \mathcal{P}_n^1 F_1(\bu_n(s)) \, \d \beta^1(s)\biggr\rVert^{\frac87}.
					\end{aligned}
				\end{equation}
By H\"older's inequality, the It\^o isometry, \eqref{Eq3.10}$_1$, and \eqref{eq5.42}, we see that there exists a constant $C>0$ such that for all $n\in \mathbb{N}$, we have 
				\begin{align*}
					&\mathbb{E} \left\|\int_{\tau_n}^{\tau_n + \theta} \mathcal{P}_n^1 F_1(\bu_n(s)) \, \d \beta^1(s)\right\|_{V'}^{\frac87}
					\leq \biggl[	\mathbb{E} \left\|\int_{\tau_n}^{\tau_n + \theta} \mathcal{P}_n^1 F_1(\bu_n(s)) \, \d \beta^1(s)\right\|_{V'}^2\biggr]^{\frac47} \\
					& \leq \biggl[\mathbb{E} \int_{\tau_n}^{\tau_n + \theta} \|\mathcal{P}_n^1 F_1(\bu_n(s))\|_{L_2(U_1,V')}^2 \, \d s \biggr]^{\frac47}
                    \leq \biggl[4C_1\mathbb{E} \int_{\tau_n}^{\tau_n + \theta} |\bu_n(s)|^2  \, \d s\biggr]^{\frac47} 
                    \\
					& \leq \biggl[ 4 C_1 \mathbb{E} \sup_{ s\in [0,T]} |\bu_n(s)|^2 \theta \biggr]^{\frac47} \leq C \theta^{\frac47}.
				\end{align*} 
Owing to H\"older's inequality and Lemma \ref{lem5}, we easily prove that there exists a constant $C>0$ such that for all $n\in \mathbb{N}$ 
	\begin{align*}
		&\mathbb{E} \biggl\lVert  \int_{\tau_{n}}^{\tau_{n} + \theta}\biggl( - \nu A \bu_n(s)-\mathcal{P}_n^1 B_0(\bu_n(s),\bu_n(s))-\alpha \mathcal{P}_n^1 R_0(\bu_n(s), \bw_n(s))  \,\biggr) \d s  \\
		&\qquad + \mu_0  \int_{\tau_{n}}^{\tau_{n} + \theta}   [\mathcal{P}_n^1 M_0(\bM_n(s),\bH_n(s))  
		+ \mathcal{P}_n^1 R_1(\bH_n(s), \bH_n(s))] \, \d s \biggr\lVert_{V^\prime}^{\frac87}\le C (\theta^{\frac{2}{7}} + \theta^{\frac17}).
	\end{align*}
From the above inequalities, we easily complete the proof of \eqref{Eq:Aldous-un}. Similar reasoning can be used to prove the remaining inequalities, so we leave their proofs as an exercise for the reader. Thus, we have completed the proof of the corollary.  
\end{proof}
\begin{remark}
 By the Chebyshev inequality, the result \eqref{Eq:Aldous-un} provides a sufficient condition that the sequence $(\bu_n,\bw_n,\bM_n,\bH_n,\bB_n)_n$ satisfies the Aldous condition (cf. \cite[Appendix A, Lemma 6.3]{Motyl}). Moreover, since the sequence $(\bu_n,\bw_n,\bM_n,\bH_n,\bB_n)_n$ satisfies the Aldous condition in the space $V^\prime \times \mathbb{H}_0^{-1} \times V_1^\prime \times V_1^\prime \times V_2^\prime$, and since they are uniformly bounded in the spaces
\begin{align*}
\tilde{\mathcal{Z}}_{\bu}&:= L_w^2(0,T;V) \cap L^2(0,T;H) \cap \mathcal{C}([0,T];V') \cap \mathcal{C}_w([0,T];H), 
\\
\tilde{\mathcal{Z}}_{\bw}&:= L_w^2(0,T;\mathbb{H}_0^1) \cap L^2(0,T;\mathbb{L}^2) \cap \mathcal{C}([0,T];\mathbb{H}^{-1}(\mathcal{O})) \cap \mathcal{C}_w([0,T];\mathbb{L}^2),  
    \\
\tilde{\mathcal{Z}}_{\bM}&:= L_w^2(0,T;V_1) \cap L^2(0,T;\mathbb{L}^2) \cap \mathcal{C}([0,T];V'_1) \cap \mathcal{C}_w([0,T];\mathbb{L}^2), 
       \\
\tilde{\mathcal{Z}}_{\bB}&:= L_w^2(0,T;V_2) \cap L^2(0,T;\mathbb{L}^2) \cap \mathcal{C}([0,T];V'_2) \cap \mathcal{C}_w([0,T];\mathbb{L}^2), 
           \\
\tilde{\mathcal{Z}}_{\bH} &\:=L_w^2(0,T;V_1) \cap L^2(0,T;\mathbb{L}^2) \cap \mathcal{C}([0,T];V'_1) \cap \mathcal{C}_w([0,T];\mathbb{L}^2),
\end{align*}
respectively, then by \cite[Lemma 3.1]{Motyl1}, $(\bw_n)_n,\, (\bM_n)_n,\, (\bH_n)_n$, and $(\bB_n)_n$ are relatively compact in the spaces $\tilde{\mathcal{Z}}_{\bu}$, $\tilde{\mathcal{Z}}_{\bw}$, $\tilde{\mathcal{Z}}_{\bM}$, $\tilde{\mathcal{Z}}_{\bB}$, and $\tilde{\mathcal{Z}}_{\bH}$, respectively. Consequently, the laws of $(\bw_n)_n,\, (\bM_n)_n,\, (\bH_n)_n$, and $(\bB_n)_n$ form  tight sequences on the spaces $\tilde{\mathcal{Z}}_{\bu}$, $\tilde{\mathcal{Z}}_{\bw}$, $\tilde{\mathcal{Z}}_{\bM}$, $\tilde{\mathcal{Z}}_{\bB}$, and $\tilde{\mathcal{Z}}_{\bH}$, respectively, with each space endowed with the supremum  topology.
\end{remark}

Now, we consider the constant sequence of cylindrical Wiener processes
              \begin{equation*}
             	\beta^{1,n} \equiv \beta^1,\quad  \beta^{2,n} \equiv \beta^2, \quad  \beta^{3,n} \equiv \beta^3, \quad  \beta^{4,n} \equiv \beta^4.
             \end{equation*}	
\begin{lemma}
The family of laws of $(\beta^{1,n})_{n\in\mathbb{N}},\, (\beta^{2,n})_{n\in\mathbb{N}},\, (\beta^{3,n})_{n\in\mathbb{N}},\, (\beta^{4,n})_{n\in\mathbb{N}}$ are tight in $\mathcal{C}([0,T];U_1), \, \mathcal{C}([0,T];U_1),\, \mathcal{C}([0,T];U_1),\, \mathcal{C}([0,T];U_2)$, respectively.	
\end{lemma}	
\begin{proof}
It directly follows from the fact that every measure on a complete separable metric space is tight.	
\end{proof}	
Finally, we can say that the family of laws of $(\bu_n,\bw_n,\bM_n,\bB_n,\bH_n, \beta^{1,n},\beta^{2,n},\beta^{3,n},\beta^{4,n})_{n\in\mathbb{N}}$ is tight in the cartesian product space
     \begin{equation*}
     	\mathcal{Z} = \tilde{\mathcal{Z}}_{\bu} \times \tilde{\mathcal{Z}}_{\bw} \times \tilde{\mathcal{Z}}_{\bM} \times \tilde{\mathcal{Z}}_{\bB} \times \tilde{\mathcal{Z}}_{\bH} \times \mathcal{C}([0,T];U_1) \times \mathcal{C}([0,T];U_1) \times \mathcal{C}([0,T];U_1) \times \mathcal{C}([0,T];U_2).
     \end{equation*}
Hence, an application of the Jakubowski Skorokhod embedding theorem, see \cite{Jakubowski_1997}, entails the following results. 
\begin{proposition}\label{Prop:Application-Prokhorov-Skorokhod}
There exists a complete probability space $$(\bar{\Omega},\bar{\mathcal{F}},\bar{\mathbb{P}})$$ and a subsequences of random vectors $(\bar{\bu}_{n_k},\bar{\bw}_{n_k},\bar{\bM}_{n_k},\bar{\bB}_{n_k}, \bar{\bH}_{n_k}, \bar{\beta}^{1,n_k}, \bar{\beta}^{2,n_k},\bar{\beta}^{3,n_k},\bar{\beta}^{4,n_k})_{k\in\mathbb{N}}$ with values in $\mathcal{Z}$ such that
\begin{itemize}
\item[(1)]
$(\bar{\bu}_{n_k},\bar{\bw}_{n_k},\bar{\bM}_{n_k},\bar{\bB}_{n_k},\bar{\bH}_{n_k},\bar{\beta}^{1,n_k}, 
\bar{\beta}^{2,n_k},\bar{\beta}^{3,n_k},\bar{\beta}^{4,n_k})_{k\in\mathbb{N}}$ has the same probability distribution as $(\bu_{n_k},\bw_{n_k},\bM_{n_k},\bB_{n_k}, \bH_{n_k}^{\lambda}, \beta^{1,n_k}, \beta^{2,n_k},\beta^{3,n_k},\beta^{4,n_k})_{k\in\mathbb{N}}$,
\item[(2)]
$(\bar{\bu}_{n_k},\bar{\bw}_{n_k},\bar{\bM}_{n_k},\bar{\bB}_{n_k},\bar{\bH}_{n_k},\bar{\beta}^{1,n_k},
\bar{\beta}^{2,n_k},\bar{\beta}^{3,n_k},\bar{\beta}^{4,n_k})_{k\in\mathbb{N}}$ converges in the topology of $\mathcal{Z}$ to a random element $(\bar{\bu},\bar{\bw},\bar{\bM},\bar{\bB}, \bar{\bH}, \bar{\beta}^1,\bar{\beta}^2,\bar{\beta}^3,\bar{\beta}^4) \in \mathcal{Z}$ with probability one on $(\bar{\Omega},\bar{\mathcal{F}},\bar{\mathbb{P}})$, as $k\to \infty$.
\item[(3)] $\bar{\beta}^{1,n}(\bar{\omega}) = \bar{\beta}^1(\bar{\omega}), \,
\bar{\beta}^{2,n}(\bar{\omega}) = \bar{\beta}^2(\bar{\omega}), \, \bar{\beta}^{3,n}(\bar{\omega}) = \bar{\beta}^3(\bar{\omega}), \, \bar{\beta}^{4,n}(\bar{\omega}) = \bar{\beta}^4(\bar{\omega})$  for all $\bar{\omega} \in \bar{\Omega}$.       
\end{itemize}
\end{proposition}
To simplify our notation, we will denote these sequences again by 
$(\bu_n,\bw_n,\bM_n,\bB_n, \bH_n,\beta^{1,n},$ $\beta^{2,n}, \beta^{3,n},\beta^{4,n})$ and $(\bar{\bu}_{n},\bar{\bw}_{n},\bar{\bM}_{n},\bar{\bB}_{n},\bar{\bH}_n,\bar{\beta}^{1,n},\bar{\beta}^{2,n},\bar{\beta}^{3,n},\bar{\beta}^{4,n})$, respectively. Moreover, from the definition of the space $\mathcal{Z}$, we deduce that $\bar{\mathbb{P}}$-a.s.
  \begin{align}\label{eq6.66}
  	\bar{\bu}_{n} \to \bar{\bu} ~ &\text{in} ~ \tilde{\mathcal{Z}}_{\bu},  \notag
  	\\
  	\bar{\bw}_{n} \to \bar{\bw} ~ &\text{in} ~ \tilde{\mathcal{Z}}_{\bw},  \notag
  	\\
  	\bar{\bM}_{n} \to \bar{\bM} ~ &\text{in} ~  \tilde{\mathcal{Z}}_{\bM}, 
  	\\
  	\bar{\bB}_{n}\to \bar{\bB} ~ &\text{in} ~ \tilde{\mathcal{Z}}_{\bB},  \notag
  	\\
  	\bar{\bH}_n \to \bar{\bH} ~ &\text{in} ~ \tilde{\mathcal{Z}}_{\bH}. \notag
  \end{align}
\subsection*{Step 4: Passage to the limit} 
Let us recall the following fact, for the sake of what follows. If given a finite positive measure space $(E,\mathcal{M},\mu)$ and $X$ be a Banach space, then the Bochner space $L^r(E;X)$ is reflexive if and only if $L^r(E,\mu)$ and $X$ are reflexive (see, for instance, \cite[Corollary 2, p. 100]{Diestel}). \newline
The following lemma is of independent interest and will be useful later.
\begin{lemma}\label{lem5.8}
Let $X_1, X_2, Y_1, Y_2$ and $Z$ be (reflexive) Banach spaces such that the embeddings $X_1\subset X_2$ and $Y_1\subset Y_2$ are continuous. 
Let $\Sigma: X_1\times Y_1 \to Z^\prime$ be a bilinear map satisfying:  \\ there exist constants  $C>0$ and  $r \in (0,1]$  such that 
	\begin{equation}\label{Eq:Cont-Sigma}
	\lvert \langle \Sigma(x, y), z\rangle \lvert 
	\le C \|x\|_{X_2}^r \|x\|_{X_1}^{1-r} \|y\|_{Y_2}^r \| y\|_{Y_1}^{1-r} \|z\|_Z, \quad \forall x\in X_1, \, y\in Y_1, \, z\in Z. 
	\end{equation} 
Let $([x_\ell, y_\ell])_{\ell \in \mathbb{N}}\subset L^2(0,T; X_1\times Y_1)$ be a bounded sequence such that  as $\ell \to \infty$
     \begin{align}
        (x_\ell, y_\ell)&\to (x,y)\text{ strongly  in }  L^2(0,T; X_2\times Y_2),\label{Eq:Strong-Conv-X2}\\
        (x_\ell, y_\ell) &\rightharpoonup (x,y)\text{ weakly  in }  L^2(0,T; X_1\times Y_1). \label{Eq:Weak-Conv-X1}
     \end{align}
Then, 
\begin{align*}
  \int_0^T \langle \Sigma(x_\ell, y_\ell), z\rangle\,\d s \to \int_0^T \langle \Sigma(x, y), z\rangle \,\d s, \quad \forall z \in Z.  
\end{align*}
\end{lemma}
\begin{proof}
Let $z\in Z$. 	Since $\Sigma$ is bilinear, we have 
	\begin{equation*}
	 \int_0^T \langle \Sigma(x_\ell, y_\ell), z\rangle\, \d s - \int_0^T \langle \Sigma(x, y), z\rangle\, \d s = 	\int_0^T \langle \Sigma(x_\ell-x, y_\ell), z\rangle\, \d s + \int_0^T \langle \Sigma(x, y_\ell-y), z \rangle\, \d s.
	\end{equation*}
Thanks to the inequality \eqref{Eq:Cont-Sigma}, we see that the linear map $L^2(0,T; Y_1) \ni v \mapsto  \int_0^T \langle \Sigma(x, v), z\rangle \d s \in  \mathbb{R}$ is linear continuous (hence weakly continuous) which along with the  assumption \eqref{Eq:Weak-Conv-X1} implies that 
	\begin{align*}
		\int_0^T \langle \Sigma(x, y_\ell-y), z\rangle\, \d s\to 0 \text{ as } \ell \to \infty. 
	\end{align*}
Next, using \eqref{Eq:Cont-Sigma} and H\"older's inequality,  we obtain 
	\begin{align*}
		\biggl \lvert 	\int_0^T \langle \Sigma(x_\ell-x, y_\ell), z \rangle\, \d s \biggr\lvert &\leq  C \|z\|_{Z} \int_0^T \|x_\ell -x\|^{r}_{X_2} \| x_\ell -x\|^{1-r}_{X_1} \| y_\ell\|_{Y_1} \, \d s \\
		&\le C \|z\|_Z  \biggl(\int_0^T \| x_\ell -x\|^{2}_{X_2}\, \d s\biggr)^{\frac r2} \biggl( \int_0^T \| x_\ell -x\|^{2}_{X_1}\, \d s\biggr)^\frac{1-r}{2} \biggl(\int_0^T \| y_\ell \|^2_{Y_1}\, \d s \biggr)^\frac12, 
	\end{align*}
from which along with the boundedness of $(x_\ell-x)_{\ell \in \mathbb{N}}\subset L^2(0,T; X_1)$, $(y_\ell)_{\ell\in \mathbb{N}}\subset L^2(0,T; Y_1)$ and the strong convergence \eqref{Eq:Strong-Conv-X2}, we infer that 
	 \begin{equation*}
	 \int_0^T \langle \Sigma(x_\ell-x, y_\ell), z\rangle\,\d s \to 0 \text{ as } \ell \to \infty. 
	 \end{equation*}
Thus, the proof of the desired convergence in the lemma is now complete. 
\end{proof}
We also prove the following lemma. 
\begin{lemma}\label{cor5.9}
Let $Z$ be a Banach space and $Z^\prime$ its dual. 
Let $(\Sigma_\ell)_{\ell \in \mathbb{N}} \subset L^2(0,T; Z^\ast)$ be a bounded sequence such that 
	\begin{align}\label{Eq:Assum-WeakConv}
		\lim_{\ell \to \infty} \int_0^T\langle \Sigma_\ell(s), z \rangle \,\d s =\int_0^T \langle \Sigma(s), z\rangle \,\d s,
	\end{align}
for some $\Sigma \in L^2(0,T; Z^\prime)$ and for all $z\in Z$. 
If $P_\ell: Z\to Z$ is a linear map such that $P_\ell z\to z \text{ strongly in } Z$ for any $z\in Z$, then we have
  \begin{equation*}
  	\int_0^T \langle \Sigma_\ell (s), P_\ell z\rangle\, \d s \to \int_0^T \langle \Sigma(s), z\rangle\, \d s, \quad \forall z \in Z.  
  \end{equation*}
\end{lemma}
\begin{proof}
Let $(\Sigma_\ell)_{\ell \in \mathbb{N}} \subset L^2(0,T; Z^\ast)$  as in the statement of the lemma, $z\in Z$, $\ell\in \mathbb{N}$. We have 
	\begin{align*}
		\int_0^T \langle \Sigma_\ell(s), P_\ell z\rangle\, \d s - \int_0^T \langle \Sigma(s), z\rangle\, \d s= 	\int_0^T \langle \Sigma_\ell(s)-\Sigma(s), z\rangle\, \d s - \int_0^T \langle \Sigma_\ell(s), P_\ell z - z\rangle\, \d s .
	\end{align*}
By Cauchy-Schwarz's inequality, the boundedness of $(\Sigma_\ell)_{\ell \in \mathbb{N}} \subset L^2(0,T; Z^\ast)$, and the strong convergence $P_\ell z\to z$, we see that the second term  on the right-hand side of the above equality converges to $0$. With this in mind, and the assumption \eqref{Eq:Assum-WeakConv}, we easily complete the proof  of the lemma.
\end{proof}
Since the processes $(\bar{\bu}_{n},\bar{\bw}_{n},\bar{\bM}_{n},\bar{\bH}_n)$ and $(\bu_n,\bw_n,\bM_n,\bH_n)$ have the same law, it follows from Lemma \ref{lem6.4} that there exists a constant $C>0$ such that for every $n\in \mathbb{N}$ an for all $p \in [2,4]$, 
\begin{equation}\label{eq6.69}
	\begin{aligned}
		& \bar{\mathbb{E}} \sup_{t \in [0,T]} |\bar{\bu}_n(t)|^{2p} \leq C, ~ \bar{\mathbb{E}} \sup_{t \in [0,T]} |\bar{\bw}_n(t)|^{2p} \leq C, ~ \bar{\mathbb{E}} \sup_{t \in [0,T]} |\bar{\bM}_n(t)|^{2p} \leq C, 
		\\
		& \bar{\mathbb{E}} \sup_{t \in [0,T]} |\bar{\bH}_n(t)|^{2p} \leq C, ~ \bar{\mathbb{E}} \left(\int_0^{T} |\nabla \bar{\bu}_n(t)|^2 \d t \right)^p \leq C, ~ \bar{\mathbb{E}} \left(\int_0^{T} |\nabla \bar{\bw}_n(t)|^2 \d t \right)^p \leq C,
		\\ 
		& \bar{\mathbb{E}} \left(\int_0^{T} |\bar{\bH}_n(t)|^2 \d t \right)^p \leq C, ~ \bar{\mathbb{E}} \left(\int_0^{T} |\curl\bar{\bH}_n(t)|^2 \d t \right)^p \leq C, 
		\\
		& \bar{\mathbb{E}} \left(\int_0^{T} |\curl\bar{\bM}_n(t)|^2 \d t \right)^p \leq C, ~ \bar{\mathbb{E}} \left(\int_0^{T} |\diver\bar{\bM}_n(t)|^2 \d t \right)^p \leq C.
	\end{aligned}
\end{equation}
Here, $\bar{\mathbb{E}}$ denotes the mathematical expectation with respect to $\bar{\mathbb{P}}$. \newline
From \eqref{eq6.69} and the Banach-Alaoglu theorem, we conclude that there exists a subsequence of $(\bar{\bu}_{n})_{n}$, $(\bar{\bw}_{n})_{n}$, $(\bar{\bM}_{n})_{n}$, and $(\bar{\bH}_n)_{n}$ weakly convergent, and the limit processes fulfill the following regularity properties for all $p \in [2,4]$,
\begin{equation}\label{eqt6.70}
\begin{aligned}	
& \bar{\mathbb{E}} \sup_{t \in [0,T]} |\bar{\bu}(t)|^{2p} \leq C, ~ \bar{\mathbb{E}} \sup_{t \in [0,T]} |\bar{\bw}(t)|^{2p} \leq C, ~ \bar{\mathbb{E}} \sup_{t \in [0,T]} |\bar{\bM}(t)|^{2p} \leq C, 
  \\
& \bar{\mathbb{E}} \sup_{t \in [0,T]} |\bar{\bH}(t)|^{2p} \leq C, ~  
\bar{\mathbb{E}} \left(\int_0^{T} |\nabla \bar{\bu}(t)|^2 \d t \right)^p\leq C, ~ \bar{\mathbb{E}} \left( \int_0^{T} |\nabla \bar{\bw}(t)|^2 \d t \right)^p \leq C, ~
  \\ 
& \bar{\mathbb{E}} \left(\int_0^{T} |\curl\bar{\bH}(t)|^2 \d t \right)^p \leq C, ~ \bar{\mathbb{E}} \left(\int_0^{T} |\curl\bar{\bM}(t)|^2 \d t \right)^p \leq C, 
  \\
& \bar{\mathbb{E}} \left(\int_0^{T} |\diver\bar{\bM}(t)|^2 \d t \right)^p \leq C, ~ \bar{\mathbb{E}} \left(\int_0^{T} |\bar{\bH}(t)|^2 \d t \right)^p \leq C.
\end{aligned}
\end{equation}
In order to shorten the presentation, we will  rewrite the problem \eqref{Eq4.3} and its Galerkin approximation \eqref{Eq:Galerkin-App}  as an abstract stochastic evolution equations.  For this purpose, we set 
	$$\mathbb{V}= V \times \mathbb{H}_0^1 \times V_1 \times V_2 ~ \text{and} \quad \mathcal{U}= \mathcal{U}_1 \times \mathcal{U}_1 \times \mathcal{U}_1 \times \mathcal{U}_2.$$ 
For four Banach spaces $X_i, \,i=1,\ldots, 4$,  we understand that $(\prod_{i=1}^4 X_i )^\prime := \prod_{i=1}^4 X_i^\prime$. 
Then, we  set
\begin{align*}
	\langle Z^\ast, Z\rangle_{\mathbb{V}^\prime, \mathbb{V}}=\begin{pmatrix}
		\langle Z_1^\ast, Z_1\rangle_{V^\prime,V } \\
		\vdots \\
		\langle Z_4^\ast, Z_4 \rangle_{V_2^\prime, V_2}
	\end{pmatrix}, \quad \forall Z\in \mathbb{V}, \, Z^\ast \in \mathbb{V}^\prime. 
\end{align*}
We also define the map $\mathcal{P}_n: \mathbb{V} \to \mathbb{V}$ by 
$$ \mathcal{P}_nZ= (\mathcal{P}_n^1Z_1, \mathcal{P}_n^2Z_2,  \mathcal{P}_n^3Z_3, \mathcal{P}_n^3Z_4,), \quad Z\in \mathbb{V}.$$
Next,   we introduce the following mappings
$\mathbf{A}: \mathbb{V} \to \mathbb{V}'$, $\mathcal{B}: \mathbb{V} \times \mathbb{V} \to \mathbb{V}'$,  $\mathcal{R}: \mathbb{V} \to \mathbb{V}'$, and  $\bG:\mathbb{V} \to L_2(\mathcal{U},\mathbb{H})$   that are respectively defined by
\begin{align*}
      \mathbf{A} \Upsilon 
        &= \begin{pmatrix} 
		   \nu A \bu \\ 
		   \lambda_1 A_1 \bw  - (\lambda_1 + \lambda_2) R_5(\bw)\\
		   \lambda  R_6(\bM) - \lambda R_5(\bM) + \frac{1}{\tau} (\bM - \chi_0 \bH) \\
		      \frac{1}{\sigma} R_6(\bH)
           \end{pmatrix}, ~            
           \\
\mathcal{B}(\Upsilon, \Upsilon) 
& =\begin{pmatrix}
  B_0(\bu,\bu) - \mu_0 M_0(\bM,\bH) + \alpha R_0(\bu, \bw)
   \\
    \\
  B_1(\bu,\bw) - 2 \alpha R_2(\bu,\bw) - \mu_0 R_3(\bM,\bH)
	  \\
	   \\
  B_2(\bu,\bM) - R_3(\bw,\bM) 
		\\
		 \\
 - \tilde{M}_2(\bu,\bB)
	\end{pmatrix}, 
	\\
	       \mathcal{R}(\Upsilon) 
	        &=\begin{pmatrix}
		    - \mu_0 R_1(\bH, \bH)
		     \\
		      \\
		    0
		        \\
		         \\
		     0
		          \\
		           \\
		   0
	          \end{pmatrix},  
\end{align*}
and 
\begin{align*}
	 \bG(\Upsilon) 
	=\mathrm{diag} \left(
		F_1(\bu),\; 
		F_2(\bw),\; 
		G(\bM),\; 
		F_3(\bH)
	\right),
\end{align*}
where $\bH= \frac{1}{\mu_0} \bB - \bM$ and $\Upsilon=(\bu,\bw,\bM,\bB) \in \mathbb{V}$. \newline 
In view of these notations, we can now rewrite the variational formulation of \eqref{eq4.2a}-\eqref{eq4.2d} in the following form: 
    \begin{equation}\label{eq5.33}
    	\langle \d \Upsilon_n, Z\rangle_{\mathbb{V}^\prime, \mathbb{V}} + \langle(\mathbf{A} \Upsilon_n + \mathcal{B} (\Upsilon_n,\Upsilon_n) + \mathcal{R}(\Upsilon_n)) \d t, \mathcal{P}_n Z \rangle_{\mathbb{V}^\prime, \mathbb{V}} = \langle \bG (\Upsilon_n) \d {\beta}^{n}, Z\rangle_{\mathbb{V}^\prime, \mathbb{V}},
    \end{equation}
for all $Z\in \mathbb{V}$, where $\Upsilon_n= ({\bu}_n,{\bw}_n,{\bM}_n,{\bB}_n)$, ${\bB}_n={\bM}_n + \mu_0\bH_n$, and ${\beta}^{n}= ({\beta}^{1},{\beta}^{2},{\beta}^{3},{\beta}^{4})^{Tr}$. \newline
Next, let $\bar{\bH}_n$ and $( \bar{\bu}_n, \bar{\bw}_n, \bar{\bM}_n, \bar{\bB}_n)$, with $ \bar{\bB}_n= \mu_0 (\bar{\bM}_n+  \bar{\bH}_n)$, the processes from Proposition \ref{Prop:Application-Prokhorov-Skorokhod}. We set 
\begin{align*}
	 \widebar{\Upsilon}_n:=& (\widebar{\bu}_n,\widebar{\bw}_n,\widebar{\bM}_n,\widebar{\bB}_n), \\
	 \widebar{\beta}^{n}:=& (\widebar{\beta}^{1,n},\widebar{\beta}^{2,n},\widebar{\beta}^{3,n},\widebar{\beta}^{4,n})^{Tr},
\end{align*}
where the processes appearing on the RHS of the second equation  are from Proposition \ref{Prop:Application-Prokhorov-Skorokhod}. \newline
Since, by Proposition \ref{Prop:Application-Prokhorov-Skorokhod}, the laws of $\Upsilon_{n}$ and $\widebar{\Upsilon}_n$ (resp. $\beta^n$ and $ \bar{\beta}^n$) are equal, we can argue as in \cite{Bensoussan_1995} and show that 
\begin{equation}\label{eq5.33-B}
	\langle \d \widebar{\Upsilon}_n, Z\rangle_{\mathbb{V}^\prime, \mathbb{V}} + \langle( \mathbf{A} \widebar{\Upsilon}_n+ \mathcal{B} (\widebar{\Upsilon}_n, \widebar{\Upsilon}_n) + \mathcal{R}(\widebar{\Upsilon}_n))\d t, \mathcal{P}_n Z \rangle_{\mathbb{V}^\prime, \mathbb{V}} = \langle \bG (\widebar{\Upsilon}_n) \d \bar{\beta}^{n}, \mathcal{P}_n Z\rangle_{\mathbb{V}^\prime, \mathbb{V}},\; Z\in \mathbb{V}. 
\end{equation}
In order to prove Theorem \ref{theo 1}, we will pass to the limit in the integral version of  \eqref{eq5.33-B}. For that purpose, we state the following proposition. 
\begin{proposition}\label{propo5.10}
Let $\bar{\bH}$ and $\widebar{\Upsilon}=(\bar{\bu}, \bar{\bw}, \bar{\bM}, \bar{\bB})$, with $ \bar{\bB}= \mu_0 (\bar{\bM} +  \bar{\bH})$, be the stochastic processes from Proposition \ref{Prop:Application-Prokhorov-Skorokhod}. Then,   for all $\bU= (\bv, \upsilon, \boldsymbol{\psi}_1, \boldsymbol{\phi}_1) \in \mathbb{V}$,
	\begin{align}
	 \int_0^T \langle \mathbf{A} \widebar{\Upsilon}_n(s), \mathcal{P}_n \bU \rangle_{\mathbb{V}', \mathbb{V}} \, \d s \to & \int_0^T \langle \mathbf{A} \widebar{\Upsilon}(s), \bU \rangle_{\mathbb{V}', \mathbb{V}} \, \d s, ~ &\bar{\mathbb{P}}\text{-a.s.,} ~ \text{as} ~ n \to \infty, \label{eq5.35}
	     \\
	 \int_0^T \langle \mathcal{B}(\widebar{\Upsilon}_n(s), \widebar{\Upsilon}_n(s)), \mathcal{P}_n \bU\rangle_{\mathbb{V}', \mathbb{V}} \, \d s \to & \int_0^T \langle \mathcal{B}(\widebar{\Upsilon}(s), \widebar{\Upsilon}(s)), \bU \rangle_{\mathbb{V}', \mathbb{V}} \, \d s, ~ &\bar{\mathbb{P}}\text{-a.s.,} ~ \text{as} ~ n \to \infty,\label{eq5.34} 
	         \\
	\int_0^T \langle \mathcal{R}(\widebar{\Upsilon}_n(s)), \mathcal{P}_n\bU \rangle_{\mathbb{V}', \mathbb{V}} \, \d s \to & \int_0^T \langle \mathcal{R}(\widebar{\Upsilon}(s)), \bU \rangle_{\mathbb{V}', \mathbb{V}} \, \d s, ~ &\bar{\mathbb{P}}\text{-a.s.,} ~ \text{as} ~ n \to \infty.\label{eq5.36}
	\end{align}
%
%
\end{proposition}
\begin{proof}
Since the maps $\mathcal{P}_n^i: X^i  \to X_n^i $ (where $(X^i: i=1,2,3)= (H, \mathbb{H}_0^1, V_1)$) are orthogonal projections onto the finite dimensional spaces $X_n^i$, with  $(X_n^i: i=1,2,3)= (\bar{V}_n,\bar{H}^1_{0n} ,V_{1n})$ and $V_2\subset V_1$, we see that  for any $i\in \{1,2,3,4\}$:
\begin{align}\label{Eq:Conv-proj}
	\mathcal{P}^i_n \mathbf{U}_i \to \mathbf{U}_i \text{ in $\mathbb{V}_i$ as $n\to \infty,$ }  
\end{align}
where we understand that $\mathcal{P}_n^4= \mathcal{P}_n^3$, and $\mathbf{U}_i$,  $\mathbb{V}_i$ denote the $i^{\text{th}}$ component of the vector $\bU$ and the product space $\mathbb{V}$, respectively. \newline
Now,  the convergences stated in  \eqref{eq6.66}, the continuity of the  linear maps $A: V\to V^\prime$, $A_1:\mathbb{H}_0^1 \to [H^{-1}(\mathcal{O})]^3$, $R_5: X\to X^\prime$, $R_6: Z\to Z^\prime$ with $X\in \{\mathbb{H}_0^1, V_1\}$ and $Z\in \{V_1, V_2\}$  imply that  for any $i\in \{1,2,3,4\}$ 
	$$\int_0^T \langle \mathbf{A}_i \widebar{\Upsilon}_n(s), \bU_i \rangle_{\mathbb{V}_i', \mathbb{V}_i} \, \d s \to  \int_0^T \langle \mathbf{A}_i \widebar{\Upsilon}(s), \bU_i \rangle_{\mathbb{V}_i', \mathbb{V}_i} \, \d s, ~ \bar{\mathbb{P}} \text{-a.s.,} ~ \text{as} ~ n \to \infty,$$
from which along with \eqref{Eq:Conv-proj} and Lemma  \ref{cor5.9}, we deduce the  convergence \eqref{eq5.35}. \newline
As for the convergence \eqref{eq5.34}, we first notice that owing to the definition of $R_0$ and Lemma \ref{lem:Bilinear-R0}, we can decompose $R_0$ as the sum of two continuous linear maps $R_0^1: V\to V^\ast $ and $R_0^2: \mathbb{H}_0^1 \to [H^{-1}(\mathcal{O})]^3$. Thus, arguing as above, we prove that 
\begin{align}\label{Eq:Conv-R0}
	\int_0^T \langle R_0(\bar{\bu}_n, \bar{\bw}_n), \mathcal{P}^1_n \mathbf{U}_1\rangle_{V^\prime,V} \,\d s \to 	\int_0^T \langle R_0(\bar{\bu}, \bar{\bw}),  \mathbf{U}_1\rangle_{V^\prime,V} \,\d s, ~ \bar{\mathbb{P}}\text{-a.s.,} ~ \text{as} ~ n \to \infty. 
\end{align} 
Second, thanks to \eqref{eq2.6a} along with Lemmas \ref{lem1} and \ref{lem1B}, the maps $M_0$, $B_k,\; k=0,1,2$, $R_k,\; k=2,3$ and $\tilde{M}_2$ satisfy the assumptions of Lemma \ref{lem5.8}. Thus, applying Lemma \ref{lem5.8} in conjunction with the convergences stated in \eqref{eq6.66} and \eqref{Eq:Conv-proj}, respectively, Lemma \ref{cor5.9}, and \eqref{Eq:Conv-R0}, we obtain that 
for any $i\in \{1,2,3,4\}$ 
$$\int_0^T \langle \mathcal{B}_i (\widebar{\Upsilon}_n(s),\widebar{\Upsilon}_n(s)), \mathcal{P}^i_n \bU_i \rangle_{\mathbb{V}_i', \mathbb{V}_i} \, \d s \to  \int_0^T \langle \mathcal{B}_i (\widebar{\Upsilon}(s),\widebar{\Upsilon}(s)), \bU_i \rangle_{\mathbb{V}_i', \mathbb{V}_i} \, \d s, ~ \bar{\mathbb{P}}\text{-a.s.,} ~ \text{as} ~ n \to \infty,$$
from which we easily conclude the proof of \eqref{eq5.34}. \newline
In order to prove \eqref{eq5.36}, we note that by Lemma \ref{Lem:Bilinear-R1}, we have 
\begin{equation*}
	\begin{aligned}
		&\langle R_1(\bar{\bH}_n, \bar{\bH}_n),\mathbf{U}_1\rangle_{V',V} - \langle R_1(\bar{\bH}, \bar{\bH}),\mathbf{U}_1 \rangle_{V',V} \\
		&= \int_\mathcal{O} [\curl(\bar{\bH}_n-\bar{\bH}) \times (\bar{\bH}_n-\bar{\bH})] \cdot \mathbf{U}_1\, \d x 
		+ \int_{\mathcal{O}} [\curl(\bar{\bH}_n-\bar{\bH}) \times \bar{\bH}] \cdot  \mathbf{U}_1\, \d x \\
		&\qquad + \int_{\mathcal{O}} [\curl\bar{\bH} \times (\bar{\bH}_n-\bar{\bH})] \cdot \mathbf{U}_1\, \d x,
	\end{aligned}
\end{equation*}
from which along with the fifth convergence in  \eqref{eq6.66}  implies 
	$$\int_0^T \langle R_1 (\widebar{\Upsilon}_n(s),\widebar{\Upsilon}_n(s)), \bU_1 \rangle_{\mathbb{V}_1', \mathbb{V}_1} \, \d s \to  \int_0^T \langle R_1 (\widebar{\Upsilon}(s),\widebar{\Upsilon}(s)), \bU_1 \rangle_{\mathbb{V}_1', \mathbb{V}_1} \, \d s, ~ \bar{\mathbb{P}}\text{-a.s.,} ~ \text{as} ~ n \to \infty.$$
With this at hand and Lemma \ref{cor5.9}, we can readily complete the proof of \eqref{eq5.36}. 
\end{proof}
\begin{proposition}\label{Prop:Proposition-12}
Let $\bar{\bH}$ and $\widebar{\Upsilon}=( \bar{\bu}, \bar{\bw}, \bar{\bM}, \bar{\bB})$, with $ \bar{\bB}= \mu_0 (\bar{\bM} +  \bar{\bH})$, be the stochastic processes as given by Proposition \ref{Prop:Application-Prokhorov-Skorokhod}. Then, there exists a subsequence of  $\widebar{\Upsilon}_n$ and $\bar{\beta}^n$, still denoted with the same symbols, such that   as $n\to \infty$
	\begin{equation}
	  \int_0^t \bG^n(\widebar{\Upsilon}_n(s)) \, \d \widebar{\beta}^{n}(s) \to \int_0^t \bG(\widebar{\Upsilon}(s)) \, \d \bar{\beta}(s) ~ \text{  in probability in} ~ L^2(0,T;\mathbb{V}^\prime),
	\end{equation}
for all $t \in[0,T]$. 
\end{proposition}
\begin{proof}
Let \begin{align*}
	\bG^n(\widebar{\Upsilon}) 
	=\mathrm{diag} \left(
	\mathcal{P}_n^1 F_1(\bar{\bu}),\; 
		\mathcal{P}_n^2 F_2(\bar{\bw}),\; 
		\mathcal{P}_n^3 G(\bar{\bM}),\; 
	\mathcal{P}_n^3 F_3(\bar{\bH})
	\right), \quad \widebar{\Upsilon} \in \mathbb{V}. 
\end{align*}
Note also that since the components of $\bG^n$  are linear maps, $\bG^n$ is  also a linear, hence Lipschitz continuous,  which along with the second part of Proposition \ref{propo-1} and  \eqref{eq6.66} we infer that $\bG^n(\widebar{\Upsilon}_n) \to \bG(\widebar{\Upsilon})$ in $L^2(0,T;L_2(\mathcal{U},\mathbb{V}^\prime))$ $\bar{\mathbb{P}}$-a.s., as $n \to \infty$. Since each orthogonal projection $\mathcal{P}_n^i: \mathbb{V}_i\to \mathbb{V}_i$ for $i\in {1,2,3}$ is non-expansive as an operator, we obtain from  \eqref{Eq3.10} that there exists a constant $C>0$ such that for all $n\in \mathbb{N}$ 
  \begin{equation*}
	\|\bG^n(\widebar{\Upsilon}_n)\|_{L_2(\mathcal{U},\mathbb{V}^\prime)}^2 
	\leq C \left(|\bar{\bu}_n|^2 + |\bar{\bw}_n|^2 +  |\bar{\bM}_n|^2 +  |\bar{\bH_n}|^2\right),
\end{equation*}
which along with \eqref{eq6.69} imply that the sequence $\bG^n(\widebar{\Upsilon}_n)$ is uniformly bounded in $L^2(\bar{\Omega}; L^2(0,T;L_2(\mathcal{U},\mathbb{V}^\prime)))$. Consequently, by the Vitali-Convergence Theorem,  we have 
    \begin{equation}\label{eqt5.39-A}
    	\bG^n(\widebar{\Upsilon}_n) \to \bG(\widebar{\Upsilon}) ~ \text{in} ~ L^2(\bar{\Omega};L^2(0,T;L_2(\mathcal{U},\mathbb{V}^\prime))) ~ \text{as} ~ n \to \infty.
    \end{equation}
With this in mind, we can now apply \cite[Lemma 2.1]{Debussche+Glatt+Temam_2011} and obtain
\begin{equation}
	\int_0^t \bG^n(\widebar{\Upsilon}_n(s)) \, \d \widebar{\beta}^{n}(s) \to \int_0^t \bG(\widebar{\Upsilon}(s)) \, \d \bar{\beta}(s) ~ \text{  in probability in} ~ L^2(0,T;\mathbb{V}^\prime),
\end{equation}
which completes the proof of the proposition.
\end{proof}
We are now ready to complete the proof of Theorem \ref{theo 1}. 
\begin{proof}[Proof of Theorem \ref{theo 1}]
Recall that  the integral version of \eqref{eq5.33-B} is 
  \begin{align*}
  	& \langle \widebar{\Upsilon}_{n}(t) - \widebar{\Upsilon}_{n}(0), \bU \rangle_{\mathbb{V}^\prime, \mathbb{V}} + \int_0^t \langle(\mathbf{A} \widebar{\Upsilon}_n(s) + \mathcal{B} (\widebar{\Upsilon}_n(s), \widebar{\Upsilon}_n(s)) + \mathcal{R}(\widebar{\Upsilon}_n(s))), \mathcal{P}_n \bU \rangle_{\mathbb{V}^\prime, \mathbb{V}} \,\d s \\
  	& \quad - \left \langle \int_0^t \bG^n(\widebar{\Upsilon}_n(s)) \d \bar{\beta}^{n}(s),  \bU \right \rangle_{\mathbb{V}^\prime, \mathbb{V}}
  		=0, \quad \bar{\mathbb{P}}\text{-a.s.,} \quad \forall t \in [0,T] \quad \text{and} \quad \forall \bU \in \mathbb{V}.
  \end{align*}
By using Propositions \ref{propo5.10} and \ref{Prop:Proposition-12}  and passing to the limit in the above equation, we deduce that $\widebar{\Upsilon}$ satisfies for all $\bU \in \mathbb{V}$ and $t \in [0,T]$,
  \begin{equation}
   \begin{aligned}
  	 & \langle \widebar{\Upsilon}(t) - \widebar{\Upsilon}(0), \bU \rangle_{\mathbb{V}^\prime, \mathbb{V}} + \int_0^t \langle(\mathbf{A} \widebar{\Upsilon}(s) + \mathcal{B} (\widebar{\Upsilon}(s), \widebar{\Upsilon}(s)) + \mathcal{R}(\widebar{\Upsilon}(s))), \bU \rangle_{\mathbb{V}^\prime, \mathbb{V}} \,\d s \\
  	 & \quad - \left \langle\int_0^t \bG(\widebar{\Upsilon}(s)) \d \bar{\beta}(s), \bU \right \rangle_{\mathbb{V}^\prime, \mathbb{V}}
   	 =0, \quad \bar{\mathbb{P}}\text{-a.s.},
  \end{aligned}
  \end{equation}
where $\widebar{Y}(0)=(\bu_0, \bw_0, \bM_0, \bB_0)$. This complete the proof of Theorem \ref{theo 1}. 
\end{proof}

\appendix 

\section{Proofs of the auxiliary results in Section \ref{sect2}}\label{Appendix-1}

In this appendix we give the proofs of the auxiliary results, namely Lemmata \ref{lem1}-\ref{lem:Bilinear-R0} and Proposition \ref{propo-1}, stated in Section \ref{sect2}. 

\begin{proof}[\textbf{Proof of Lemma \ref{lem1}}]
We first recall that 
\begin{equation*}
	M_1(\bM,\bH,\bv):= \sum_{i,j=1}^3 \int_{\mathcal{O}} \bM_i(x) (\partial_i \bH_j(x)) \bv_j(x) \, \d x, \quad \forall \bM,\,\bH \in V_1,\; \bv \in V.
\end{equation*}
For the proof of \eqref{eqt2.2}, we refer the reader to \cite[Lemma 1]{Kamel}. \newline
Let $\bv \in V$ and $\bM, \bH \in V_1$ be arbitrary. 
By integration by parts along with the fact that $\diver(\bM + \bH)= 0$ (cf. \eqref{Eq4.5a}), $\bv|_{\partial\mathcal{O}}=0$ and $\diver\bv=0$, we find
	\begin{align*}
		M_1(\bM,\bH,\bv)  
		= \int_{\mathcal{O}} \bM_i (\partial_i \bH_j) \bv_j \, \d x
		&= \int_{\mathcal{O}} [(\bM_i + \bH_i) \partial_i \bH_j - \bH_i (\partial_i \bH_j - \partial_j \bH_i) - \bH_i (\partial_j \bH_i)] \bv_j \,\d x
		\\
		& = - \int_{\mathcal{O}} [(\bM + \bH) \cdot \nabla] \bv \cdot \bH \, \d x - \int_{\mathcal{O}} \curl\bH \cdot (\bH \times \bv) \, \d x.
	\end{align*}
Here we have also used the summation convention on repeated indices. This proves \eqref{eq2.2}.
\newline
Next, from the definition of $M_1$ and H\"older's inequality, we obtain
	\begin{equation*}
		|M_1(\bM,\bH,\bv)| 
		\leq \|\bM\|_{\mathbb{L}^4} |\nabla \bH| \|\bv\|_{\mathbb{L}^4} 
		\leq C(\mathcal{O}) |\bM|^\frac{1}{4} |\nabla \bH| \|\bv\|_{\mathbb{L}^4}  \|\bM\|_{V_1}^\frac{3}{4},
	\end{equation*} 
from which we get \eqref{eq2.3}.
By \eqref{eq2.2} together with H\"older's inequality, we easily derive \eqref{eq2.4}. \newline
For the case $\bv \in V \cap \mathbb{H}^2$, we use \eqref{eq2.2}, H\"older's inequality together with the embedding of $\mathbb{H}^2$ in $\mathbb{L}^\infty$ and the embedding of $\mathbb{H}^1$ in $\mathbb{L}^4$, respectively. Thus,
\begin{align*}
|M_1(\bM,\bH,\bv)| 
& \leq |\bH +\bM| \|\nabla \bv\|_{\mathbb{L}^4} \|\bH\|_{\mathbb{L}^4} + |\curl\bH| |\bH| \|\bv\|_{\mathbb{L}^\infty} \\
&\leq C(\mathcal{O}) |\bH +\bM| \|\bv\|_{\mathbb{H}^2} \|\bH\|_{\mathbb{L}^4} + C(\mathcal{O}) |\curl\bH| |\bH| \|\bv\|_{\mathbb{H}^2}.
\end{align*}
This proves \eqref{eq2.5}. The proof of Lemma \ref{lem1} is now complete.
\end{proof}
We now proceed with the proofs of Lemmata \ref{lem1B} and \ref{lem:Bilinear-R0}. 
\begin{proof}[\textbf{Proof of Lemma \ref{lem1B}}]
Recall that 
$$ M_2(\bu,\bB, \psi)
:= \int_{\mathcal{O}} [\curl (\bu \times \bB)] \cdot \psi \, \d x,\; \forall \bB \in V_2,\, \bu \in V \text{ and } \psi \in \mathbb{H}^1.$$ 
We note that for  $\bB \in V_2$ and $\bu \in V$, we have the following equality, which is to be understood in the weak sense $$\curl (\bu \times \bB)= (\diver\bB + \bB \cdot \nabla) \bu - (\diver\bu + \bu \cdot \nabla) \bB.$$ Now, since $\diver\bB=0$ and $\diver\bu=0$ in $\mathcal{O}$, we can rewrite the previous equality as follows: $\curl(\bu \times \bB)= (\bB \cdot \nabla)\bu - (\bu \cdot \nabla) \bB$. Consequently, 
	\begin{align*}
		M_2(\bu,\bB, \psi)
		& = \int_{\mathcal{O}} (\bB(x) \cdot \nabla)\bu(x) \cdot \psi(x) \d x - \int_{\mathcal{O}} (\bu(x) \cdot \nabla) \bB(x) \cdot \psi(x) \, \d x 
                  \\
		& = - b(\bB,\psi,\bu) - b(\bu,\psi,\bB), \quad \forall \bB \in V_2, \quad \bu \in V \quad \text{and} \quad \psi \in \mathbb{H}^1.
	\end{align*}
From this previous equality, we easily derive that for any $\bB \in V_2, ~ \bu \in V, \psi \in \mathbb{H}^1$,
	\begin{equation}\label{eq2.11}
		|M_2(\bu,\bB, \psi)| 
		\leq 2 \|\bB\|_{\mathbb{L}^4} |\nabla \psi| \|\bu\|_{\mathbb{L}^4} 
		\leq C(\mathcal{O}) \|\bB\|_{\mathbb{H}^1}  \|\bu\|_{\mathbb{L}^4} \|\psi\|_{\mathbb{H}^1}. 
	\end{equation}
Thus, for an arbitrary but fixed $\psi \in V_2 \subset \mathbb{H}^1$, the mapping $M_2(\cdot,\cdot,\psi)$ defined on $V \times V_2$ with values in $\mathbb{R}$ is bilinear and continuous, which implies the existence of the continuous bilinear map $\tilde{M}_2: V \times V_2 \to V'_2$ satisfying the properties \eqref{eq2.9B} and \eqref{eq2.9}. This completes the proof of the lemma. 
\end{proof}
\begin{proof}[\textbf{Proof of Lemma \ref{lem:Bilinear-R0}}]
Note that $\curl(\curl \bu)= \nabla \diver \bu - \Delta \bu, \, \bu \in \mathcal{V}$ in both two and three dimensions. Thus, by density of $\mathcal{V}$ in $V$ the trilinear map $r_0: V\times H(\curl,\mathcal{O})\times V\to \mathbb{R} $ defined by 
	\begin{align*}
		r_0(\bu, \bw, \tilde{v})=& \int_{\mathcal{O}} [\curl(\curl \bu - 2 \bw)] \cdot \tilde{\bv} \, \d x\\
		=&  \int_{\mathcal{O}} \nabla \bu(x) : \nabla \tilde{\bv}(x) \, \d x -2 \int_{\mathcal{O}} \curl \bw(x) \cdot \tilde{\bv}(x) \, \d x 
		\\
		=&\langle \curl(\curl \bu - 2 \bw), \tilde{\bv}\rangle_{V^\prime, V}, \quad \bu, \,\tilde{\bv} \in V, \, \bw \in \mathbb{H}^1
	\end{align*}
is well-defined. In addition, by using H\"older's inequality, we see that 
	\begin{align*}
		|r_0(\bu, \bw, \tilde{\bv})| 
		& \leq |\nabla \bu| |\nabla \tilde{\bv}| + 2 |\curl\bw| |\tilde{\bv}|
		\\
		& \leq (|\nabla \bu| + 2 |\curl\bw|) \|\tilde{\bv}\|_V, ~ \bu, \tilde{\bv} \in V,~\bw \in H(\curl,\mathcal{O}). 
	\end{align*}
\dela{\adda{NB: the norm on V is the gradient norm.}}
This implies the existence of a continuous bilinear map $R_0(\cdot,\cdot): V \times H(\curl,\mathcal{O}) \to V'$ satisfying the assertions in the lemma. Consequently, the proof of the lemma is now complete.  
\end{proof}
Finally, we give the proof of Proposition \ref{propo-1}.
\begin{proof}[\textbf{Proof of Proposition \ref{propo-1}}]
We only prove the proposition for the map $F_1$.  Let $\zeta=(\zeta_k) \in \mathbb{R}^3$ and $\bu \in V$ be arbitrary but  fixed. Notice that the first inequality in \eqref{eq3.4} is equivalent to 
    \begin{equation*}
    	\sum_{i,j=1}^{3} \sum_{k=1}^{\infty} b_k^{(i)}(x) b_k^{(j)}(x) \zeta_i \zeta_j \leq 2 \sum_{i,j=1}^{3} \delta_{ij} \zeta_i \zeta_j -c_1|\zeta|^2 = (2-c_1) |\zeta|^2,
    \end{equation*}
which in turn implies \eqref{eq3.9}. Since the map $F_1$ is linear and bounded due to \eqref{eq3.9}, it is also Lipschitz continuous. 
Next, to prove \eqref{Eq3.10}$_1$, we can apply an argument devised by Brze\'zniak and Motyl in \cite{Motyl1}. For this purpose, we  introduce the bilinear form $\hat{b}(\cdot,\cdot) : \mathcal{V} \times \mathcal{V} \subset H \times V \to \mathbb{R}$ defined by
\begin{equation*}
	\hat{b}(\bu,\bv):= \sum_{i,j=1}^{3} \int_{\mathcal{O}} b_k^{(i)}(x) \partial_i \bu_j(x)  \bv_j(x) \, \d x = \int_{\mathcal{O}} (b_k(x) \cdot \nabla) \bu(x) \cdot \bv(x) \, \d x, \quad \forall \bu, \bv \in \mathcal{V}.
\end{equation*}
By integration by parts and H\"older's inequality, we have for all $\bu,\, \bv \in \mathcal{V}$
\begin{align*}
\hat{b}(\bu,\bv)
& = - \int_{\mathcal{O}} [\diver b_k (x) \bu(x) \cdot \bv(x) + (b_k(x) \cdot \nabla) \bv(x) \cdot \bu(x)] \, \d x 
  \\
        \\
& \leq \|\diver b_k\|_{L^\infty(\mathcal{O})} |\bu| \|\bv\|_{V} + \|b_k\|_{\mathbb{L}^\infty} \|\bv\|_{V} |\bu|.
\end{align*} 
Hence, we can define a linear map $\hat{B}$ by $\hat{B} (\bu):= \hat{b}(\bu,\cdot)$ from $H$ to $V'$, such that the following inequality holds
\begin{equation*}
	\|\hat{B}(\bu)\|_{V'} \leq (\|\diver b_k\|_{L^\infty(\mathcal{O})} +\|b_k\|_{\mathbb{L}^\infty}) |\bu|, \quad \bu \in H.
\end{equation*}
This previous inequality can be rewritten as follows
\begin{equation}\label{eq3.18}
	\|(b_k\cdot \nabla)\bu\|_{V'} \leq  (\|\diver b_k\|_{L^\infty(\mathcal{O})} +\|b_k\|_{\mathbb{L}^\infty)} |\bu|, \quad \bu \in H.
\end{equation}
Note that from  \eqref{Eqt3.6}$_1$, we have $F_1(\bu) e^1_i = (b_i \cdot \nabla) \bu$. Thus, from \eqref{eq3.18}, we obtain
\begin{equation}\label{eq3.20}
	\begin{aligned}
		\|F_1(\bu)e^1_i\|_{V'}^2 
		=\|(b_i \cdot \nabla) \bu\|_{V'}^2 
		&\leq \|(b_i \cdot \nabla) \bu\|_{V'}^2 \\  
		&\leq 2 (\|\diver b_i\|_{L^\infty(\mathcal{O})}^2 +\|b_i\|_{\mathbb{L}^\infty}^2) |\bu|^2, \quad \forall \bu \in H.
	\end{aligned}
\end{equation}
In light of \eqref{eq3.20}, we further obtain
\begin{equation}\label{eq3.21}
	\begin{aligned}
		\|F_1(\bu)\|_{L_2(U_1,V')}^2 
		= \sum_{k=1}^{\infty}\|F_1(\bu)e_k^1\|_{V'}^2 
		&\leq 2 \sum_{k=1}^{\infty} ( \|\diver b_k\|_{L^\infty(\mathcal{O})}^2 + \|b_k\|_{\mathbb{L}^\infty}^2) |\bu|^2 \\
		&\leq 2 C_1 |\bu|^2, \quad \forall \bu \in H.
	\end{aligned}
\end{equation}
Hence, $F_1(\bu) \in L_2(U_1,V')$ and \eqref{Eq3.10}$_1$ is proved. It is now easily seen that the  map $F_1$ satisfies all the assertions stated in the proposition. 
In a similar way, we can prove that the other maps $F_2, \, F_3$, and $G$ also satisfy all the proposition's assertions.
\end{proof}

\end{document}